\documentclass[11pt, twoside]{book}

\newcommand{\myBiblio}{main.bib}
\newcommand{\myPackages}{MyPackages}
\newcommand{\myLayout}{MyLayoutBook}
\newcommand{\myMacros}{MyMacros}
\newcommand{\myDiagrams}{MyDiagrams}

\usepackage{\myPackages}
\usepackage{\myLayout}
\usepackage{\myMacros}
\usepackage{\myDiagrams}

\newenvironment{dedication}
  {\clearpage           %
   \thispagestyle{empty}%
   \vspace*{\stretch{1}}%
   \itshape             %
   \raggedleft          %
  }
  {\par %
   \vspace{\stretch{3}} %
   \clearpage           %
  }

\begin{document}

\frontmatter		%
\newcommand{\HRule}{\rule{\linewidth}{0.5mm}}

\begin{titlepage}
\begin{center}
\includegraphics[width=0.4\textwidth]{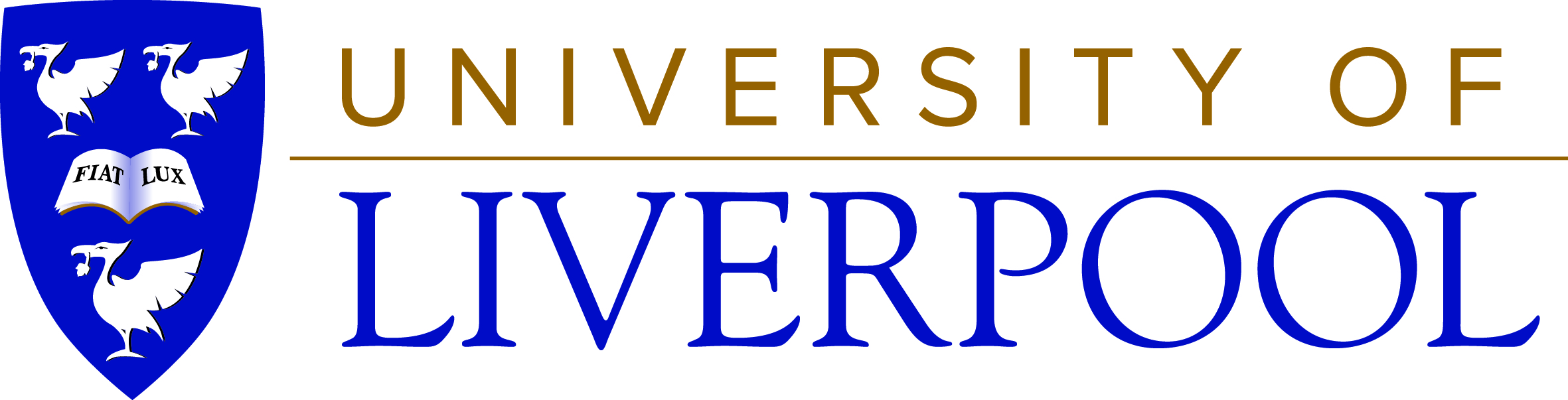}~\\[0.8cm]

\vspace*{1.1cm}
\HRule \\[0.4cm]
\textsc{ \LARGE A convenient category of locally stratified spaces \\[0.4cm] }
\HRule \\[1.5cm]

\begin{minipage}[t]{0.4\textwidth}
\begin{flushleft} \large
\emph{Candidate:}\\
Stefano \textsc{Nicotra}
\end{flushleft}
\end{minipage}%
\begin{minipage}[t]{0.5\textwidth}
\begin{flushright} \large
\emph{Supervisor:} \\
Dr.~Jon \textsc{Woolf}
~\\
\end{flushright}
\end{minipage}
~\\[2.4cm]

\vfill
~\\[0.4cm]
\HRule \\
{Thesis submitted in accordance with the requirements of the University of Liverpool for the  degree  of  Doctor  in  Philosophy  \par}
\vspace{1cm}
{\large February, 2020}
\end{center}
\end{titlepage}

\clearpage
\chapter{Abstract}
In this thesis we define the notion of a locally stratified space.
Locally stratified spaces are particular kinds of streams and d-spaces which are locally modelled on stratified spaces. 
We construct a locally presentable and cartesian closed category of locally stratified spaces that admits an adjunction with the category of simplicial sets.
Moreover, we show that the full subcategory spanned by locally stratified spaces whose associated simplicial set is an $\infty$-category has the structure of a category with fibrant objects.

We define the fundamental category of a locally stratified space and show that the canonical functor $\theta_{A}$ from the fundamental category of a simplicial set $A$ to the fundamental category of its realisation is essentially surjective.
We show that the functor $\theta_{A}$ sends split monomorphisms to isomorphisms, in particular we show that $\theta_{A}$ is not necessarily an equivalence of categories.
On the other hand, we show that the fundamental category of the realisation of the simplicial circle is equivalent to the monoid of the natural numbers.
To conclude, we define left covers of locally stratified spaces and we show that, under suitable assumptions, the category of representations of the fundamental category of a simplicial set is equivalent to the category of left covers over its realisation.

\cleardoublepage
\begin{dedication}
Ai miei genitori.
\end{dedication}
\cleardoublepage
\chapter{Acknowledgements}

Without my family I would not even have an education, let alone having completed a PhD thesis. They have always believed in me and supported me in all my decisions, I love them.

I am extremely grateful to my supervisor Jon Woolf, this thesis begins and ends with him (quite literally, see front-page and \cite{woo08}).
I will never forget the interminable afternoons spent together trying to figure out the details of a proof (mostly me) and then drawing a picture that makes everything self evident (mostly him).

I  thank my second supervisor and PGR Director Nicola Pagani and Alice Rizzardo for organising interesting seminars and for their kindness in replying to all my questions, inside and outside of mathematics.

I am grateful to Yossi for carefully reading various drafts of this thesis and correcting many grammar mistakes in my exposition (the reader can easily see he has not read the current section).

I am thankful to Stephen Nand-Lal for all the time spent answering my questions, discussing stratified spaces and drawing lifting problems at the blackboard. All of this while he was writing his own thesis.

I will never forget the days spent with Alessio doing some work, drinking lots of coffee and playing the guitar together.

I thank my good friend Andrea for the countless games of pool we played together.

A special mention goes to my spiritual gurus Dean and Benji, my Talking Cricket Leti and her partner in life Lucas.
They know better than me how important has been their influence on my growth in the past four years.

I would like to mention my former travelling companions and good friends. 
Anwar, for his kindness in answering all my questions and his taste in mathematics and Oliver, for sharing with me most of the journey and finding true happiness at the end of it.

I am grateful to my lovely housemates for turning a house into a proper home. In particular I am thankful to Flavione for the existential debates, to Caitlin for feeding me in the last few weeks before the deadline and to Rick for his Lego system.

I would like to thank all the people that shared an office with me here in Liverpool, unless they were being too loud all the time (except Felix: I thank him anyway because he is awesome).

I thank my good friends and old pals Edoardo and Andrea. The amount of times they answered one of my stupid doubts (usually with an insult) is uncountable.

I am grateful to Flavio and Stefania for turning the end of a journey into the beginning of another trip.

The list goes on and I will definitely forget someone, therefore I am thankful to whomever is reading these acknowledgements and feels he should have been included.

\cleardoublepage
{\scshape\tableofcontents}
\clearpage

\chapter{Introduction\label{ch:intro}}

The idea that a topological space $X$ should be determined, up to homotopy, by a combinatorial gadget containing all the homotopy groups of $X$, envisaged by Dan Kan, 
was formalised for the first time by Dan Quillen.
In his seminal paper \cite{qui67} Quillen proved that the categories $\Topa$ of topological spaces and $\sSet$ of simplicial sets both admit a \emph{model category structure}.
Moreover, he showed that there exists an adjunction
\begin{equation*}
\Adjoint{\sSet}{\Topa}{\abs{\blank}}{\Sing}
\end{equation*}
which defines an equivalence between their underlying homotopy theories.
The singular simplicial set $\Sing X$ of a topological space $X$ is an $\infty$-\emph{groupoid}, whose `$1$-categorical shadow' recovers the classical \emph{fundamental groupoid} $\Pi_{1} X$ of $X$. 

\subsection*{}
Following the results of an unpublished work of Bob MacPherson, later generalised by David Treumann (\cite{tre09}), one could argue that the right generalisation of the fundamental groupoid to the framework of stratified spaces is the so called \emph{exit path category}.
Morphisms in the exit path category $\Exit(X)$ of a stratified space $X$ are homotopy classes of paths that `wind outward from a deeper stratum'.
As in the classical \emph{monodromy equivalence}, representations of the exit path category of $X$ completely determine the category of stratified covering spaces of $X$, in what \cite{BGH18} call an \emph{exodromy equivalence}:
\[\Set^{\Exit X}\cong \sCover{X}.\]
\subsection*{}
As the first invariant associated to a stratified space is a category rather than a groupoid, it seems natural that, in the context of stratified spaces, Quillen's results should be related to the homotopy theory of $\infty$-categories, rather than $\infty$-groupoids.
Pursuing this point of view, Stephen Nand-Lal  in \cite{nan18} has constructed a model structure on a full subcategory of the category of stratified spaces, transferred from the Joyal model structure on the category of simplicial sets via an adjunction
\begin{equation*}
\Adjoint{\sSet}{\Strata}{\sabs{\blank}}{\sSing}.
\end{equation*}
Recently, work in this direction has been done by Sylvain Douteau and Peter Haine .
Douteau, realising ideas of Henriques, has shown in \cite{dou18} that the category $\overcat{\Topa}{P}$ of topological spaces with a stratification map to $P$ has a combinatorial simplicial model structure.
Barwick, Glasman and Haine define the $\infty$-category $\Strata_{P}$ of abstract $P$-stratified homotopy types as a full subcategory of the slice category $\overcat{\Cat_{\infty}}{P}$ spanned by the $\infty$-categories with a conservative functor to $P$.
Peter Haine in \cite{hai18} has recently unified the two approaches, showing that $\Strata_{P}$ is equivalent to an accessible localisation of the underlying $\infty$-category of the Douteau-Henriques model structure on $\Topa_{P}$

\subsection*{}
However, independently of the approach chosen, the homotopy theory of stratified spaces is not equivalent to the homotopy theory of all $\infty$-categories.
This can be seen already at a $1$-categorical level, as a stratified space has a global preorder on the set of its points that forces `exit loops' to remain in a single stratum.
A natural question to pose is whether we can refine the topological side further, taking into consideration topological spaces with a local preorder on their points that admit more interesting non-invertible paths.

\subsection*{}
Luckily the way has been paved in the area of directed topology, a field that studies the behaviour of topological spaces with a natural notion of `directed paths'.
Different models of directed topological spaces and locally preordered spaces are present in the literature, most notably d-spaces, introduced by Marco Grandis in \cite{gra03}, and streams defined by Sanjeevi Krishnan in \cite{kri09}.

\subsection*{}
In this thesis we address the aforementioned question and we define the notion of a \emph{locally stratified space}.
Locally stratified spaces are particular kinds of streams and d-spaces which are locally modelled on stratified spaces. 
In particular, the global preorder on the points of a locally stratified space is a stratified space and every stratified space naturally gives rise to a locally stratified space.
While streams and d-spaces are of interest in computer science and more precisely in the study of concurrent processes (see  for example \cite{hau12} and \cite{FRG06}) locally stratified spaces have a more topological flavour.
For instance the geometric realisation of any simplicial set has a natural locally-stratified structure.

\subsection*{}
We construct a locally presentable and cartesian closed category of locally stratified spaces that admits an adjunction:
\begin{equation*}
\Adjoint{\sSet}{\LocStrata}{\dabs{\blank}}{\dSing}
\end{equation*}
with the category of simplicial sets.
Moreover, we show that the full subcategory of $\LocStrata$ spanned by the locally stratified spaces $X$ such that $\dSing X$ is an $\infty$-category, has the structure of a category with fibrant objects.

\subsection*{}
However the adjunction between locally stratified spaces and simplicial sets does not define an equivalence between the homotopy theory of $\infty$-categories and the homotopy theory of locally stratified spaces.
As in the case of stratified spaces this can be detected at a 1-categorical level, and we do so by considering the \emph{fundamental category} $\dPi_{1}X$ of a locally stratified space $X$.
For a simplicial set $A$ there is a canonical functor 
\[\theta_{A}\colon\tau_{1}A\to \dPi_{1}\dabs{A}\]
from the fundamental category of $A$ to the fundamental category of its realisation.
We show that the functor $\theta_{A}$ sends split monomorphisms to isomorphisms and we conjecture that $\dPi_{1}\dabs{A}$ is the localisation of $\tau_{1}A$ at the class of split monomorphisms.
We give a few examples that verify the conjecture.
If $S^{1}$ is the simplicial circle, we show that 
\[\theta_{S^{1}}\colon \tau_{1}S^{1}\to \dPi_{1}\dabs{S^{1}}\]
defines an equivalence of categories from the monoid of the natural numbers to the fundamental category of $\dabs{S^{1}}$.
If $R$ is the walking retraction simplicial set, we show that the fundamental category of the realisation of $R$ is equivalent to the terminal category.
In particular, the canonical functor:
\[\theta_{R}\colon \tau_{1}R\to \dPi_{1}\dabs{R}\]
is not an equivalence of categories, and this provides an explicit example of a simplicial set $R$ for which the unit of the adjunction is not a weak equivalence.

\subsection*{}
A left cover over a locally stratified space $X$ is the geometric counterpart of a representation of $\dPi_{1}X$. 
Although we define left covers using lifting properties, we characterise left covers over the realisation $\dabs{A}$ of a simplicial set in terms of \'etale maps and topological covers.
Moreover, we define an adjunction
\begin{equation*}
\Adjoint{\Set^{\tau_{1}A}}{\LCover{\dabs{A}}}{\rec}{\fib}
\end{equation*}
between the category of representations of the fundamental category of $A$ and the category of left covers of its realisation.
Under suitable assumptions, we show that this adjunction defines an equivalence of categories.
As expected, the walking retraction simplicial set provides a counterexample where the above adjunction fails to be an equivalence, as we show that the category of left covers over $\dabs{R}$ is equivalent to the category of sets.

\subsection*{}
Summarising, we construct a category of generalised stratified spaces which captures more information coming from the homotopy theory of $\infty$-categories.
However it is still an open question whether a suitable category of topological spaces equipped with some order-theoretic data, encapsulating the homotopy theory of stratified spaces, can recover the homotopy theory of all $\infty$-categories. A positive answer to this question would allow for example to construct classifying `spaces' for $\infty$-categories capturing fully their homotopical information.

\section*{Outline}
\addcontentsline{toc}{section}{Outline}
\subsection*{}
Sections \ref{ch:1sec:1}, \ref{ch:1sec:2} and \ref{ch:1sec:3} are standard overviews of basic category theory, enriched category theory and factorisation systems.
In Section \ref{ch:1sec:4} we recall the theory of topological functors and topological constructs.

\subsection*{}
In Sections \ref{ch:2sec:1} and \ref{ch:2sec:2}  we give a brief tour of the basics of classical abstract homotopy theory.
Section  \ref{ch:2sec:3} is a review of Cisinski's theory of model structures on presheaves.
The results in Section \ref{ch:2sec:4} are original, where we set up a general framework for the study of the homotopy theory of topological constructs.

\subsection*{}
In Section \ref{ch:3sec:1} we recall the constructions of the Kan-Quillen and Joyal model structures, using the results of Section \ref{ch:2sec:3}.
Section \ref{ch:3sec:2} is an overview of the theory of covering spaces from a simplicial perspective.
In Section \ref{ch:3sec:3} we give an original presentation of the equivalence between the category of left covers over a simplicial set and the category of representations of its fundamental category.

\subsection*{}
In Section \ref{ch:4sec:1} we recall some classical results on the relationship between preorders and topological spaces. 
Moreover, we study the category of preordered spaces.
In Section \ref{ch:4sec:2} we construct a convenient category of numerically generated stratified spaces as a full subcategory of the category of preordered spaces.
Section \ref{ch:4sec:3} is a recollection of the homotopy theory of stratified spaces.

\subsection*{}
In Section \ref{ch:5sec:1} we give an overview of the theory of prestreams and streams.
The material of the rest of the chapter is original.
In Section \ref{ch:5sec:2} we construct a locally presentable and cartesian closed category of locally stratified spaces.
In Section \ref{ch:5sec:3} we define the homotopy theory of locally stratified spaces and show that fibrant locally stratified spaces form a category of fibrant objects.
In Section \ref{ch:5sec:4} we define the fundamental category of a locally stratified space and the category of left covers.
We show that the fundamental category of a simplicial set admits an essentially surjective functor to the fundamental category of the associated locally stratified space.
To conclude, in Section \ref{ch:5sec:5} we give a characterisation of left covers over the realisation of a simplicial set in terms of \'etale maps and we show that, under suitable hypothesis,
the category of left covers over $A$ is equivalent to the category of left covers over its realisation.
\section*{Conventions}
Throughout the document, we assume the \emph{axiom of universes} and we fix a universe $\cat{U}$ once and for all.
Sets will be understood as $\cat{U}$-sets and classes as $\cat{U}$-classes.

\mainmatter				%

\chapter{Preliminaries}
\label{ch:1}
\section{Prelude}
\label{ch:1sec:1}

Left Kan extensions are ubiquitous in mathematics and they provide a formal way to extend a functor $F\colon \cat{C}\to\cat{E}$ along a functor $K\colon \cat{C}\to \cat{D}$ in a universal way.
In most of our applications, $\cat{E}$ will be a category of topological spaces with some extra structure, $\cat{C}$ will be the category of finite non-empty ordinals and $\cat{D}$ will be the category of simplicial sets.
In such a framework, the existence of a left Kan extension for any functor $F$ along the Yoneda embedding is ensured by Theorem \ref{thm:kan-ext}.
The strength of Theorem \ref{thm:kan-ext} lies in that it does also provide a concrete formula for the value of the left Kan extension of $F$ along $K$ at every object $d\in \cat{D}$.
Moreover, Theorem \ref{thm:kan-presheaves} implies that such a Kan extension has a right adjoint.
Towards the end of the section we illustrate some examples of Kan extensions involving the category of simplicial sets.

\begin{definition}
Let $F\colon \cat{C}\to \cat{E}$ and $K\colon \cat{C}\to \cat{D}$ be functors.
A \emph{right Kan extension} of $F$ along $K$ is a functor 
\[\Ran_{K}F\colon \cat{D}\to \cat{E}\]
together with a natural transformation $\epsilon\colon \Ran_{K}F\circ K\to F$ such that for any other pair $(G\colon \cat{D}\to\cat{E}, \gamma\colon GK\to F)$ there exists a unique natural transformation $\alpha\colon G\to \Ran_{K}F$ such that $\epsilon \alpha K=\gamma$, as shown in the following diagram:
\begin{equation*}
\begin{tikzcd}
{\cat{C}}\ar[rr, "F", ""{name=U, below}] \ar[ddr, "K"']& & \cat{E}\\
&&\\
& \cat{D}.	\ar[Rightarrow, to=U, "\epsilon" {description, pos=0.4}]
		\ar[uur, bend left = 25,  " " {name=V ,below}, "{\Ran_{K}F}" {description, pos=0.65} ]
		\ar[uur, bend right = 25, "G"', " "{name=W}] & 
		\ar[Rightarrow, from=W, to=V, "{\exists !}"]
\end{tikzcd}
\end{equation*}
Dually, a \emph{left Kan extension} of $F$ along $K$ is a functor $\Lan_{K}F\colon \cat{D}\to \cat{E}$ together with a natural transformation $\eta\colon F\to \Lan_{K}F\circ K$ such that for any other pair $(G\colon \cat{D}\to\cat{E}, \delta\colon F\to GK)$ there exists a unique natural transformation $\beta\colon \Lan_{K}F\to G$ such that $\beta K \eta=\delta$, as shown in the following diagram:
\begin{equation*}
\begin{tikzcd}
{\cat{C}}\ar[rr, "F", ""{name=U, below}] \ar[ddr, "K"']& & \cat{E}\\
&&\\
& \cat{D}	\ar[Rightarrow, from=U, "\eta" {description, pos=0.4}]
		\ar[uur, bend left = 25,  " " {name=V ,below}, "{\Lan_{K}F}" {description, pos=0.65} ]
		\ar[uur, bend right = 25, "G"', " "{name=W}] & 
		\ar[Rightarrow, from=V, to=W, "{\exists !}"']
\end{tikzcd}
\end{equation*}
\end{definition}

\subsection{}
Let $F\colon \cat{C}\to \cat{E}$ and $K\colon \cat{C}\to \cat{D}$ be functors and let $\cat{E}^{\cat{D}}$ be the category of functors from $\cat{D}$ to $\cat{E}$ and natural transformations between them.
Then, the universal property of the left Kan extension of $F$ along $K$ is equivalent to the functor
\[\cat{E}^{{\cat{C}}}(F, \blank\circ K)\colon\cat{E}^{\cat{D}}\to \Set\]
being representable.
Indeed, if this is the case, precomposition with $K$ induces a natural isomorphism:
\[\cat{E}^{\cat{D}}(\Lan_{K}F,G)\cong {\cat{E}^{\cat{C}}}(F, G\circ K)\]
whose inverse translates into the universal property of a left Kan extension.

\subsection{}
\label{subsec:kan-ext-adjoints}
In particular, let $K\colon \cat{C}\to \cat{D}$ be a functor, let $\cat{E}$ be a category and assume that the left and right Kan extension for every functor $\cat{C}\to\cat{E}$ along $K$ exist. 
Then, they provide a left and right adjoint to the functor $K^{*}\colon \cat{E}^{\cat{D}}\to\cat{E}^{\cat{C}}$ given by precomposition with $K$.
\begin{equation}
\label{eq:triadjoint-kan-ext}
\Triadjoint{\cat{E}^{\cat{C}}_{\phantom{C}}}{\cat{E}^{\cat{D}}_{\phantom{C}}}{K^{*}}{\Lan_{K}}{\Ran_{K}}
\end{equation}

We now recall the definition of ends and coends.
Coends are special kinds of colimits over a functor in two variables.
In Theorem \ref{thm:kan-ext} we will use a coend formula to compute the left Kan extension of a functor $F\colon \cat{C}\to\cat{E}$ along a functor $K\colon\cat{C}\to\cat{D}$ under suitable hypotheses.

\begin{definition}
Let $H\colon \cat{C^{\op}}\times \cat{C}\to \cat{E}$ be a functor. 
The \emph{end} of $H$ is an object $\int_{{c\in \cat{C}}}H(c,c)$ in $\cat{E}$ together with a morphism $\zeta_{c}\colon \int_{c\in \cat{C}}H(c,c)\to H(c,c)$ for every $c\in \cat{C}$ such that the diagram:
\begin{equation*}
\csquare{\displaystyle\int_{{c\in \cat{C}}}H(c,c)}{H(c', c')}{H(c, c')}{H(c, c)}{\zeta_{c'}}{f^{*}}{f_{*}}{\zeta_{c}}
\end{equation*}
commutes for every morphism $f\colon c\to c'$ and $\left(\int_{c\in\cat{C}}H(c,c),(\zeta_{c})_{c\in \cat{C}}\right)$ is universal among pairs with this property.
When the end of $H$ exists it is given by the equaliser of the diagram:
\begin{equation}
\label{eq:end-equaliser}
\begin{tikzcd}
{\displaystyle\int_{c\in\cat{C}}H(c,c)}
	\ar[r, dashed] &
{\displaystyle\prod_{c} H(c,c)}	
	\ar[r, "{f^{*}}", shift left]
	\ar[r, "{f_{*}}"', shift right] &
{\displaystyle\prod_{c\in \cat{C}}H(c,c')}.
\end{tikzcd}
\end{equation}
Dually, the \emph{coend} of $H$ is an object $\int^{{c\in \cat{C}}}H(c,c)$ in $\cat{E}$ together with a morphism $\theta_{c}\colon H(c,c)\to \int^{c\in \cat{C}}H(c,c)$ for every $c\in \cat{C}$ such that the diagram:
\begin{equation*}
\csquare{H(c',c)}{H(c',c')}{\displaystyle\int^{{c\in \cat{C}}}H(c,c)}{H(c, c)}{f_{*}}{\theta_{c'}}{\theta_{c}}{f^{*}}
\end{equation*}
commutes for every morphism $f\colon c\to c'$ and $\left(\int^{c\in\cat{C}}H(c,c),(\theta_{c})_{c\in \cat{C}}\right)$ is universal among pairs with this property.
When the coend of $H$ exists it is given by the coequaliser of the diagram:
\begin{equation}
\label{eq:coend-coequaliser}
\begin{tikzcd}
{\displaystyle\bigcoprod_{f:c\to c'} H(c',c)}	
	\ar[r, "{f^{*}}", shift left]
	\ar[r, "{f_{*}}"', shift right] &
{\displaystyle\bigcoprod_{c\in \cat{C}}H(c,c)}
	\ar[r, dashed] &
{\displaystyle\int^{c\in\cat{C}}H(c,c)}.
\end{tikzcd}
\end{equation}
\end{definition}

\begin{example}
\label{ex:nat-trans-end}
Let $F,G\colon \cat{C}\to \cat{E}$ be functors and assume that $\cat{C}$ is small and $\cat{E}$ is complete. Consider the functor
\[\cat{E}(F\blank,G\blank)\colon \cat{C}^{\op}\times\cat{C}\to \Set\]
Then, there is a natural isomorphism:
\[\int_{{c\in\cat{C}}}\cat{E}(Fc,Gc)\cong \cat{E^{\cat{C}}}(F,G)\]
Indeed, by inspection one can see that the following diagram is an equaliser of sets:
\begin{diagram}
{ \cat{E^{\cat{C}}}(F,G)}
	\ar[r] &
{\displaystyle\prod_{c\in \cat{C}} \cat{E}(Fc,Gc)}	
	\ar[r, "{f^{*}}", shift left]
	\ar[r, "{f_{*}}"', shift right] &
{\displaystyle\prod_{f\colon c\to c'}\cat{E}(Fc,Gc')}
\end{diagram}
\end{example}

\subsection{} 
\label{subsec:power-copower-set}
Let $\cat{E}$ be a category with arbitrary coproducts.
For every set $S$ and every object $e\in \cat{E}$ the \emph{copower} or \emph{tensor} of $e$ by $S$ is the coproduct 
\[S\otimes e = \bigcoprod_{s\in S}e.\]
of copies of $e$ indexed by $S$.
This defines a functor:
\[\otimes\colon \Set\times \cat{E}\to \cat{E}.\]
Dually, if $\cat{E}$ has arbitrary products, the \emph{power} or \emph{cotensor} of $e$ by $S$ is the product
\[e^{S} =\prod_{s\in S}e\]
of $S$ many copies of $e$.

\begin{theorem}
\label{thm:kan-ext}
Let $\cat{C}$ be a small category, let $\cat{D}$ be a locally small category and let $\cat{E}$ be a cocomplete category. 
Then, the left Kan extension of any functor $F\colon \cat{C}\to \cat{E}$ along any functor $K\colon \cat{C}\to \cat{D}$ exists and is given by the formula:
\begin{equation}
\label{eq:lan-coend}
\Lan_{K}F(d)=\int^{c\in \cat{C}}\cat{D}(Kc,d)\otimes Fc.
\end{equation}
\end{theorem}
\begin{proof}
See \cite[X.4.1-2]{mac13}
\end{proof}

\subsection{}
Let $F\colon \cat{C}\to \Set$ be a functor from a category $\cat{C}$ to the category of sets.
The \emph{category of elements} of $F$, denoted by $\Elts(F)$, is the category whose objects are pairs $(c,f)$ where $c\in \cat{C}$ and $f\in F(c)$ and morphisms are maps  $\alpha\colon c\to c'$ in $\cat{C}$ such that $F(\alpha)(f)=f'$.
The category of elements of $F$ comes equipped with a forgetful functor
\begin{align*}
\phi_{F}\colon\Elts(F)&\to\cat{C}\\
(c,f)&\mapsto c
\end{align*}
The pair $(\Elts(F),\phi_{F})$ is often called the \emph{Grothendieck construction} of $F$.

\subsection{}
In particular, for any functor $K\colon \cat{C}\to \cat{D}$ between locally small categories and any object $d\in \cat{D}$ we can consider the category of elements of the functor ${\cat{D}}(K\blank,d)$ and we denote it as $\overcat{K}{d}$ and call it the \emph{overcategory} of $K$ over $d$. Unraveling the definition, the objects of $\overcat{K}{d}$ are pairs $(c,f)$ where $c$ is an object of $\cat{C}$ and $f\colon Kc\to d$ is a morphism in $\cat{D}$ and a morphism $\alpha\colon (c,f)\to (c',f')$ is a map $\alpha\colon c\to c'$ in $\cat{C}$ such that the triangle
\begin{diagram}
 & Kc'\ar[d, "f'"]\\
Kc\ar[ur, "K\alpha"]\ar[r,"f"']&d
\end{diagram}
commutes.
When $K$ is a fully faithful functor, we will use the notation $\overcat{\cat{C}}{d}$ and leave the functor $K$ implicit.
In particular, when $K$ is the identity functor on $\cat{C}$, we recover the \emph{slice category} $\overcat{\cat{C}}{c}$ of $\cat{C}$ over $c$.
Dually, the \emph{undercategory} of $K$ under $d$, denoted by $\undercat{K}{d}$ is the category of elements of the functor ${\cat{D}}(d, K\blank)$ and when $K$ is the identity functor on $\cat{C}$ we recover the \emph{coslice} category $\coslice{\cat{C}}{c}$ of $\cat{C}$ over $c$.

\begin{notation}
\label{not:grothendieck-construction-colimit}
Let $K\colon \cat{C}\to \cat{D}$ be a functor between locally small categories and let $F\colon \cat{C}\to \cat{E}$ be a functor to a cocomplete category.
Then, we will use the notation
\[\colim_{Kc\to d}Fc= \colim\left(\overcat{K}{d}\to \cat{C}\stackrel{F}{\to} \cat{E}\right)\]
to denote the colimit of the composite of $F$ with the forgetful functor of the Grothendieck construction associated to the functor ${\cat{D}}(K\blank,d)$.
\end{notation}

\begin{remark}
\label{rmk:pointwise-kan-ext}
Let $F\colon \cat{C}\to \cat{E}$ and $K\colon \cat{C}\to \cat{D}$ be functors as in Theorem \ref{thm:kan-ext}. 
Then, by Equations \eqref{eq:coend-coequaliser} and \eqref{eq:lan-coend} the left Kan extension of $F$ along $K$ at $d$ can be computed as the colimit:
\begin{equation*}
\Lan_{K}F(d)=\colim_{Kc\to d}Fc.
\end{equation*}
\end{remark}

\begin{example}
\label{ex:kan-ext-colimit}
Let $!\colon \cat{C}\to \term$ be the unique functor from a small category $\cat{C}$ to the terminal category and let $F\colon \cat{C}\to \cat{E}$ be a functor to a cocomplete category.
Then, Remark \ref{rmk:pointwise-kan-ext} implies the left Kan extension of $F$ along $!$ is simply the colimit of $F$:
\[\Lan_{!}F=\colim_{\cat{C}}F.\]
Dually, the right Kan extension of $F$ along $!$ is the limit of $F$. 
In particular, \eqref{eq:triadjoint-kan-ext} specialises to the following diagram:
\begin{equation*}
\Triadjoint{\cat{E}^{\cat{C}}_{\phantom{C}}}{\cat{E}^{\phantom{\cat{D}}}_{\phantom{C}}}{\const_{\cat{C}}}{\colim_{\cat{C}}}{\lim_{\cat{C}}}.
\end{equation*}
\end{example}

\begin{definition}
Let $F\colon \cat{C}\to \cat{E}$ be a functor, where $\cat{C}$ is a small category and $\cat{E}$ is a cocomplete, locally small category. We say that $F$ is \emph{dense} if $\Lan_{F}(F)=\id_{\cat{E}}$.
By Remark \ref{rmk:pointwise-kan-ext} this is equivalent to saying that for every object $e\in \cat{E}$ there is a natural isomorphism:
\begin{equation*}
e\cong \colim_{Fc\to e}Fc.
\end{equation*}
\end{definition}

\begin{theorem}[Yoneda Lemma]
\label{thm:yoneda-lemma}
Let $F\colon \cat{C}\to \Set$ be a functor, then there exists a natural isomorphism
\[Fc\cong \Set^{\cat{C}}(\cat{C}(c,\blank), F).\]
\end{theorem}
\begin{proof}
We have the following chain of isomorphisms
\begin{align*}
Fc 	&\cong \Ran_{\id_{\cat{C}}}Fc\\
	&\cong \int_{x\in \cat{C}}(Fx)^{\cat{C}(x,c)}\\
	&\cong \int_{x\in \cat{C}}\Set(\cat{C}(x,c),Fx)\\
	&\cong \Set^{\cat{C}}\left(\cat{C}(c,\blank), F\right)
\end{align*}
where the second isomorphism follows from the dual of Theorem \ref{thm:kan-ext} and the last from Example \ref{ex:nat-trans-end}.
\end{proof}

We now recall the definition of the category of presheaves on a small category and specialise Theorem \ref{thm:kan-ext} to this context.
\subsection{}
Let $\cat{A}$ be a small category, recall that a presheaf on $\cat{A}$ is a functor $X\colon \cat{A}^{\op}\to \Set$. We denote by $\hat{\cat{A}}$ the category $\Set^{\cat{A}^{\op}}$ of presheaves on $\cat{A}$.
The functor
\begin{align*}
	h \colon \cat{A} & \to \hat{\cat{A}}\\
			a	& \mapsto \cat{A}(\blank, a)
\end{align*}
is fully faithful by Theorem \ref{thm:yoneda-lemma} and is called the \emph{Yoneda embedding} of $\cat{A}$.

\subsection{} Let $X$ be a presheaf on a small category $\cat{A}$. 
The Yoneda Lemma implies that the category $\Elts(X)$ of elements of $X$ is isomorphic to the overcategory $\overcat{\cat{A}}{X}$ of $h$ over $X$.

\begin{theorem}
\label{thm:kan-presheaves}
Let $\cat{A}$ be a small category and let $\cat{E}$ be a cocomplete category. Any functor $u\colon \cat{A}\to \cat{E}$ induces an adjunction:
\begin{equation*}
\Adjoint{\hat{\cat{A}}}{\cat{E}}{u^{!}}{u_{*}}
\end{equation*}
where $u^{!}=\Lan_{h}u$ and $u_{*}e$ is the presheaf defined by
\[u_{*}e(a)=\cat{E}(u(a), e)\]
\end{theorem}
\begin{proof}
We have the following natural isomorphisms:
\begin{align*}
\cat{E}(\Lan_{h}u(X),e) 
	& \cong\cat{E}\left(\int^{a\in \cat{A}}X(a)\otimes ua,e\right)\\
	&\cong \int_{a\in\cat{A}}\cat{E}\left(X(a)\otimes ua, e\right)\\
	&\cong \int_{a\in \cat{A}}\Set\left(X(a),\cat{E}(u(a),e)\right)\\
	&\cong \hat{\cat{A}}(X,u_{*}(e))
\end{align*}
Where the first isomorphism follows from Theorem \ref{thm:kan-ext} together with the Yoneda Lemma and the last isomorphism from Example \ref{ex:nat-trans-end}.
\end{proof}

The following Corollary of Theorem \ref{thm:kan-presheaves} is the well known fact that every presheaf is, canonically, a colimit of representable presheaves.
\begin{corollary}[Density Theorem]
\label{cor:density-theorem}
The Yoneda embedding is a dense functor. More explicitly, for every presheaf $X$ on $\cat{A}$ there is a canonical isomorphism:
\[X\cong \colim_{\cat{A}(\blank, a)\to X}\cat{A}(\blank, a)\]
\end{corollary}
\begin{proof}
By Theorem \ref{thm:kan-presheaves} and the Yoneda Lemma we have the following natural isomorphisms:
\begin{align*}
\hat{\cat{A}}(\Lan_{h}h(X),Y)
	&\cong \hat{\cat{A}}(X,h_{*}(Y))\\
	&\cong \hat{\cat{A}}(X,Y).
\end{align*}
\end{proof}

We introduce now the category $\DDelta$ of finite non-empty ordinals and give a few examples of adjunctions coming from Theorem \ref{thm:kan-presheaves} involving the category of presheaves on $\DDelta$.
\subsection{} 
\label{subsec:simplicial-sets}
Let $\DDelta$ be the category whose objects are the finite non-empty ordinals $[n]=\{0<1<\ldots<n\}$ and morphisms are the weakly monotonic maps between them (see Definition \ref{def:preorder}).
A presheaf $X$ on $\DDelta$ is said to be a \emph{simplicial set}. 
The set of $n$-simplices of a simplicial set $X$ is the evaluation $X([n])$ of $X$ at $[n]$ and it will be denoted by $X_{n}$.
The category of simplicial sets will be denoted by $\sSet=\hat{\DDelta}$.

\begin{example}
\label{ex:nerve-fundamental-category}
The category $\DDelta$ of finite ordinals comes equipped with a full embedding $\iota\colon \DDelta\to\Cat$ to the category of small categories.
By Theorem \ref{thm:kan-presheaves} the inclusion functor induces an adjunction:
\begin{equation*}
\Adjoint{\sSet}{\Cat}{\tau_1}{\Nerv}
\end{equation*}
where $\tau_1$, the left Kan extension of $\iota$ along the Yoneda embedding, is called the \emph{fundamental category} functor and $\Nerv$ is the \emph{nerve functor}.
In particular, if $C$ is a small category, the $n$-simplices of $\Nerv(C)$ are given by:
\[\Nerv(C)_{n}=\Cat([n],C)\]
In other words, they are composable strings of $n$-morphisms in $C$.
\end{example}

\begin{example}
\label{ex:geometric-realisation}
Let $\Topa$ be the category of topological spaces. For every $n\ge 0$ the \emph{geometric standard $n$-simplex} is the topological space given by:
\begin{equation}
\abs{\Delta^n}= \left\{\left(t_0,\ldots, t_n\right) \in \RR^{n+1}: t_i\ge 0,\quad \sum_{i=0}^nt_i = 1\right\}
\end{equation}
This, gives rise to a functor, $i\colon\Delta\to \Topa$ which induces an adjunction 
\begin{equation*}
\Adjoint{\sSet}{\Topa}{\abs{\blank}}{\Sing}
\end{equation*}
by Theorem \ref{thm:kan-presheaves}.
The left adjoint is called the \emph{geometric realisation functor} and the right adjoint is called the \emph{singular complex functor}.
\end{example}

We conclude the section with a brief recollection of the main facts about monads and comonads.
For the purposes of this thesis, we are mainly interested in the concept of idempotent comonad (see \ref{def:idempotent-monad}) and its relationship with coreflective subcategories (see Proposition \ref{prop:idempotent-monad}) as we will use it to give a categorical definition of the category of Streams in Definition \ref{def:stream}.

\subsection{}
Recall that a \emph{monad} on a category $\cat{C}$ is a triple $\Topp=(\top,\mu, \eta)$ where $\top\colon \cat{C}\to\cat{C}$ is an endofunctor on $\cat{C}$ and 
\[\mu\colon \top\top\to\top,\quad \eta\colon \id_{\cat{C}}\to\top\]
are an associative and a unital natural transformation.
Dually, a \emph{comonad} on $\cat{C}$ is a triple $\Bot=(\bot,\nu,\epsilon)$ where $\bot\colon \cat{C}\to\cat{C}$ is an endofunctor on $\cat{C}$ and
\[\nu\colon \bot\to\bot\bot,\quad \eta\colon \bot\to \id_{\cat{C}}\]
are a coassociative and a counital natural transformation.

\subsection{}
\label{subsec:adjunction-monad}
Let $\cat{C}$ and $\cat{D}$ be categories and let 
\begin{equation}
\Adjoint{\cat{C}}{\cat{D}}{F}{U}
\end{equation}
be an adjunction.
Then, the endofunctor $UF\colon \cat{C}\to\cat{C}$ has the structure of a monad, with multiplication and unit given by 
\[U\epsilon F\colon UFUF\to UF, \quad \eta\colon \id_{\cat{C}}\to UF\]
where $\eta$ is the unit and $\epsilon$ is the counit of the adjunction.
Dually, we can define a comonad $FU\colon \cat{D}\to\cat{D}$ with comultiplication and counit given by
\[F\eta U\colon FU\to FUFU,\quad \epsilon\colon FU\to \id_{\cat{D}}.\]

\begin{example}
Let $\Set_{\term}$ be the category of pointed sets with basepoint preserving functions between them. 
The forgetful functor $U\colon \Set_{\term}\to \Set$ has a left adjoint
\begin{align*}
(\blank)_{+}\colon \Set & \to \Set_{*}\\
S &\mapsto S\coprod \term
\end{align*}
given by adding a disjoint basepoint.
The adjuncion defines a monad on $\Set$ by \ref{subsec:adjunction-monad}.
\end{example}

\subsection{}
Let $\Bot=(\bot,\nu,\epsilon)$ be a comonad on a category $\cat{C}$. 
A $\Bot$-\emph{algebra} is an object $a$ of $\cat{C}$ together with a morphism $\alpha\colon a\to \bot a$ such that the diagrams
\begin{equation}
\csquare{a}{\bot a}{\bot\bot a}{\bot a}{\alpha}{\nu_{a}}{\bot\alpha}{\alpha}\quad
\begin{tikzcd}
a	\ar[r, "\alpha"]
	\ar[dr, "\id_{a}"]&
\bot a\ar[d, "\epsilon_{a}"]\\
&
a
\end{tikzcd}
\end{equation}
commute.
A morphism of $\Bot$-algebras is a map $f\colon a\to a'$ in $\cat{C}$ such that the diagram

\begin{equation}
\csquare{a}{a'}{\bot a'}{\bot a}{f}{\alpha'}{\bot f}{\alpha}
\end{equation}
commutes.
We denote by $\cat{C}^{\Bot}$ the category of $\Bot$-algebras in $\cat{C}$. 
This comes equipped with a forgetful functor $U\colon\cat{C}^{\Bot}\to\cat{C}$ which has a left adjoint $F\colon \cat{C}\to\cat{C}^{\Bot}$ sending an object $c\in \cat{C}$ to the \emph{free} $\Bot$-\emph{algebra} $F(c)$ on $c$ given by $(\bot c,\epsilon_{c}\colon \bot c\to c)$.

\begin{definition}
\label{def:idempotent-monad}
A monad $\Topp=(\top,\mu,\eta)$ on a category $\cat{C}$ is said to be an \emph{idempotent monad} if the multiplication map $\mu\colon \top\top\to\top$ is a natural isomorphism. 
Dually, a comonad $\Bot=(\bot,\nu,\epsilon)$ in $\cat{C}$ is said to be an \emph{idempotent comonad} if the comultiplication map $\nu\colon \bot\to\bot\bot$ is a natural isomorphism.
\end{definition}

\subsection{}
Recall that a subcategory $\iota\colon\cat{A}\to \cat{C}$ is said to be \emph{reflective} (\resp \emph{coreflective}) if the inclusion functor $\iota$ has a left adjoint (\resp a right adjoint).

\begin{proposition}
\label{prop:idempotent-monad}
Let $\Bot=(\bot,\nu,\epsilon)$ be a comonad on a category $\cat{C}$, the following are equivalent
\begin{enumerate}
\item
The comonad $\Bot$ is idempotent,
\item
The maps $\bot\epsilon,\epsilon\bot\colon \bot\to\bot\bot$ are equal
\item
The category $\cat{C}^{\Bot}$ of $\Bot$-algebras is a coreflective subcategory of $\cat{C}$.
\end{enumerate}
\end{proposition}
\begin{proof}
See \cite[Proposition 4.2.3.]{bor942} for a proof of the dual statement.
\end{proof}

\begin{example}
\label{ex:skeleton-coskeleton}
Let $\DDelta_{{\le n}}$ be the full subcategory of $\DDelta$ spanned by the finite ordinals $[m]$ with $m\le n$.
The category of presheaves on $\DDelta_{\le n}$ is the category of $n$-\emph{truncated simplicial sets} and will be denoted by $\sSet_{\le n}$.
By \ref{subsec:kan-ext-adjoints} the inclusion functor $\iota_{n}\colon\DDelta_{{\le n}}\to \DDelta$ induces adjunctions:
\begin{equation}
\label{eq:n-truncated-simplicial-set}
\Triadjoint{\sSet_{\le n}}{\sSet}{\iota_{n}^{*}}{\Lan_{\iota_{n}}}{\Ran_{\iota_{n}}}
\end{equation}
The composites
\begin{align*}
\sk_{n}=\Lan_{\iota_{n}}\circ \iota_{n}^{*}\colon \sSet		& \to \sSet\\
\cosk_{n}=\Ran_{\iota_{n}}\circ\iota_{n}^{*}\colon\sSet	& \to \sSet
\end{align*}
are called the $n$-\emph{skeleton} and $n$-\emph{coskeleton} functor, respectively.
Being a monad and comonad coming from an adjuntion of this form, they form themselves an adjunction:
\begin{equation*}
\Adjoint{\sSet}{\sSet}{\sk_{n}}{\cosk_{n}}
\end{equation*}
\end{example}

\begin{example}
\label{ex:0-truncated-simplicial-set}
A $0$-truncated simplicial set is a presheaf of sets on the terminal category, hence just a set. Thus, for $n=0$ the adjunction \eqref{eq:n-truncated-simplicial-set} specialises to:
\begin{equation}
\label{eq:0-truncated-simplicial-set}
\Triadjoint{\Set}{\sSet}{\ev_{0}}{\const}{\Ran_{\iota_{0}}}
\end{equation}
where $\ev_{0}\colon \sSet\to\Set$ is the functor that evaluates a simplicial set to its set of $0$-simplices,
while $\const_{\DDelta^{\op}}\colon \Set\to \sSet$ is the constant functor.
In particular, $\const_{\DDelta^{\op}}$ has itself a left adjoint given by the colimit functor (see Example \ref{ex:kan-ext-colimit}).
For every $X\in \sSet$, we denote by $\pi_{0}X$ the colimit of $X$ and call it the set of \emph{connected components} of $X$.
A simple calculation shows that $\pi_{0}X$ fits in the following coequaliser diagram:
\begin{equation}
\begin{tikzcd}
X_{1}\ar[r, shift left, "d^{0}_{1}"]
	\ar[r, shift right, "d^{1}_{1}"']&
X_{0}\ar[r, dashed] &
\pi_{0}X
\end{tikzcd}
\end{equation}
where $d^{i}_{1}\colon X_{1}\to X_{0}$ corresponds to the map $\partial^{1}_{i}\colon [0]\to [1]$ that skips the value $i$ (see \ref{subsec:face-degeneracy} and Notation \ref{not:simplicial-sets}).
\end{example}

\begin{example}
\label{ex:reflexive-quiver}
For $n=1$ the category $\DDelta_{{\le 1}}$ is given by
\begin{diagram}
{[0]}\ar[r, "\partial^{1}_{0}", shift left=2]\ar[r,"\partial^{1}_{1}"', shift right=2]& 
{[1]}\ar[l, "\sigma^{1}_{0}" description]
\end{diagram}
where  $\sigma_{0}^{0}$ is the unique map from $[1]$ to $[0]$ (see also \ref{subsec:face-degeneracy}).
A $1$-truncated simplicial set is said to be a \emph{reflexive quiver}.
Unraveling the definition, a reflexive quiver is a set $Q_{0}$ of objects and a set $Q_{1}$ of edges, with source and target maps $s,t \colon Q_{1}\to Q_{0}$ and a map $1\colon Q_{0}\to Q_{1}$ that assigns to every object $q\in Q$ the identity edge $1_{q}$ on $q$.
We denote by $\RefQuiv$ the category of reflexive quivers.
Notice that every small category has an underlying reflexive quiver and we have an adjunction
\begin{equation*}
\Adjoint{\RefQuiv}{\Cat}{F}{U}
\end{equation*}
where $U\colon \Cat\to \RefQuiv$ is the forgetful functor and $F\colon \RefQuiv\to\Cat$ is the functor that associates to every reflexive quiver $Q$ the \emph{free category} on $Q$.
\end{example}

\begin{example}
\label{ex:core-gpd}
Recall that a \emph{groupoid} is a category $\cat{G}$ such that every morphism in $G$ is an isomorphism.
Let $\Gpd$ be the category of small groupoids, then the forgetful functor $U\colon \Gpd\to \Cat$ admits a right adjoint:
\[\core{(\blank)}\colon \Cat \to \Gpd\]
that associates to every category $C$ the \emph{core groupoid} $\core{C}$ of $C$, the maximal subgroupoid of $C$.
Moreover, $U$ has also a left adjoint:
\[\gpd{(\blank)}\colon \Cat\to \Gpd\]
that associates to every category $C$ the \emph{groupoidification} $\gpd{C}$ of $C$ by formally inverting all the morphisms in $C$ (See \cite[Proposition 5.2.2]{bor941}).
\end{example}

\section{Monoidal and enriched categories}
\label{ch:1sec:2}
In the first part of the section, we give a brief overview of the theory of monoidal categories.
For the purposes of this thesis, we are mainly interested in Cartesian monoidal categories and Cartesian closed categories (see \ref{subsec:cartesian-closed-categories}).
The second half of the section is a recollection of the basics of enriched category theory.
In most of our examples, the base of the enrichment will be the category of simplicial sets. 
In particular, Theorem \ref{cor:adjunction-v-adjunction} will be used extensively throughout the text.
We conclude the section with functor tensor products, as they play a role in Sections \ref{ch:3sec:2}, \ref{ch:3sec:3} and \ref{ch:5sec:4}.

\subsection{}										%
Recall that a \emph{monoidal category} is a triple $(\cat{V}, \otimes, \1)$ where $\cat{V}$ is a category,
$\otimes\colon \cat{V}\times\cat{V}\to \cat{V}$ is a functor called the \emph{monoidal product} and $\1$ is an object of $\cat{V}$ called the \emph{unit object}, together with natural isomorphisms:
\begin{equation}
\label{eq:monoidal}
u\otimes(v\otimes w)\cong (u\otimes v)\otimes w, \quad \1\otimes v \cong v\cong v \otimes \1.
\end{equation}
This data is subject to a list of coherence conditions that ensure that any diagram involving the isomorphisms in \eqref{eq:monoidal} is commutative (see \cite[Chapter XI]{mac13}).

\begin{example}									%
\label{ex:endofunctors}
Let $\cat{C}$ be a category and let $\iEnd(\cat{C})$ be the category whose objects are the \emph{endofunctors} of $\cat{C}$ \ie the functors from $\cat{C}$ to itself, and whose morphisms are the natural transformations between endofunctors.
Composition of endofunctors defines a functor $\circ\colon \iEnd(\cat{C})\times\iEnd(\cat{C})\to\iEnd(\cat{C})$ which endows $\iEnd(\cat{C})$ with the structure of a monoidal category, with unit object given by the identity functor $\id_\cat{C}\colon \cat{C}\to\cat{C}$. Such monoidal structure is \emph{strict} in the sense that the isomorphisms in \eqref{eq:monoidal} become the following identities
\[F\circ(G\circ H)=(F\circ G)\circ H,\quad \id_{\cat{C}}\circ F=F=F\circ \id_\cat{C}.\]
\end{example}

\subsection{}										%
A \emph{lax monoidal functor} $F\colon (\cat{V},\otimes,\1)\to (\cat{V'},\otimes',\1')$ between monoidal categories is a functor $F\colon \cat{V}\to \cat{V'}$ together with a natural transformation and a morphism
\begin{equation}
\label{eq:monoidal-functor}
Fu\otimes' Fv\to F(u\otimes v), \quad \1'\to F(\1)
\end{equation}
satisfying associative and unital coherence conditions.
A \emph{strong monoidal functor} is a lax monoidal functor $F\colon (\cat{V},\otimes,\1)\to (\cat{V'},\otimes',\1')$ for which the natural transformations in \eqref{eq:monoidal-functor} are natural isomorphisms.

\begin{definition}									%
Let $\cat{C}$ be a category and let $(\cat{V},\otimes,\1)$ be a monoidal category.
A \emph{left action} of $\cat{V}$ on $\cat{C}$ is a strong monoidal functor from $\cat{V}$ to the monoidal category of endofunctors of $\cat{C}$.
Equivalently, a left action of $\cat{V}$ on $\cat{C}$ is given by a bifunctor $\otimes \colon \cat{V}\times\cat{C}\to \cat{C}$ together with natural isomorphisms:
\begin{equation}
\label{eq:action}
\1\otimes x\cong x, \quad (v\otimes w)\otimes x\cong v\otimes(w\otimes x)
\end{equation}
for every $v$ and $w$ in $\cat{V}$ and for every $x$ in $\cat{C}$, satisfying associative and unital coherence conditions.
\end{definition}

\begin{example}									%
Every monoidal category $(\cat{V},\otimes,\1)$ acts on itself via the monoidal product $\otimes\colon \cat{V}\times \cat{V}\to \cat{V}$.
\end{example}

\begin{example}									%
Let $\cat{C}$ be a category. The identity on $\iEnd(\cat{C})$ defines a tautological action of $\iEnd(\cat{C})$ on $\cat{C}$ which corresponds by transposition to the evaluation map
\begin{align*}
\iEnd(\cat{C})\times\cat{C} & \to \cat{C}\\
(F,x)&\mapsto F(x)
\end{align*}
\end{example}

\subsection{}										%
A \emph{symmetric monoidal category} is a monoidal category $(\cat{V},\otimes,\1)$ together with a natural isomorphism:
\[u\otimes v\cong v\otimes u\]
satisfying a compatibility condition with the associator.

\begin{example}								%
\label{ex:cartesian-monoidal}
Let $\cat{C}$ be a category with finite products, then the Cartesian product $\times\colon \cat{C}\times\cat{C}\to \cat{C}$ defines the structure of a symmetric monoidal category on $\cat{C}$ with unit object given by the terminal object $\term$. 
If this is the case, $(\cat{C},\times,\term)$ is said to be a \emph{Cartesian monoidal category}.
In particular, many well known categories such as $\Set$, $\Topa$, $\sSet$ and $\Cat$ are Cartesian monoidal.
\end{example}

\begin{example}								%
The representable functor $\hom_{\cat{V}}(\1,\blank)\colon\cat{V}\to\Set$ defines a lax monoidal functor from $(\cat{V},\otimes,\1)$ to $(\Set,\times,\term)$.
Indeed, the isomorphism $\1\cong \1\otimes\1$ induces an associative morphism
\[\hom_{\cat{V}}(\1,u)\times\hom_{\cat{V}}(\1,v)\to\hom_{\cat{V}}(\1,u\otimes v)\]
natural in $u$ and $v$ and the identity on $\1$ is equivalently a map
\[\term\to \hom_{\cat{V}}(\1,\1)\]
satisfying unital coherence conditions.
Moreover, if $\cat{C}$ is Cartesian monoidal, then the functor $\hom_{\cat{C}}(\term,\blank)$ is strong monoidal.
\end{example}

\begin{example}									%
Let $\Ab$ be the category of abelian groups and group homomorphisms between them.
The tensor product of abelian groups defines a symmetric monoidal structure on $\Ab$
\[\otimes_{\ZZ}\colon \Ab\times \Ab\to \Ab\]
whose unit object is given by the integers $\ZZ\in \Ab$.
\end{example}

\begin{definition}												%
A symmetric monoidal category $(\cat{V},\otimes,\1)$ is said to be \emph{closed} if for every object $v\in \cat{V}$, the functor $\blank\otimes v\colon \cat{V}\to \cat{V}$ admits a right adjoint
$\ihom_{\cat{V}}(v,\blank)\colon \cat{V}\to \cat{V}$ called \emph{internal-hom}.
\end{definition}

\begin{example}												%
\label{ex:set-cartesian-closed}
The prototypical example of a closed symmetric monoidal category is given by the Cartesian monoidal category $(\Set,\times,\term)$ of sets and maps between them.
In this case, given two sets $S$ and $T$ the internal-hom is given by the hom-set:
\begin{equation}
\label{eq:exponential-in-set}
\ihom_{\Set}(S,T) \cong \hom_{\Set}(S,T) \cong T^{S}
\end{equation}
where $T^{S}$ stands for the power $\prod_{s\in S}T$.
\end{example}

\begin{example}												%
\label{ex:ab-closed-monoidal}
The symmetric monoidal category $(\Ab,\otimes_{\ZZ},\ZZ)$ of abelian groups is closed.
Indeed, pointwise sum gives the set $\hom_{\Ab}(A,B)$ of group homomorphisms between any two abelian groups $A$ and $B$ the structure of an abelian group $\ihom_{\Ab}(A,B)$.
\end{example}

\subsection{}													%
\label{subsec:cartesian-closed-categories}
Generalising Example \ref{ex:set-cartesian-closed}, let $(\cat{C}, \times,\term)$ be a Cartesian monoidal category.
If $x$ is an object of $\cat{C}$ and the functor $\blank\times x\colon \cat{C}\to \cat{C}$ has a right adjoint, we say that $x$ is \emph{exponentiable} and, for every object $y\in \cat{C}$, we call
$\ihom_{\cat{C}}(x,y)$ the \emph{exponential} of $y$ by $x$ and we use interchangeably the notation $y^{x}$.
If every object in $\cat{C}$ is exponentiable, \ie if $\cat{C}$ is a closed symmetric monoidal category with respect to the Cartesian product, then $(\cat{C},\times,\term)$ is said to be  \emph{Cartesian closed}.

\begin{example}												%
\label{ex:cat-cartesian-closed}
Let $\Cat$ be the category of small categories and functors between them.
Then, for any two small categories $C$ and $D$, the category $D^{C}$ whose objects are functors from $C$ to $D$ and morphisms are natural transformations defines an exponential object.
In particular, $\Cat$ is Cartesian closed.
\end{example}

\begin{example}									%
The category $\hat{A}$ of presheaves of sets on a small category $A$ is Cartesian closed, where the internal hom from a presheaf $X$ to a presheaf $Y$ is given by the presheaf $\ihom(X,Y)$ defined by the rule:
\[\ihom(X,Y)(a)=\ihom(A(\blank,a)\times X,Y)\]
for every object $a\in A$. 
In particular, specialising to the category $\DDelta$ of finite non-empty ordinals we deduce that the category $\sSet$ of simplicial sets is Cartesian closed.
\end{example}

\begin{example}									%
The category $\Topa$ of all topological spaces is not Cartesian closed (see \cite[Proposition 7.1.2.]{bor942}). 
However, one can consider different subcategories of the category of topological spaces which are \emph{convenient} in the sense of \cite{ste97} and in particular Cartesian closed.
We adopt the category $\Top$ of numerically generated spaces, which we introduce in Example \ref{ex:numerically-generated}, as our preferred Cartesian closed category of topological spaces.
\end{example}

\begin{construction}												%
\label{con:pointed-monoidal-category}
Let $\cat{C}$ be a category with a terminal object $\term$.
The category of pointed objects $\cat{C}_{\term}$ in $\cat{C}$ is the coslice $\coslice{\cat{C}}{\term}$ under the terminal object of $\cat{C}$.
Assume that $\cat{C}$ is a Cartesian closed category and that it is complete and cocomplete for simplicity.
Then, the category $\cat{C}_{\term}$ of pointed objects in $\cat{C}$ inherits the structure of a Cartesian closed category as follows.
For every two pointed objects $c$ and $c'$ in $\cat{C}$ let $c\wedge c'$ be the pushout diagram in $\cat{C}$
\begin{equation*}
\push{c\coprod c'}{c\times c'}{c\wedge c'}{\term}{(1_{c},1_{c'})}{}{}{}
\end{equation*}
where the top horizontal map is induced canonically by the basepoints of $c$ and $c'$.
Let $c\wedge c'$ be equipped with the natural basepoint $\term\cong\term\times\term\to c\times c'\to c\wedge c'$.
Then, the functor $\wedge\colon \cat{C}_{\term}\times\cat{C}_{\term}\to \cat{C}_{\term}$ defines a symmetric monoidal product, with unit object given by the coproduct $\term\coprod\term$ of two copies of the terminal object.
Moreover, let us consider the following pullback diagram in $\cat{C}$
\begin{equation*}
\pull{\ihom_{\term}(c,c')}{\ihom(c,c')}{\ihom(\term,c')}{\ihom(\term,\term)}{}{}{}{}
\end{equation*}
where the right vertical map is induced by precomposition with the basepoint of $c$ and the bottom horizontal map by postcomposition with the basepoint of $c'$.
One can show that the functor 
\[\ihom_{\term}\colon\cat{C}_{\term}^{\op}\times\cat{C}_{\term}\to\cat{C}_{\term}\]
defines an internal hom functor for the monoidal category $(\cat{C}_{\term},\wedge,\term\coprod\term)$.
\end{construction}

\subsection{}												%
The forgetful functor $U\colon \cat{C}_{\term}\to \cat{C}$ has a left adjoint
\[(\blank)_{+}\colon \cat{C}\to\cat{C}_{\term}\]
given by adding a basepoint.
Moreover, if $\cat{C}$ is a Cartesian closed category, the functor $(\blank)_{+}$ defines a strong monoidal functor and one has the following natural isomorphisms:
\[(\term)_{+}\cong \term\coprod\term, \quad (c\times c')_{+}\cong c_{+}\wedge c'_{+}.\]

We conclude the first half of this section introducing left  cones and coslices.
These constructions will be useful in \ref{subsec:left-cone-strat} when we will specialise them to the category of stratified spaces.

\subsection{}												%
\label{subsec:segment-object}
Following \cite{BM03} we say that an object $a$ in a symmetric monoidal category $\cat{V}$ is a \emph{segment object} if it comes equipped with a decomposition
\begin{equation*}
\begin{tikzcd}
\1\coprod \1 	\ar[r, "{(0,1)}"']
			\ar[rr, bend left]&
a
			\ar[r, "\epsilon"']&
\1
\end{tikzcd}
\end{equation*}
of the codiagonal on the unit object $\1$ and an associative morphism
\[m\colon a\otimes a\to a\]
which has $0$ as its neutral element and $1$ as its absorbing element and such that $\epsilon$ is a counit for $m$.

\begin{construction}							%
\label{subsec:left-cone}
Let $\cat{C}$ be a Cartesian monoidal category with a segment object $a$ and let $c$ be an object of $\cat{C}$.
Assume that $\cat{C}$ has finite limits and colimits.
The $a$-based \emph{left cone} of $c$ is the pushout in $\cat{C}$.
\begin{equation}
\push{\term\times c}{a\times c}{\lcone{c}.}{\term}{0\times \id_{a}}{}{u}{}
\end{equation}
This defines a functor 
\[\lcone{\blank}\colon\cat{C}\to \cat{C}\]
which naturally factors through the category of pointed objects in $\cat{C}$.
Moreover, if the object $a$ is exponentiable, and $c$ is a pointed object in $\cat{C}$ with basepoint $u\colon\term\to c$, the \emph{coslice} or \emph{based path space} of $c$ at $u$ is the pullback
\begin{equation}
\pull{\lpath{c}}{c^{a}}{c.}{\term}{}{c^{0}}{u}{}
\end{equation}
This defines a coslice functor 
\[\lpath{\blank}\colon \cat{C}_{\term}\to\cat{C}_{\term}.\]
Moreover, such functors induce an adjunction
\begin{equation}
\Adjoint{\cat{C}}{\cat{C}_{\term}.}{\lcone{\blank}}{\lpath{\blank}}
\end{equation}
\end{construction}

\begin{example}												%
\label{ex:cone-top}
The category $\Topa$ of all topological spaces is cartesian monoidal and the standard interval $\II=[0,1]$ has the structure of a segment object, with multiplication:
\begin{align*}
\min\colon \II\times\II&\to \II\\
(t,t')&\mapsto \min(t,t')
\end{align*}
For every topological space $X$, the left cone $\lcone{X}$ on $X$ is the classical cone on $X$.
In particular, notice that the left cone on the geometric standard $(n-1)$-simplex is the geometric standard $n$-simplex.
Since $\II$ is a compact Hausdorff space, and in particular a core-compact space, $\II$ is exponentiable (see \ref{subsec:top-exponential}).
Therefore, the cone functor has a right adjoint which sends a pointed topological space $(X,x)$ to the usual based path space $\lpath{X}$ of $X$ at $x$.
\end{example}

\begin{example}												%
\label{ex:cone-cat}
The category $\Cat$ of small categories has a segment object given by the ordinal $[1]$ with multiplication given by the minimum:
\[\min\colon[1]\times [1]\to [1]\]
For every small category $C$, the left cone $\lcone{C}$ on $C$ is the category $C$ with an adjoint initial object.
In particular, notice that the left cone on the finite ordinal $[n-1]$ is the finite ordinal $[n]$.
Since $\Cat$ is Cartesian closed, the left cone functor comes equipped with a right adjoint, which sends a category $C$ with a distinguished object $c$ to the coslice category $\coslice{C}{c}$. 
\end{example}

\begin{example}
\label{ex:cone-sset}
Since $N[1]=\Delta^{1}$ and $N$ preserves limits and coproducts, the standard $1$-simplex has the structure of a segment object in $\sSet$.
An explicit description of cones and slices in $\sSet$ can be found in \cite[Chapter 3, Section 4]{cis16}.
\end{example}

We now introduce the basics of the theory of enriched categories. 
We fix a  symmetric monoidal category $(\cat{V}, \otimes, \1)$. 
\begin{definition}									%
A $\cat{V}$-\emph{enriched category} or simply $\cat{V}$-\emph{category} $\cat{C}$ consists of
\begin{enumerate}
\item
A class of objects $\ob(\cat{C})$,
\item
for each pair $x,y$ of objects in $\cat{C}$, a \emph{mapping object} $\icat{C}(x,y)$ in $\cat{V}$,
\item
for each object $x\in \cat{C}$ a morphism $\id_x\colon \1\to\icat{C}(x,x)$ in $\cat{V}$,
\item
for each triple $x,y, z$ of objects in $\cat{C}$ a morphism 
\[\circ\colon \icat{C}(y,z)\otimes\icat{C}(x,y)\to \icat{C}(x,z)\]
in $\cat{V}$,
\end{enumerate}
such that the following diagrams commute:

\begin{equation*}
\begin{matrix}
\begin{tikzcd}[ampersand replacement=\&]									%
{\icat{C}(x, y)\otimes\1} 
	\ar[r, "1\otimes \id_{x}"]
	\ar[dr, "\cong"']\&
{\icat{C}(x,y)\otimes\icat{C}(x, x)}
	\ar[d, "\circ"]\\
\&
{\icat{C}(x,y)}
\end{tikzcd}\\
\begin{tikzcd}[ampersand replacement=\&]								%
{\icat{C}(z,w)\otimes\icat{C}(y,z)\otimes\icat{C}(x,y)}
	\ar[r,"{1\otimes \circ}"] \ar[d, "{\circ\otimes 1}"']\&
{\icat{C}(z,w)\otimes\icat{C}(x,z)}
	\ar[d, "{\circ}"]\\
{\icat{C}(y,w)\otimes\icat{C}(x,y)}
	\ar[r, "{\circ}"'] \&
{\icat{C}(x,w)}
\end{tikzcd} \\
\begin{tikzcd}[ampersand replacement=\&]									%
{\icat{C}(y, y)\otimes\icat{C}(x, y)} 
	\ar[d, "\circ"'] \&
{\1\otimes\icat{C}(x,y)}
	\ar[l, "\id_{y}\otimes1"']\ar[dl, "\cong"]\\
{\icat{C}(x,y)}
\&
\end{tikzcd}
\end{matrix}
\end{equation*}
\end{definition}

\begin{definition}											 	%
Let $\cat{C}$ be a $\cat{V}$-category. 
The \emph{opposite} $\cat{V}$\emph{-category} $\cat{C}^{\op}$ to $\cat{C}$ is the enriched category with same class of objects as $\cat{C}$ and with mapping object given by
\[\icat{C}^{\op}(x,y)=\icat{C}(y,x).\]
Composition is inherited by the composition of $\cat{C}$ using the fact that $\cat{V}$ is symmetric monoidal and the identity morphisms are the ones of $\cat{C}$.
\end{definition}

\begin{example}												%
A category with one object enriched over the monoidal category of abelian groups is a ring $R$ with identity.
Indeed, composition defines a monoid structure whose unit is given by the identity morphism $1\colon \ZZ\to R$. 
The enrichment defines the addition of the ring and the additive identity.
Distributivity and left and right absorption laws follow from the fact that the maps involved are group homomorphisms.
\end{example}

\begin{example}													%
\label{ex:sym-mon-self-enriched}
Let $(\cat{V},\otimes,\1)$ be a closed symmetric monoidal category, then $\cat{V}$ is enriched over itself with mapping object given by the internal hom:
\[\icat{V}(x,y)=\ihom_{\cat{V}}(x,y).\]
\end{example}

\begin{definition}												%
A $\cat{V}$-\emph{enriched functor} or $\cat{V}$-\emph{functor} $F\colon \cat{C}\to \cat{C'}$ 
between $\cat{V}$-categories is a map $x\mapsto Fx$ 
between the objects of $\cat{C}$ and $\cat{C}'$ together with a collection of morphisms
\[F_{x,y}\colon \icat{C}(x,y)\to\icat{D}(Fx,Fy)\]
in $\cat{V}$ for every pair of objects $x, y$ in $\cat{C}$ such that the following diagrams commute
\begin{equation*}
\begin{tikzcd}[column sep = small]									%
{\icat{C}(y,z)\otimes\icat{C}(x,y)}
	\ar[r, "\circ"]\ar[d, "{F_{y,z}\otimes F_{x,y}}"'] &
{\icat{C}(x,z)}
	\ar[d, "{F_{x,z}}"]\\
{\icat{D}(Fy,Fz)\otimes\icat{D}(Fx,Fy)}
	\ar[r, "\circ"'] &
{\icat{D}(Fx,Fz)}
\end{tikzcd}
\\
\begin{tikzcd}[column sep = small]									%
\1 \ar[r, "\id_x"]
	\ar[rd, "\id_{Fx}"']&
{\icat{C}(x,x)}
	\ar[d, "F_{x, x}"]\\
	&
{\icat{D}(Fx,F,x)}
\end{tikzcd}
\end{equation*}
for every $x,y$ and $z$ in $\cat{C}$.
\end{definition}

\begin{example}												%
An $\Ab$-functor $M\colon R\to \Ab$ from a ring with unit to the category of abelian groups is a left $R$-module.
Dually, an $\Ab$-functor $N\colon R^{\op}\to \Ab$ from the opposite enriched category is a right $R$-module.
\end{example}

\subsection{}							%
Let $F\colon (\cat{V},\otimes,\1) \to (\cat{V'},\otimes',\1')$ be a lax monoidal functor and let $\cat{C}$ be a $\cat{V}$-category. 
Then one can define a $\cat{V}'$-category $\cat{C}_{\cat{V}'}$ with the same objects as $\cat{C}$ and with mapping objects given by
\[\icat{C_{\cat{V'}}}(x,y)=F(\icat{C}(x,y))\]

\subsection{}											%
In particular, since the functor $\hom_{\cat{V}}(\1,\blank)\colon\cat{V}\to \Set$ is lax monoidal, one can define a category with the same objects as $\cat{C}$ and with morphisms given by
\[\hom_{\cat{C}}(x,y):=\hom_{\cat{V}}\left(\1, \icat{C}(x,y)\right)\]
We call this the \emph{underlying category} to the $\cat{V}$-category $\cat{C}$ and we still denote it by $\cat{C}$.

\subsection{}
Let $\cat{C}$ be a $\cat{V}$-enriched category and let $f\colon \1 \to \icat{C}(y,z)$ be an arrow in the underlying category of $\cat{C}$.
Then, for every $x\in \cat{C}$ there is an induced morphism $f_{*}\colon \icat{C}(x,y)\to \icat{C}(x,z)$ defined as the composition 
\begin{equation*}
\begin{tikzcd}
{f_{*}\colon\icat{C}(x,y)\cong\1\times\icat{C}(x,y)}
\ar[r, "{f\times 1}"]&
{\icat{C}(y,z)\times\icat{C}(x,y)}
\ar[r,"\circ"]&
{\icat{C}(x,z)}
\end{tikzcd}
\end{equation*}
Moreover, this construction defines an unenriched representable functor $\icat{C}(x,\blank)\colon \cat{C}\to \cat{V}$ whose composition with the functor $\hom_{\cat{V}}(\1,\blank)$ coincides with the representable functor $\cat{C}(x,\blank)$ (see \cite[Exercise 3.4.13]{rie14}).
Moreover, if $\cat{V}$ is closed symmetric monoidal category, the functor $\icat{C}(x,\blank)$ promotes to a $\cat{V}$-functor (see \cite[Example 3.5.4]{rie14}).

\begin{definition}												%
A $\cat{V}$-\emph{natural transformation} $\alpha\colon F\to G$ between a pair of $\cat{V}$-functors $F,G\colon \cat{C}\to \cat{D}$ 
is a family of morphisms $\left(\alpha_x\colon \1\to \icat{D}(Fx,Gx)\right)_{x\in \cat{C}}$ in $\cat{V}$ such that the following diagram commutes:
\begin{equation*}
\csquare{\icat{C}(x,y)}{\icat{D}(Fx,Fy)}{\icat{D}(Fx,Gy)}{\icat{D}(Gx,Gy)}{F_{x,y}}{\alpha_{y,*}}{\alpha_{x,*}}{G_{x,y}}
\end{equation*}
for every $x, y$ in $\cat{C}$.
\end{definition}

\begin{definition}												%
A $\cat{V}$-\emph{adjunction} consists of $\cat{V}$-functors
$F\colon\cat{C}\to \cat{D}$ and $G\colon\cat{D}\to\cat{C}$ together with $\cat{V}$-natural isomorphisms
\[\icat{D}(FC,D)\cong \icat{C}(C, GD)\]
in $\cat{V}$.
\end{definition}

\begin{definition}												%
Let $\cat{C}$ be a $\cat{V}$-category.
We say that $\cat{C}$ is \emph{tensored} if for every $v\in \cat{V}$ and every $x\in \cat{C}$, the functor
\[\ihom_{\cat{V}}(v, \icat{C}(x,\blank))\colon \cat{C}\to \cat{V}\]
is a representable $\cat{V}$-functor.
In other words, if there exists an object $v\otimes x$ called the \emph{tensor of $x$ with $v$} and a natural isomorphism
\[\icat{C}(v\otimes x,y)\cong \ihom_{\cat{V}}(v, \icat{C}(x,y)).\]
Dually, we say that $\cat{C}$ is \emph{cotensored}  if for every $v\in \cat{V}$ and every $y\in \cat{C}$, the functor
\[\ihom_{\cat{V}}(v, \icat{C}(\blank,y))\colon \cat{C}^{\op}\to \cat{V}\]
is a representable $\cat{V}$-functor.
In other words, if there exists an object $y^{v}$ called the \emph{cotensor of $y$ with $v$} and a natural isomorphism
\[\icat{C}(x,y^{v})\cong \ihom_{\cat{V}}(v, \icat{C}(x,y)).\]
\end{definition}

\subsection{}
If $\cat{C}$ is a $\cat{V}$-category, which is tensored and cotensored, the functors:
\begin{equation*}
\otimes\colon \cat{V}\times\cat{C}\to\cat{C}\quad (\blank)^{\blank}\colon\cat{C}^{\op}\times\cat{V}\to\cat{C}\quad \ihom\colon\cat{C}^{\op}\times\cat{C}\to\cat{V}
\end{equation*}
define a \emph{two-variables adjunction}, in the sense that
\[\cat{C}(v\otimes x,y)\cong \cat{V}(v,\hom(x,y))\cong \cat{C}(x,y^{v}).\]

\begin{example}
\label{ex:locally-small-tensor-cotensor}
Let $\cat{C}$ be a locally small category and assume that $\cat{C}$ has arbitrary products and coproducts, then the definitions introduced in \ref{subsec:power-copower-set} endow $\cat{C}$ with the structure of a category tensored and cotensored over the category of sets. Notice that a category enriched over the category of sets is exactly a locally small category.
\end{example}

\begin{proposition}
\label{prop:enriched-adjunction}
Let $\cat{C}$ and $\cat{D}$ be tensored and cotensored $\cat{V}$-categories and let 
\begin{equation*}
\Adjoint{\cat{C}}{\cat{D}}{F}{G} 
\end{equation*}
be an adjunction between their underlying categories. 
Then, the following are equivalent:
\begin{enumerate}
\item
There exists a $\cat{V}$-adjunction
\[\icat{D}(Fx,y)\cong \icat{C}(x,Gy)\]
whose underlying adjunction is the given one.
\item
There exists a $\cat{V}$-functor $F\colon \cat{C}\to \cat{D}$ together with a natural isomorphism $F(v\otimes x)\cong v\otimes Fx$ for every $v$ in $\cat{V}$ and every $x$ in $\cat{C}$, whose underlying functor is $F$.
\item
There exists a $\cat{V}$-functor $G\colon \cat{D}\to \cat{C}$ together with a natural isomorphism $G(y^v)\cong (Gy)^v$ for every $v$ in $\cat{V}$ and every $y$ in $\cat{D}$, whose underlying functor is $G$.
\end{enumerate}
\end{proposition}
\begin{proof}
See \cite[Proposition 3.7.10]{rie14}.
\end{proof}

\begin{theorem}
\label{thm:strong-monoidal-enrichment}
Let $\cat{V}$ and $\cat{U}$ be closed symmetric monoidal categories and let 
\begin{equation*}
\Adjoint{\cat{V}}{\cat{U}}{F}{G}
\end{equation*}
be an adjunction, such that the functor $F$ is strong monoidal.
Then any tensored, cotensored and enriched $\cat{U}$-category is canonically a tensored and cotensored enriched $\cat{V}$-category.
\end{theorem}
\begin{proof}
If $\cat{C}$ is a $\cat{U}$-category, we define a $\cat{V}$-enrichment, tensor and cotensor as follows:
\[ v\otimes_{\cat{V}}x=Fv\otimes x, \quad \icat{C}_{\cat{V}}(x,y)=G\left(\icat{C}(x,y)\right),\quad \{v,x\}_{\cat{V}}=x^{Fv}.\]
It is not hard to see that with these definitions, $\cat{C}$ is a tensored and cotensored $\cat{V}$-category (see \cite[Theorem 3.7.11]{rie14} for a complete proof).
\end{proof}

\begin{corollary}
\label{cor:adjunction-v-adjunction}
Let $\cat{V}$ and $\cat{U}$ be closed symmetric monoidal categories and let 
\begin{equation}
\label{eq:v-u-adjunction}
\Adjoint{\cat{V}}{\cat{U}}{F}{G}
\end{equation}
be an adjunction, such that the functor $F$ is strong monoidal.
Then, $\cat{U}$ is a tensored and cotensored enriched $\cat{V}$-category.
Moreover, the adjunction \eqref{eq:v-u-adjunction} is a $\cat{V}$-adjunction.
\end{corollary}
\begin{proof}
This follows immediately from Theorem \ref{thm:strong-monoidal-enrichment}, Example \ref{ex:sym-mon-self-enriched} and Proposition \ref{prop:enriched-adjunction}.
\end{proof}

\begin{example}
Let $\Top$ be the category of numerically generated spaces (or any other Cartesian closed category of topological spaces containing the realisations of the standard simplices).
Then, the adjunction introduced in Example \ref{ex:geometric-realisation} restricts to an adjunction
\begin{equation}
\label{eq:top-sing-adjunction}
\Adjoint{\sSet}{\Top}{\abs{\blank}}{\Sing}
\end{equation}
Moreover, since the functor $\abs{\blank}$ preserves products of standard simplices and $\Top$ is Cartesian closed, the realisation functor preserves all products.
In particular, by Corollary \ref{cor:adjunction-v-adjunction} the category $\Top$ is enriched, tensored and cotensored over the category of simplicial sets and the adjunction \eqref{eq:top-sing-adjunction} is a simplicially enriched adjunction.
Notice that the mapping space $\map(X,Y)$ between two topological spaces $X$ and $Y$ is given by the singular simplicial complex of the internal hom between $X$ and $Y$. Morever, for a simplicial set $K$ and topological spaces $X$ and $Y$ we have
\[K\otimes X=\abs{K}\times X\quad Y^{K}= \ihom(\abs{K},Y)\]
\end{example}

\begin{definition}
\label{def:functor-tensor-product}
Let $\cat{C}$ be a small category, let $\cat{V}$ be a closed monoidal category and let $\cat{E}$ be a $\cat{V}$-category which is tensored and cotensored over $\cat{V}$.
Given functors $F\colon\cat{C}\to\cat{V}$ and $G\colon\cat{C}^{\op}\to \cat{E}$, the \emph{tensor product} of $F$ and $G$ is the object of $\cat{E}$ given by:
\[G\otimes_{\cat{C}}F=\int^{c\in\cat{C}}Gc\otimes Fc\]
\end{definition}

\begin{example}
Let $M\colon R\to\Ab$ and $N\colon R^{\op}\to\Ab$ be a left and a right $R$-module respectively. Then, the tensor product of the underlying unenriched functors to $M$ and $N$ in the sense of Definition \ref{def:functor-tensor-product} is the usual tensor product:
\[N\otimes_{R}M=\int^{R}N\otimes_{\ZZ} M\]
of a right and a left $R$-module.
\end{example}

\begin{example}
Let $F\colon \cat{C}\to\cat{E}$ and $K\colon\cat{C}\to\cat{D}$ be functors.
By Theorem \ref{thm:kan-ext}, the left Kan extension of $F$ along $K$ can be expressed as the tensor product:
\[\Lan_{K}F(d) = \cat{D}(K\blank, d)\otimes_{\cat{C}}F.\]
\end{example}

\begin{example}
In particular, the Yoneda Lemma implies that, for any simplicial set $K$:
\[\abs{K}=\sSet(\Delta^{\bullet},K)\otimes_{\DDelta}\Delta^{\bullet}\cong K_{\bullet}\otimes_{\DDelta}\Delta^{\bullet},\]
where $\abs{K}$ denotes the geometric realisation of $K$ defined in Example \ref{ex:geometric-realisation}.
\end{example}

\begin{proposition}
\label{prop:tensor-product-functors}
Let $F\colon \cat{C}\to \Set$ and $G\colon \cat{C}^{\op}\to \cat{E}$ be functors and assume that $\cat{E}$ is a cocomplete and locally small category.
Then, for every $e\in \cat{E}$, there exists an isomorphism:
\[\cat{E}\left(G\otimes_{\cat{C}}F, e\right)\cong \Set^{\cat{C}}\left(F, \cat{E}(G,e)\right)\]
which is natural in all variables.
\end{proposition}
\begin{proof}
We have the following chain of natural isomorphisms:
\begin{align*}
\cat{E}\left(G\otimes_{\cat{C}}F, e\right) 
	& = \cat{E}\left(\int^{c\in{\cat{C}}}Gc\otimes Fc, e\right)\\
	& \cong \int_{c\in\cat{C}}\cat{E}\left(Gc\otimes Fc, e\right)\\
	& \cong \int_{c\in\cat{C}}\Set\left(Fc,\cat{E}(Gc, e)\right)\\
	& \cong \Set^{\cat{C}}\left(F,\cat{E}(G,e)\right)
\end{align*}
where the last isomorphism follows from Example \ref{ex:nat-trans-end}.
\end{proof}

\section{Factorisation systems}
\label{ch:1sec:3}
We develop in parallel the theory of weak factorisation systems and orthogonal factorisation systems.
Weak factorisation systems encode lifting properties and they are central in abstract homotopy theory.
On the other hand, orthogonal factorisation systems encode unique lifting properties and they play an important role in the interplay between covering spaces and monodromy.
In particular, Corollary \ref{cor:orthogonal-factorisation-theorem} will be used extensively in Sections \ref{ch:3sec:2}, \ref{ch:3sec:3} and \ref{ch:5sec:4}.
We follow mainly \cite{rie14} and \cite{faj08}.

\begin{definition}
Let $i\colon a\to b$ and $p\colon x\to y$ be two morphisms in a category $\cat{C}$. 
A \emph{lifting problem} of $i$ against $p$ is a solid commutative square
\begin{equation}
\label{eq:lifting-problem}
\begin{tikzcd}
a	\ar[r, "u"]
	\ar[d, "i"']&
x	\ar[d, "p"]\\
b	\ar[ur, dashed, "k" description]
	\ar[r, "v"']&
y	
\end{tikzcd}
\end{equation}
A \emph{solution} to the lifting problem is a dotted arrow as $k$ making the diagram commutative.
We say that $i$ has the \emph{(unique) left lifting property} with respect to $p$ or, equivalently, that $p$ has the \emph{(unique) right lifting property} with respect to $i$ and we write $i\olift p$ (resp. $i\perp p$) if any lifting problem of $i$ against $p$ has a (unique) solution.
\end{definition}

\begin{notation}
Let $\cat{C}$ be a category and let $\cat{A}$ and $\cat{B}$ be two classes of morphisms in $\cat{C}$.
We write $\rlift{\cat{A}}$ (\resp $\urlift{\cat{A}}$) for the class of morphisms with the (unique) right lifting property with respect to every morphism in $\cat{A}$ and $\llift{\cat{B}}$ (\resp $\ullift{\cat{B}}$) for the class of morphisms with the (unique) left lifting property with respect to every morphism in $\cat{B}$.
We write $\cat{A}\olift \cat{B}$ (\resp $\cat{A}\oulift \cat{B}$) if $\cat{A}$ is contained in $\llift{\cat{B}}$ (\resp $\cat{A}$ is contained in $\ullift{\cat{B}}$).
\end{notation}

\begin{lemma}
Let $\cat{A}$ be a class of maps in a category $\cat{C}$ and let $\cat{B}$ be a class of maps in a category $\cat{D}$.
Assume we have an adjunction
\begin{equation*}
\Adjoint{\cat{C}}{\cat{D}}{F}{G} 
\end{equation*}
Then 
\[F\cat{A}\olift \cat{B}\quad \text{if and only if}\quad \cat{A}\olift G\cat{B}\]
\end{lemma}

\begin{construction}
\label{cons:push-pull}
Let $\otimes\colon\cat{C}\times\cat{D}\to\cat{E}$ be a functor and assume that $\cat{E}$ has pushouts. Then, given morphisms $i\colon c\to c'$ and $j\colon d\to d'$ we can form the pushout square
\begin{diagram}
c\otimes d	\ar[r, "i\otimes \id_{d}"]
		\ar[d, "\id_{c}\otimes j"']&
c'\otimes d\ar[d]
		\ar[ddr, bend left=25, "\id_{c'}\otimes j"]&
		\\
c\otimes d'\ar[r]
		\ar[drr, bend right=15, "i\otimes \id_{d'}"']&
{\displaystyle c\otimes d'\bigcoprod_{c\otimes d}c'\otimes d}
		\ar[dr, dashed, "i\hat{\otimes}j" description]&
		\\
		&
		&
		c'\otimes d'.
\end{diagram}
We call the induced map
\[i\hat{\otimes}j\colon c\otimes d'\bigcoprod_{c\otimes d}c'\otimes d\to c'\otimes d'\]
The \emph{pushout-product} of $i$ and $j$ with respect to $\otimes$.

Dually if $\ihom\colon \cat{D^{\op}}\times\cat{E}\to\cat{C}$ is a functor, given maps $j\colon d\to d'$ and $p\colon e\to e'$ we can construct the pullback diagram:
\begin{diagram}
\ihom(d',e)	\ar[rrd, bend left=15]
			\ar[rdd, bend right=25]
			\ar[rd, dashed]&
			&
			\\
			&
\ihom(d',e')\times_{\ihom(d,e')}\ihom(d,e)
			\ar[r]
			\ar[d]&
\ihom(d,e)
			\ar[d]\\
			&
\ihom(d',e')	\ar[r]&
\ihom(d,e').
\end{diagram}
We call the induced map
\[\hat{\ihom}(j,p)\colon\ihom(d',e)\to \ihom(d',e')\times_{\ihom(d,e')}\ihom(d,e)\]
the \emph{pullback-power} of $j$ and $p$ with respect to $\ihom$.
\end{construction}

\subsection{}
Let $\cat{C}$ be a category. The \emph{arrow category} $\arr{\cat{C}}$ of $\cat{C}$ is the category of functors from $[1]$ to $\cat{C}$ (see \ref{ex:nerve-fundamental-category}).
In other words, $\arr{\cat{C}}$ is the category whose objects are maps in $\cat{C}$ and morphisms are commutative diagrams between them.

\begin{lemma}
\label{lem:two-variable-lifting}
Let us consider a two variable adjunction
\begin{equation*}
\otimes\colon \cat{C}\times\cat{D}\to\cat{E}\quad \{\blank,\blank\} \colon\cat{C}^{\op}\times\cat{E}\to\cat{D}\quad \ihom\colon\cat{D}^{\op}\times\cat{E}\to\cat{C}
\end{equation*}
and assume that $\cat{C}$ and $\cat{D}$ have pullbacks and $\cat{E}$ has pushouts. 
\begin{enumerate}
\item
\label{item:two-variable-lifting-1}
The pushout-product and pullback-power constructions induce a two variable adjunction
\begin{equation*}
\hat{\otimes}\colon \cat{C}^{[1]}\times\cat{D}^{[1]}\to\cat{E}^{[1]}\quad \hat{\{\blank,\blank\}} \colon\left(\cat{C}^{\op}\right)^{[1]}\times\cat{E}^{[1]}\to\cat{D}^{[1]}\quad \hat{\ihom}\colon\left(\cat{D}^{\op}\right)^{[1]}\times\cat{E}^{[1]}\to\cat{C}^{[1]}
\end{equation*}
\begin{equation*}
\hom_{\arr{\cat{E}}}\left(i\hat{\otimes}j,p\right)\cong \hom_{\arr{\cat{D}}}\left(j, \hat{\{i,p\}}\right)\cong\hom_{\arr{\cat{C}}}\left(i, \hat{\ihom}(j,p)\right)
\end{equation*}
\item
\label{item:two-variable-lifting-2}
If $\cat{L}$, $\cat{M}$ and $\cat{R}$ are classes of maps in $\cat{C}$, $\cat{D}$ and $\cat{E}$ respectively, then
\[\cat{L}\hat{\otimes}\cat{M}\olift\cat{R}\quad \Leftrightarrow\quad \cat{M}\olift \hat{\{\cat{L},\cat{R}\}}\quad \Leftrightarrow \quad \cat{L}\olift\hat{\ihom}(\cat{M},\cat{R}).\]
\end{enumerate}
\end{lemma}

\subsection{}
An object $a$ in a category $\cat{C}$ is said to be a \emph{retract} of an object $x$ in $\cat{C}$ if there exist morphisms $i\colon a\to x$ and $r\colon x\to a$ such that the composition $r\circ i$ equals the identity on $a$.
A morphism $f\colon a\to b$ is a retract of a morphism $g\colon x\to y$ if it is so in the arrow category $\arr{\cat{C}}$.
Given an ordinal $\alpha$ and a functor $F\colon \alpha\to \cat{C}$, the transfinite composite of $F$ is the induced morphism
\[F(0)\to \colim_{\beta<\alpha}F(\beta).\]
A class of morphims $\cat{A}$ in $\cat{C}$ is said to be \emph{closed under transfinite composition} if for every ordinal $\alpha$ and every functor $F\colon \alpha\to \cat{C}$ such that
\[F(\beta)\to F(\beta+1)\]
is in $\cat{A}$, the transfinite composite of $F$ is in $\cat{A}$.

\begin{definition}
A class $\cat{S}$ of morphisms in a category $\cat{C}$ is said to be \emph{weakly saturated} if it is closed under taking retracts, pushouts and transfinite composition.
\end{definition}

\begin{lemma}
\label{lemma:left-saturated-class}
Let $\cat{I}$ be a class of maps in a category $\cat{C}$, then:
\begin{enumerate}
\item
The class $\llift{\cat{I}}$ is weakly saturated.
\item
The class $\ullift{\cat{I}}$ is closed under colimits.
\item
The class $\ullift{\cat{I}}$ has the left cancellation property: if $f$ and $gf$ are in $\ullift{\cat{I}}$, then $g$ is in $\ullift{\cat{I}}$.
\end{enumerate}
\end{lemma}

\begin{definition}
A \emph{weak factorisation system} (\resp an \emph{orthogonal factorisation system}) on a category $\cat{C}$ is a pair $(\cat{L},\cat{R})$ of classes of morphisms such that
\begin{enumerate}
\item
Every map in $\cat{C}$ can be factored as a map in $\cat{L}$ followed by a map in $\cat{R}$.
\item
The equalities $\cat{L}=\llift{\cat{R}}$ and $\rlift{\cat{L}}=\cat{R}$ (resp. $\cat{L}=\ullift{\cat{R}}$ and $\urlift{\cat{L}}=\cat{R}$) hold.
\end{enumerate}
\end{definition}

\begin{definition}
Let $\cat{C}$ be a category, let $c$ be an object of $\cat{C}$ and let $\kappa$ be a regular cardinal.
\begin{enumerate}
\item
We say that $c$ is $\kappa$-\emph{presentable} if the functor 
\[\hom_{\cat{C}}(c,\blank)\colon\cat{C}\to \Set\]
commutes with $\kappa$-filtered colimits.
\item
We say that $c$ is \emph{presentable} if it is $\kappa$-presentable for some $\kappa$.
\item
We say that $c$ is \emph{finitely presentable} if it is $\omega$-presentable.
\end{enumerate}
\end{definition}

\begin{definition}
Let $\kappa$ be a regular cardinal. A category $\cat{C}$ is said to be $\kappa$-\emph{locally presentable} if it is locally small, cocomplete and if there exists a set $\cat{P}$ of $\kappa$-presentable objects of $\cat{C}$ such that every object of $\cat{C}$ is a $\kappa$-filtered colimit of objects in $\cat{P}$. 
We say that $\cat{C}$ is \emph{locally presentable} if it is $\kappa$-locally presentable for some regular cardinal $\kappa$.
We say that $\cat{C}$ is \emph{locally finitely presentable} if it is $\omega$-locally presentable.
\end{definition}

\begin{theorem}[Small object argument]
\label{thm:wfs}
Let $\cat{C}$ be a locally small cocomplete category, let $\cat{I}$ be a set of objects in $\cat{C}$ and assume that the domain of every morphism of $\cat{I}$ is presentable. Then the pair $(\llift{(\rlift{\cat{I}})},\rlift{\cat{I}})$ is a weak factorisation system.
\end{theorem}
\begin{proof}
See \cite[Proposition 2.1.9]{cis16}
\end{proof}

\begin{corollary}
\label{cor:wfs-locpres}
Let $\cat{C}$ be a locally presentable category and let $\cat{I}$ be a set of morphisms in $\cat{C}$. Then the pair $(\llift{(\rlift{\cat{I}})},\rlift{\cat{I}})$ is a weak factorisation system.
\end{corollary}
\begin{proof}
Follows from Theorem \ref{thm:wfs}
\end{proof}

\begin{corollary}
\label{cor:orthogonal-factorisation-theorem}
Let $\cat{C}$ be a locally presentable category and let $\cat{I}$ be a set of morphisms in $\cat{C}$. Then the pair $(\ullift{(\urlift{\cat{I}})},\urlift{\cat{I}})$ is an orthogonal factorisation system.
\end{corollary}
\begin{proof}
See \cite[Theorem 2.2]{faj08}
\end{proof}

\section{Topological constructs}
\label{ch:1sec:4}
A topological functor $U\colon \cat{S}\to \cat{A}$ is a functor that generalises the behaviour of the forgetful functor from topological spaces to sets.
In particular, a topological functor $U\colon \cat{S}\to \cat{A}$ is a faithful functor such that the class of ``objects that forget to a given object $A$'' can be viewed as a complete preorder of ``structures'' on $A$, so that we can form the initial and final structure on $A$, with respect to a given class of morphisms with source or target $A$, respectively. 
Given a topological functor $U\colon \cat{C}\to \Set$ and a full subcategory $\cat{I}$ of $\cat{C}$, we recall the definition of the final closure $\cat{C}_{\cat{I}}$ of $\cat{I}$ in $\cat{C}$.
In particular, we give sufficient conditions for $\cat{C}_{\cat{I}}$ to be locally presentable (Theorem \ref{thm:C_I-locally-presentable}) and Cartesian closed (Theorem \ref{thm:final-closure-cartesian-closed}).

\begin{definition}
\label{def:initial-lift}
Let $U\colon \cat{S}\to \cat{A}$ be a faithful functor, let $a$ be an object of $\cat{A}$ and let $\cat{I}$ be a class of objects in $\cat{S}$.
An \emph{initial lift} for a family of maps $(f_{i}\colon a\to Ui)_{i\in \cat{I}}$ is an object $s$ in $\cat{S}$ together with a family of morphisms $(m_i\colon s \to i)_{i\in \cat{\cat{I}}}$ satisfying the following axioms:
\begin{enumerate}
\item
$Us=a$ and $f_i=Um_i$, for every $i\in \cat{I}$.
\item
If $(m'_i\colon s'\to i)_{i\in \cat{I}}$ is a family of morphisms and $f\colon Us'\to a$ is a map such that the diagram:
\begin{diagram}
Us'\ar[rd, "Um'_{i}"']\ar[r, "f"]&
a\ar[d, "f_{i}"]\\
&
Ui
\end{diagram}
commutes for every object $i\in \cat{I}$, there exists a unique morphism $m\colon s'\to s$ such that $Um=f$. In particular, since $U$ is a faithful functor, the diagram:
\begin{diagram}
s'\ar[rd, "m'_{i}"']\ar[r, "m"]&
s\ar[d, "m_{i}"]\\
&
i
\end{diagram}
commutes for every $i\in\cat{I}$.
\end{enumerate}
\end{definition}

\begin{definition}
\label{def:topological-functor}
Let $U\colon \cat{S}\to \cat{A}$ be a faithful functor. We say that $U$ is a \emph{topological functor} if for every object $a$ in $\cat{A}$ and for every class of objects $\cat{I}$ in $\cat{S}$, every family of maps $(f_i\colon a\to Ui)_{i\in \cat{I}}$ has a unique initial lift. A \emph{topological construct} is a topological functor $U\colon\cat{C}\to \Set$ to the category of sets.
\end{definition}

\begin{notation}
In the presence of a topological functor $U\colon \cat{S}\to\cat{A}$, we think of $\cat{S}$ as a category of structured objects over the objects of $\cat{A}$.
Following the terminology already used in Definition \ref{def:initial-lift}, we refer to ``maps'' for the morphisms $f\colon a\to a'$ between the less structured objects on the base and we use the term ``morphism'' exclusively for the morphisms $m\colon s\to s'$ between structured objects.
\end{notation}

\begin{remark}
Dually, a faithful functor $U\colon \cat{S}\to \cat{A}$ has \emph{final lift} for a family of objects $\cat{I}$ in $\cat{S}$ and maps $(f_i\colon Ui\to a)_{i\in \cat{I}}$ if the opposite functor $U^\op\colon \cat{S}^\op\to \cat{A}^\op$ has initial lift for the corresponding family. Functors that have all final lifts are called \emph{cotopological}, but this definition is superfluous since one can show that every topological functor is also cotopological (see \cite[Proposition 7.3.6]{bor942}).
\end{remark}

\subsection{}
\label{subsec:preorder-fibre}
If $U\colon \cat{S}\to \cat{A}$ is a faithful functor, then the fibre of $U$ at an object $a\in \cat{A}$ is a preorder, where $s\le s'$ if there exists a morphism $m\colon s\to s'$ such that $Um=\id_{a}$.
If $U$ is a topological functor, the preorder on each fibre is a complete lattice, hence $U$ is \emph{amnestic}. 
Moreover $U$ is \emph{uniquely transportable} meaning that, for every object $s\in \cat{S}$ and every isomorphism $f\colon Us\to a$ in $\cat{A}$, there exists a unique object $s'\in \cat{S}$ and a unique isomorphism $m\colon s\to s'$ such that $Us'=a$ and $Um=f$.

\begin{example}				%
The category $\Topa$ of topological spaces is a topological construct. 
Indeed, let $S$ be a set, let $(X_i)_{i\in \cat{I}}$ be a class of topological spaces and $(f_i\colon S\to UX_{i})_{i\in I}$ be a cone in $\Set$, where $U\colon \Topa\to \Set$ is the forgetful functor.
Then, an initial lift for this data is given by the set $S$, endowed with the \emph{initial topology} with respect to the maps $f_i$. 
The initial topology on $S$ is the coarsest topology on $S$ containing the subsets $f_{i^{-1}}(U_{i})$, for every $i\in I$ and every $U_{i}$ open in $X_{i}$.
\end{example}

\begin{proposition}
\label{prop:top-functor-adjoints}
Let $U\colon \cat{S}\to\cat{A}$ be a topological functor.
Then, $U$ has a left and a right adjoint which are fully faithful.
\end{proposition}
\begin{proof}
See \cite[Proposition 7.3.7]{bor942}
\end{proof}

\begin{proposition}
\label{prop:coreflective-subcat-top-cat}.
Let $U\colon \cat{S}\to \cat{A}$ be a topological functor and let $\iota\colon\cat{C}\to \cat{S}$ be a coreflective subcategory of $\cat{S}$ with coreflector $\phi\colon \cat{S}\to \cat{C}$.
Assume that the counit $\iota\phi X\to X$ forgets to the identity morphism on $UX$ in $\cat{A}$ for every $X\in \cat{S}$.
Then $\cat{C}\to \cat{A}$ is a topological functor.
\end{proposition}
\begin{proof}
See \cite[Propositions 21.31, 21.32]{AHS04}.
\end{proof}

\subsection{}
Let $U\colon \cat{S}\to \cat{A}$ be a topological functor and let $a$ be an object in $\cat{A}$.
The image of $a$ via the left adjoint of $U$ is the minimum element of the fibre of $U$ at $a$ with respect to the preorder relation introduced in \ref{subsec:preorder-fibre}.
We call it the \emph{discrete object} on $a$ and denote it by $\disc{a}$.
Dually the image of $a$ via the right adjoint of $U$ is the greatest element of the fibre of $U$.
We call it the \emph{codiscrete object} on $a$ and denote it by $\ind{a}$.
Depending on the context, we will also use the terminology \emph{chaotic object} or \emph{indiscrete object} on $a$ for the codiscrete object on $a$.

\begin{definition}
A functor $U\colon \cat{S}\to \cat{A}$ is called \emph{fibre-small} if for every object $a\in \cat{A}$, the fibre of $U$ at $a$ is an essentially small category.
\end{definition}

\subsection{}
If $U\colon \cat{S}\to \cat{A}$ is a topological functor and $\cat{A}$ has a terminal object $\terminal$, then $\cat{S}$ also has a terminal object which is given by the codiscrete object $\ind{\terminal}$ on $\term$.
Indeed, for every object $s$ in $\cat{S}$, the following natural isomorphisms hold:
\[\terminal\cong\hom_{\cat{S}}(Us, \terminal)\cong \hom_{\cat{C}}(s, \ind{\terminal}).\]
Moreover, if $\cat{C}$ is a topological construct with topological functor $U\colon \cat{C}\to \Set$, the discrete object $\disc{\terminal}$ on the terminal object represents the forgetful functor, since we have:
\[Ux\cong \hom_{\Set}(\terminal, Ux)\cong \hom_{\cat{C}}(\disc{\terminal}, x).\]

\subsection{}
A topological functor $U\colon \cat{S}\to \cat{A}$ has \emph{discrete terminal object} if the discrete and codiscrete objects on the terminal object of $\cat{A}$ coincide.
This is equivalent to saying that the fibre of $U$ at the terminal object of $\cat{A}$ is the terminal category.
A topological construct $\cat{C}$ is said to be \emph{well fibered} if it is fibre-small and it has discrete terminal object. 
Notice that, if $\cat{C}$ is a topological construct with discrete terminal object, given any two objects $C$ and $X$ in $\cat{C}$, there is a bijection between the underlying set of $X$ and the set of morphisms $m\colon C\to X$ in $\cat{C}$ which factor through the terminal object.

\begin{definition}
\label{def:I-generated}
Let $\cat{C}$ be a topological construct with topological functor $U\colon \cat{C}\to \Set$ and let $\cat{I}$ be a full subcategory of $\cat{C}$. 
An object $X$ in $\cat{C}$ is said to be $\cat{I}$-\emph{generated}, if the class $(\alpha\colon C\to X)_{\alpha\in \overcat{\cat{I}}{X}}$ is a final lift for the family of maps $(U\alpha\colon UC\to UX)_{\alpha\in \overcat{\cat{I}}{X}}$. 
The \emph{final closure} of $\cat{I}$ in $\cat{C}$, denoted by $\cat{C}_{\cat{I}}$, is the full subcategory of $\cat{C}$ spanned by the $\cat{I}$-generated objects.
\end{definition}

\begin{proposition}
\label{prop:C_I-coreflective}
Let $U\colon\cat{C}\to \Set$ be a topological construct with discrete terminal object, let $k\colon\cat{I}\to \cat{C}$ be a full subcategory of $\cat{C}$ and let 
\[\eta\colon \Lan_{k}k\to \id_{\cat{C}}\]
be the canonical natural transformation from the left Kan extension of $k$ along itself to the identity on $\cat{C}$.
\begin{enumerate}
\item
For every $X$ in $\cat{C}$, the map
\[U\eta_X\colon U\left(\Lan_{k}k(X)\right)\to UX\]
is a bijection.
\item
An object $X$ in $\cat{C}$ is $\cat{I}$-generated if and only if $\eta_X$ is an isomorphism.
\item
An object $X$ in $\cat{C}$ is $\cat{I}$-generated if and only if it is a colimit of a diagram whose objects belong to $\cat{I}$.
\item
The category $\cat{C}_\cat{I}$ is coreflective in $\cat{C}$ and it contains $\cat{I}$ as a dense subcategory.
\end{enumerate}
\end{proposition}
\begin{proof}
See \cite[Proposition 3.5]{faj08}.
\end{proof}

\begin{lemma}
\label{lem:coreflection-I-generated}
Let $\cat{C}$ be a well fibered topological construct, let $k\colon \cat{I}\to \cat{C}$ be a full subcategory and let $X$ be an object of $\cat{C}$. 
Assume that $h\colon X'\to X$ is a morphism in the fibre of $U$ at $UX$ with $X'$ an $\cat{I}$-generated object. Then, $X'$ is equal to $\Lan_{k}k(X)$ if and only if for every object $C \in\cat{I}$ the natural map
\[\hom_{\cat{C}_{\cat{I}}}(C, X')\to \hom_{\cat{C}}(C,X)\]
induced by $h$ is a bijection.
\end{lemma}
\begin{proof}
If $A$ is an $\cat{I}$-generated object, we have that 
\[\Lan_{k}k(A)=\colim_{C\to A}C\cong A.\]
Therefore
\begin{align*}
\hom_{\cat{C}_\cat{I}}(A,X')
& = \hom_{\cat{C}_{\cat{I}}}(\colim_{C\to A}C,X')\\
& = \lim_{C\to A}\hom_{\cat{C}_{\cat{I}}}(C,X')\\
& = \lim_{C\to A}\hom_{\cat{C}}(C,X)\\
& = \hom_{\cat{C}}(\colim_{C\to A}C,X)\\
& = \hom_{\cat{C}}(A,X)
\end{align*}
\end{proof}

\begin{theorem}
\label{thm:C_I-locally-presentable}
Let $\cat{C}$ be a well-fibered topological construct and let $I$ be a small full subcategory of $\cat{C}$. Then, the category $\cat{C}_I$ is locally presentable.
\end{theorem}
\begin{proof}
See \cite[Theorem 3.6]{faj08}.
\end{proof}

\begin{example}
\label{ex:numerically-generated}
Let $\cat{I}$ be a full subcategory of $\Topa$. A topological space $X$ is $\cat{I}$-generated if and only if it has the final topology with respect to the class $(f\colon C\to X)_{f\in \overcat{\cat{I}}{X}}$of all continuous maps from objects of $\cat{I}$ to $X$.
In particular, if $\cat{I}$ is the full subcategory of $\Topa$ spanned by the geometric standard simplices, we denote by $\Top$ the category of $\cat{I}$-generated spaces.
This is known in the literature as the category of $\DDelta$-\emph{generated spaces} or \emph{numerically generated spaces}. 
Notice that, thanks to \ref{thm:C_I-locally-presentable}, $\Top$ is locally presentable.
\end{example}

\begin{definition}
Let $\cat{C}$ be a topological construct and assume that $\cat{C}$ has finite products. A class $\cat{I}$ of objects of $\cat{C}$ is said to be \emph{productive} if every object of $\cat{I}$ is exponentiable in $\cat{C}$ and binary products of elements of $\cat{I}$ are $\cat{I}$-generated.
\end{definition}

\begin{theorem}
\label{thm:final-closure-cartesian-closed}
Let $\cat{C}$ be a well fibered topological construct and let $\cat{I}$ be a productive class of objects of $\cat{C}$ containing a non empty object.
Then, the final closure $\cat{C}_{\cat{I}}$ of $\cat{I}$ in $\cat{C}$ is a Cartesian closed category.
\end{theorem}
\begin{proof}
See \cite[Section 3, Theorem 4]{gou14}.
\end{proof}

\subsection{}
\label{subsec:corecompact-exponentiable}
Recall that, given a topological space $X$ and two open subsets $U$ and $V$ in $X$, we say that $U$ is \emph{way below} $V$ and write $U\Subset V $ if every open cover of $V$ contains a finite subcover of $U$.
A topological space $X$ is said to be \emph{core-compact} if for every point $x\in X$ and every open neighbourhood $V$ of $x$, there is an open neighbourhood $U$ of $x$ such that $U\Subset V$.

\subsection{}
\label{subsec:top-exponential}
If $X$ is a core-compact topological space and $Y$ is any topological space, given open subsets $U\subset X$ and $V\subset Y$, we denote by $[U\Subset V]$ the set of continuous maps $f\colon X\to Y$ such that $U\Subset f^{-1}(V)$.
Then, the topology on $\Topa(X,Y)$ generated by $[U\Subset V]$ for all $U$ and $V$ as above is called the \emph{core-open topology} and one can take the exponential $Y^X$ to be the set of continuous maps from $X$ to $Y$ endowed with the core-open topology. 
Even more, we have the following result:

\begin{lemma}
\label{lem:corecompact-exponentiable}
A topological space is exponentiable if and only if it is a core-compact topological space.
\end{lemma}
\begin{proof}
See \cite[Theorem 4.3]{EH01}.
\end{proof}

\begin{example}
The category $\Top$ of $\Delta$-generated topological spaces is Cartesian closed, by Theorem \ref{thm:final-closure-cartesian-closed} since $\abs{\Delta^n}$ is a core compact topological space for every $n$ and the product $\abs{\Delta^n}\times\abs{\Delta^m}$ is isomorphic to the realisation $\abs{\Delta^n\times\Delta^m}$ (see \cite[Chapter III, 3.4]{GZ67}).
\end{example}
\chapter{Homotopical Algebra\label{ch:2}}
\section{Categories of fibrant object}
\label{ch:2sec:1}
A category with weak equivalences is a category $\cat{C}$ equipped with a class $\cat{W}$ of arrows in $\cat{C}$ called weak equivalences.
The localisation of $\cat{C}$ at $\cat{W}$ is called the homotopy category of $\cat{C}$.
Working with localisations is a priori a difficult task.
When $\cat{C}$ is equipped with the structure of a category of fibrant objects, one can define a meaningful notion of homotopy, that allows for a more manageable description of the homotopy category of $\cat{C}$ (see Theorem \ref{thm:homotopy-category-fibrant-objects}).
The results of this section are mainly from \cite{bro73}.

\subsection{}
\label{subsec:category-weak-equivalences}
A category with weak equivalences is a category $\cat{C}$ equipped with a class of morphisms $\cat{W}$ containing  the identity morphisms and satisfying the 2-out-of-3 property that given a pair of composable morphisms $f$ and $g$, if any two between $f$, $g$ and $gf$ are in $\cat{W}$, then so is the third.

\subsection{}
\label{subsec:homotopical-category}
A category $\cat{C}$ with weak equivalences $\cat{W}$ is said to be a \emph{homotopical category} if $\cat{W}$ satisfies the 2-out-of-6 property: given morphisms
\begin{diagram}
x\ar[r,"f"]&
y\ar[r, "g"]&
z\ar[r, "h"]&
w
\end{diagram}
if $hg$ and $gf$ are weak equivalences, then so are $f$, $g$, $h$ and $hgf$.

\subsection{} 
Recall that, given a category $\cat{C}$ and a distinguished class $\cat{W}$ of morphisms in $\cat{C}$, called weak equivalences, the \emph{homotopy category}  $\Ho(\cat{C})$ relative to $\cat{W}$ is the \emph{localisation} of $\cat{C}$ at the class $\cat{W}$. 
In other words, it is a functor:
\[\gamma\colon \cat{C}\to \Ho(\cat{C})\]
that sends the morphisms in $\cat{W}$ to invertible morphisms and which is universal among such functors.
One can show that a localisation of $\cat{C}$ by any class $\cat{W}$ always exists (see for example \cite[Proposition 5.2.2]{bor941}), provided that we are willing to pass to a higher universe.

\subsection{}
\label{subsec:cocylinder}
Let $\cat{C}$ be a category with weak equivalences and assume that $\cat{C}$ has finite products.
Let $\cat{F}$ be a class of maps in $\cat{C}$ that we call \emph{fibrations}.
A \emph{cocylinder} or \emph{path space} $(x',s, d^0, d^1)$ for an object $x$ in $\cat{C}$ is a factorisation
\begin{diagram}
x\ar[r, "s"] &
x' \ar[r, "{(d^0,\,d^1)}"] &
x\times x
\end{diagram}
of the diagonal $(1_x,1_x)\colon x\to x\times x$ , into a weak equivalence $s$ followed by a fibration $(d^0,\,d^1)$.

\begin{definition}
\label{def:category-fibrant-objects}
A \emph{category of fibrant objects} is a category with weak equivalences $(\cat{C}, \cat{W})$ together with a class $\cat{F}$ of fibrations, satisfying the following axioms:
\begin{enumerate}
	\item \label{item:category-fibrant-objects-1} The category $\cat{C}$ has finite products,
	\item \label{item:category-fibrant-objects-2} The class of fibrations is stable under arbitrary base change,
	\item \label{item:category-fibrant-objects-3} The class $\cat{W}\cap\cat{F}$ of \emph{acyclic fibrations} is stable under arbitrary base change,
	\item \label{item:category-fibrant-objects-4} Every object has a path space object,
	\item \label{item:category-fibrant-objects-5} For every object $x\in \cat{C}$ the morphism $x\to \terminal$ to the terminal object is a fibration.
\end{enumerate}
\end{definition}

\begin{lemma}
\label{lem:factorisation-lemma}
Let $\cat{C}$ be a category of fibrant objects and let $f$ be a morphism in $\cat{C}$. Then, there exists a factorisation $f=pi$ of $f$ such that $p$ is a fibration and $i$ is a section to an acyclic fibration. In particular, $i$ is a weak equivalence.
\end{lemma}
\begin{proof}
See \cite[Factorization Lemma]{bro73}.
\end{proof}

\begin{corollary}[Ken Brown's Lemma]
\label{cor:ken-brown-lemma}
Let $\cat{C}$ be a category of fibrant objects, let $\cat{D}$ be a category with weak equivalences and let $F\colon \cat{C}\to \cat{D}$ be a functor. Assume that for every acyclic fibration $p\in \cat{C}$, the map $F(p)$ is a weak equivalence in $\cat{D}$. 
Then, for every weak equivalence $w$ in $\cat{C}$, the map $F(w)$ is a weak equivalence in $\cat{D}$.
\end{corollary}
\begin{proof}
Let $w\colon c\to c'$ be a weak equivalence, then by  Lemma \ref{lem:factorisation-lemma} we have a diagram
\begin{equation}
\begin{tikzcd}
c 	\ar[r, "i" description]
	\ar[dr, "w"']&
x	\ar[d, "p"]
	\ar[l, bend right, "r"']\\
&
c'
\end{tikzcd}
\end{equation}
where $r$ is an acyclic fibration such that $ri=1_{c}$ and the map $p$ is a fibration.
In particular, since both $w$ and $i$ are weak equivalences, $p$ is a weak equivalence as well by the 2-out-of-3 property.
Hence, looking at the diagram:
\begin{equation}
\begin{tikzcd}
Fc 	\ar[r, "Fi" description]
	\ar[dr, "Fw"']&
Fx	\ar[d, "Fp"]
	\ar[l, bend right, "Fr"']\\
&
Fc'
\end{tikzcd}
\end{equation}
we have that $Fp$ and $Fr$ are weak equivalences, since $p$ and $r$ are acyclic fibrations.
In particular, $Fi$ is a weak equivalence as well by the 2-out-of-3 applied to $1_{Fc}$ and $Fr$.
To conclude, $Fw$ is a weak equivalence as well by the 2-out-of-3 applied to $Fi$ and $Fp$.
\end{proof}

\subsection{}
Two morphisms $f,g\colon a\to x$ in a category of fibrant objects $\cat{C}$ are said to be \emph{(right) homotopic} if there exists a cocylinder $(x', s, d^0,d^1)$ of $x$ and a morphism $h\colon a\to x$ such that
$d^0h=f$ and $d^1h=g$.
One can show that homotopy defines an equivalence relation on the set $\hom_{\cat{C}}(a,x)$ of maps from $a$ to $x$.

\begin{proposition}
\label{prop:cat-fibrant-objects}
Let $f,g\colon a\to x$ be maps in a category of fibrant objects $\cat{C}$ and assume that $f$ is homotopic to $g$.
\begin{itemize}
\item
If $u\colon b\to a$ is a morphism in $\cat{C}$, then $fu$ is homotopic to $gu$,
\item
If $v\colon x\to y$ is a morphism in $\cat{C}$, and $y'$ is a path space for $y$, then there exists an acyclic fibration $p\colon a'\to a$ such that $vfp$ is homotopic to $vgp$ through a homotopy $h\colon a'\to y'$.
\end{itemize}
\end{proposition}

\subsection{}
Let $\cat{C}$ be a category of fibrant objects and let $a$ and $b$ be objects in $\cat{C}$. 
We define an equivalence on $\cat{C}(a,b)$ by setting $f\sim g$ if there exists a weak equivalence $t\colon a'\to a$ such that $ft$ is homotopic to $gt$.
By Proposition \ref{prop:cat-fibrant-objects}, this equivalence relation is compatible with composition of morphisms in $\cat{C}$.
We denote by $[a,b]$ the quotient of $\cat{C}(a,b)$ with respect to $\sim$ and by $\pi\cat{C}$ the category with the same objects as $\cat{C}$ and with:
\[\pi\cat{C}(a,b)=[a,b]\]
of $\cat{C}(a,b)$ by the equivalence relation $\sim$.
Notice that the images of the weak equivalences in $\cat{C}$ endow $\pi\cat{C}$ with the structure of a category with weak equivalences.

\begin{definition}
A category $\cat{C}$ with weak equivalences $\cat{W}$ admits a \emph{calculus of right fractions} if
\begin{enumerate}
\item
Any solid span:
\begin{equation*}
\begin{tikzcd}
a'\ar[r,dashed]\ar[d,dashed]&
b'\ar[d, "t"]\\
a\ar[r]&
b
\end{tikzcd}
\end{equation*}
where $t$ is a weak equivalence, can be completed to a commuting square.
\item
If $f,g\colon a\to b$ are morphisms such that there exists a weak equivalence $t\colon b\to b'$ such that $tf=tg$, then there exists  a weak equivalence $s\colon a'\to a$ such that $fs=gs$.
\end{enumerate}
\end{definition}

\begin{proposition}
Let $\cat{C}$ be a category of fibrant objects with weak equivalences $\cat{W}$. Then, the category $\pi\cat{C}$ admits a right calculus of fractions with respect to the image of $\cat{W}$.
\end{proposition}

\begin{theorem}
\label{thm:homotopy-category-fibrant-objects}
Let $\cat{C}$ be a category of fibrant objects and let $a$ and $b$ be two objects of $\cat{C}$. Then there is a natural isomorphism:
\[\Ho(\cat{C})(a,b)\cong [a,b].\]
\end{theorem}

\section{Model categories}
\label{ch:2sec:2}

Model categories were introduced by Quillen in \cite{qui67}.
A model category is a category $\cat{M}$ with weak equivalences, together with two classes of auxiliary morphisms, called fibrations and cofibrations.
As for categories of fibrant objects, the extra structure allows for a more manageable computation of the morphisms in the homotopy category of $\cat{M}$ (see Theorem \ref{thm:cofibrant-fibrant}).
The results in this section can be found in \cite{hov07}, \cite{hir09}  or \cite{rie14}.
\begin{definition}
\label{def:model-category}
A \emph{model category} is a category $\cat{M}$ endowed with three classes of morphisms $\cat{W}$, $\cat{F}$ and $\cat{C}$, satisfying the following axioms:
\begin{enumerate}
	\item
	The category $\cat{M}$ is complete and cocomplete.
	\item
	The class $\cat{W}$ has the 2-out-of-3 property.
	\item
	Both pairs $(\cat{C},\cat{F}\cap \cat{W})$ and $(\cat{C}\cap\cat{W},\cat{F})$ are weak factorisation systems.
\end{enumerate}
\end{definition}

\subsection{} 
If $\cat{M}$ is a model category, a morphism in $\cat{W}$ (\cat{C}, or $\cat{F}$) is called a \emph{weak equivalence} (\emph{cofibration} or \emph{fibration}, respectively).
A morphism in $\cat{M}$ that is both a weak equivalence and a cofibration is called a \emph{trivial cofibration} and dually a morphism that is both a weak equivalence and a fibration is called a \emph{trivial fibration}.
An object $a$ in $\cat{M}$ such that the unique map $\initial\to a$ from the initial object is a cofibration is called \emph{cofibrant} and an object $x$ in $\cat{M}$ such that the unique map $x\to \terminal$ to the terminal object is a fibration is called \emph{fibrant}.

\begin{example}
Every bicomplete category $\cat{M}$ has a \emph{trivial model category structure} given by declaring the weak equivalences to be the isomorphisms in $\cat{M}$ and the other two classes to be equal to the class of all morphisms in $\cat{M}$.
Moreover, any permutation of these three classes gives rise to a model structure on $\cat{M}$.
\end{example}

\begin{example}
\label{ex:opposite-category-model}
If $\cat{M}$ has a model category structure, the opposite category $\cat{M}^\op$ inherits a model category structure where the weak equivalences remain the same and the classes of cofibrations and fibrations are interchanged.
\end{example}

\begin{example}
If $\cat{M}$ is a model category and $x$ is an object of $\cat{M}$, the slice category $\overcat{\cat{M}}{x}$ has a natural model structure. Weak equivalences, cofibrations and fibrations are preserved and reflected by the forgetful functor $\overcat{\cat{M}}{x}\to \cat{M}$.
Example \ref{ex:opposite-category-model} implies that also the dual statement is true: the coslice category $\undercat{\cat{M}}{x}$ has a natural model category structure.
\end{example}

\subsection{} 
Let $\cat{M}$ be a model category. Recall that a \emph{cylinder} of an object $a$ in $\cat{M}$ is a factorisation
\begin{diagram}
a\coprod a\ar[r, "{(\partial_0,\partial_1)}"] &
a' \ar[r, "\sigma"] &
a
\end{diagram}
of the codiagonal $(1_a,1_a)\colon a\coprod a\to a$ into a cofibration $(\partial_0,\partial_1)$ followed by a weak equivalence $\sigma$.
Dually, recall that a \emph{cocylinder} or \emph{path object} of an object $x$ is a factorisation
\begin{diagram}
x\ar[r, "s"] &
x' \ar[r, "{(d^0,d^1)}"] &
x\times x
\end{diagram}
of the diagonal $(1_x,1_x)\colon x\to x\times x$ , where $s$ is a weak equivalence and $(d^0,d^1)$ is a fibration (see also \ref{subsec:cocylinder}).

\begin{remark}
\label{rmk:model-cat-fibrant-cat}
Due to the third axiom in Definition \ref{def:model-category}, every object $a$ in a model category $\cat{M}$ has a cylinder object, obtained by applying the weak factorisation system $(\cat{C},\cat{W}\cap \cat{F})$ to the codiagonal $a\coprod a\to a$ (dually, every object $x$ in $\cat{M}$ has a cocylinder object).
In particular, if we define by $\cat{M}_{c}$ the full subcategory of $\cat{M}$ spanned by the cofibrant objects and by $\cat{M}_{f}$ the full subcategory of $\cat{M}$ spanned by the fibrant objects, it's  easy to see that $\cat{M}_{f}$ and $\cat{M}_{c}^{\op}$ have both the structure of a category of fibrant objects in the sense of Definition \ref{def:category-fibrant-objects} (see \cite[Example 1]{bro73}).
\end{remark}

\subsection{} 
Let $f_0, f_1\colon a\to x$ be morphisms in $\cat{M}$. A \emph{left homotopy} from $f_0$ to $f_1$ is a map $h\colon a'\to x$ from a cylinder of $a$ to $x$, such that $h\partial_i=f_i$ for $i=0,1$.
Dually, a \emph{right homotopy} from $f_0$ to $f_1$ is a morphism $k\colon a\to x'$ from $a$ to a cocylinder of $x$ such that $d^ik=f_i$ for $i=0,1$.

\subsection{}
Let $f_0,f_1\colon a\to x$ be morphisms in a model category $\cat{M}$ and assume that $a$ is cofibrant and $x$ is fibrant.
Then, one can show that the existence of a left homotopy from $f_0$ to $f_1$ is equivalent to the existence of a right homotopy.
Moreover, the existence of such homotopies is independent of the given cylinder or cocylinder.
This can be used to show that the relation $\sim$ on the set of morphisms from $a$ to $x$ given by $f_0\sim f_1$ whenever there exists a left (or right) homotopy from $f_0$ to $f_1$ is indeed an equivalence relation (see \cite[Corollary 1.2.6]{hov07}).

\begin{notation}
If $a$ is a cofibrant object and $x$ is a fibrant object in a model category $\cat{M}$, we denote by
	\[[a,x]=\hom_{\cat{M}}(a,x)/\sim\]
the quotient of the set of morphisms from $a$ to $x$ by the relation of left (or right) homotopy defined above.
\end{notation}

\begin{theorem}
\label{thm:cofibrant-fibrant}
For any cofibrant object $a$ and any fibrant object $x$, there exists a natural bijection:
\[\hom_{\Ho(\cat{M})}(a,x)\simeq[a,x].\]
\end{theorem}
\begin{proof}
See \cite[Corollary 1.2.9]{hov07}.
\end{proof}

\begin{corollary}
Let $\cat{M}_{cf}$ be the category of cofibrant and fibrant objects in $\cat{M}$ and let $\pi(\cat{M}_{cf})$ be the category with the same objects as $\cat{M}_{cf}$ and morphisms from $a$ to $x$ given by the set $[a,x]$ of homotopy classes of maps from $a$ to $x$.
Then, the inclusion functor $\cat{M}_{cf}\subset\cat{M}$ induces a  canonical equivalence of categories:
\[\pi(\cat{M}_{cf})\simeq \Ho(\cat{M}).\]
In particular, the localisation of $\cat{M}$ at the class of weak equivalences is a (locally small) category.
\end{corollary}
\begin{proof}
See \cite[Theorem 1.2.10]{hov07}.
\end{proof}

\begin{remark}
\label{rmk:model-cat-homotopical-cat}
One can show that the weak equivalences of every model category $\cat{M}$ satisfy the 2-out-of-6 property (see \cite{DHKS04}).
In particular, every model category $\cat{M}$ has an underlying homotopical category in the sense of \ref{subsec:homotopical-category}.
\end{remark}

\subsection{} %
Given $\cat{M}$ and $\cat{M}'$ model categories, a \emph{Quillen adjunction} from $\cat{M}$ to $\cat{M'}$ is an adjunction
\begin{equation*}
\Adjoint{\cat{M}}{\cat{M}'}{F}{G}
\end{equation*}
such that $F$ preserves cofibrations and $G$ preserves fibrations.
The left adjoint is called the \emph{left Quillen functor} and the right adjoint is called the \emph{right Quillen functor}.

\subsection{} Playing with the adjunction and the lifting properties, one can show that $(F,G)$ is a Quillen adjunction if and only if $F$ preserves cofibrations and trivial cofibrations.
Dually, this is equivalent to $G$ preserving fibrations and trivial fibrations.

\begin{theorem}
A Quillen adjunction 
\begin{equation}
\Adjoint{\cat{M}}{\cat{M}'}{F}{G} 
\end{equation}
induces an adjunction
\begin{equation*}
\Adjoint{\Ho(\cat{M})}{\Ho(\cat{M}')}{\Left F}{\Right G}
\end{equation*}
on the associated homotopy categories.
\end{theorem}
\begin{proof}
See \cite[Lemma 1.3.10]{hov07}.
\end{proof}

\subsection{} 
A Quillen adjunction 
\begin{equation*}
\Adjoint{\cat{M}}{\cat{M}'}{F}{G} 
\end{equation*}
is said to be a \emph{Quillen equivalence} if the induced adjunction at the level of homotopy categories is an adjoint equivalence of categories.

\begin{definition}
A model category $\cat{M}$ is said to be \emph{cofibrantly generated} if there exist sets of morphisms $I$ and $J$ in $\cat{M}$ whose domains are presentable
\footnote{The standard, slightly more general hypothesis in the definition of a cofibrantly generated model category is that $I$ and $J$ ``admit the small object argument'', see for example \cite[Chapter 11]{hir09}. We won't need that as we will deal mostly with locally presentable categories.},
such that the set of cofibrations is equal to $\llift{(\rlift{I})}$ and the set of trivial cofibrations is equal to $\llift{(\rlift{J})}$.
If this is the case, we will say that $I$ and $J$ are sets of \emph{generating cofibrations} and \emph{generating trivial cofibrations} respectively.
\end{definition}

\begin{theorem}
\label{thm:crans}
Let $\cat{M}$ be a cofibrantly generated model category, with set of generating cofibrations and trivial cofibrations $I$ and $J$ respectively.
Assume we have an adjunction
\begin{equation}
\Adjoint{\cat{M}}{\cat{M}'}{F}{G}
\end{equation}
and that the following properties hold:
\begin{enumerate}
\item
The category $\cat{M}'$ is complete and cocomplete.
\item
The domains of the maps in $F(I)$ and $F(J)$ are presentable.
\item
The image via $G$ of any morphism in $\llift{(\rlift{(FJ)})}$ is a weak equivalence.
\end{enumerate}
Then, there exists a cofibrantly generated model category structure on $\cat{M}'$ with $F(I)$ and $F(J)$ the sets of generating cofibrations and trivial cofibrations respectively.
Moreover, the weak equivalences and fibrations in $\cat{M}'$ are given by the maps in $\cat{M}'$ that are sent to weak equivalences and fibrations in $\cat{M}$ by $G$, respectively.
\end{theorem}
\begin{proof}
See for example \cite[Proposition 1.4.23]{cis06} or \cite{cra95}.
\end{proof}

\begin{definition}
Let $\cat{M}$, $\cat{N}$ and $\cat{P}$ be model categories. A functor $\otimes\colon \cat{M}\times\cat{N}\to \cat{P}$ is said to be a \emph{left Quillen bi-functor} if the following conditions are satisfied:
\begin{enumerate}
\item
Let $i\colon a\to b$ and $j\colon a'\to b'$ be cofibrations in $\cat{M}$ and $\cat{N}$, respectively. Then, the induced map
\[i\hat{\otimes} j \colon b\otimes a'\bigcoprod_{a\otimes a'} a\otimes b' \to b\otimes b'\]
is a cofibration in $\cat{P}$. Moreover, if either $i$ or  $j$ is a trivial cofibration, then $i\hat{\otimes} j$ is also a trivial cofibration.
\item
The functor $\otimes$ preserves small colimits separately in each variable.
\end{enumerate}
\end{definition}

\begin{definition}
A \emph{monoidal model category} is a closed monoidal category $(\cat{S},\otimes)$ equipped with a model structure, such that the monoidal product $\otimes\colon \cat{S}\times\cat{S}\to \cat{S}$ is a left Quillen bi-functor and the unit object $\1\in \cat{S}$ is cofibrant.
\end{definition}

\subsection{}
\label{def:enriched-model-category}
Let $(\cat{S}, \otimes)$ be a monoidal model category. An $\cat{S}$-\emph{enriched model category} is an $\cat{S}$-enriched, tensored and cotensored category $\cat{M}$ equipped with a model structure such that the tensor product $\otimes\colon \cat{M}\times\cat{S}\to \cat{M}$ is a left Quillen functor.

\begin{remark}
\label{rem:monoidal-model-structure}
Let $(\cat{S},\otimes)$ be a monoidal model category and let $\cat{M}$ be a model category. 
Suppose that $\cat{M}$, as a category, is enriched, tensored and co-tensored over $\cat{S}$. 
Then, by Lemma \ref{item:two-variable-lifting-2} the functor $\otimes\colon\cat{S}\times\cat{M}\to\cat{M}$ being  a left Quillen functor is equivalent to:
\begin{enumerate}
\item
Given any cofibration $i\colon d\to d'$ in $\cat{M}$ and any fibration $p\colon x\to y$ in $\cat{M}$, the induced map
\[\hat{\map}(i,p)\colon \map_{\cat{M}}(d',x)\to \map_{\cat{M}}(d,x)\times_{\map_{\cat{M}}(d,y)}\map_{\cat{M}}(d',y)\]
is a fibration in $\cat{S}$, which is trivial if either $i$ or $p$ is a weak equivalence.
\item
Given any cofibration $i\colon c\to c'$ in $\cat{S}$ and any fibration $p\colon x\to y$ in $\cat{M}$, the induced map:
\[\hat{p^i} \colon x^{c'}\longto x^c\times_{y^c}y^{c'}\]
is a fibration in $\cat{M}$, which is trivial if either $i$ or $p$ is a weak equivalence.
\end{enumerate}
\end{remark}
\section{Cisinski model structures\label{ch:2sec:3}}
In this section we recall how to construct Cisinski model structures on the category of presheaves on a small category $A$.
Cisinki model structures are determined by the choice of a functorial cylinder and by a class of anodyne extensions, satisfying a set of axioms.
The results of this section can be found in \cite{cis16}.

\subsection{} 
Let $A$ be a small category and let $\hat{A}$ be the category of presheaves on $A$.
A \emph{functorial cylinder} on $\hat{A}$ is an endofunctor $I\colon\hat{A}\to\hat{A}$ together with a decomposition
\begin{diagram}
1\coprod 1
\ar[r, "{(\partial_0,\partial_1)}"] &
I
\ar[r, "\sigma"] &
1
\end{diagram}
of the codiagonal on the identity endofunctor, such that the map $(\partial_0,\partial_1)$ is a monomorphism.
A functorial cylinder $I$ is said to be an \emph{exact cylinder} if it preserves monomorphisms, commutes with small colimits and if, for every monomorphism $j\colon K\to L$ of presheaves, the commutative square:
\begin{equation*}
\csquare{K}{L}{I\otimes L}{I\otimes K}{j}{\partial_\epsilon\otimes 1_L}{1_I\otimes j}{\partial_\epsilon\otimes 1_K}
\end{equation*}
is Cartesian for $\epsilon=0,1$.

\subsection{}
Given a cartesian square of presheaves
\begin{equation}
\pull{X}{Y}{T}{Z}{}{}{}{}
\end{equation}
where all morphisms are mono, the induced map $Y\coprod_X Z\to T$ is a monomorphism as well and we denote its image by $Y\cup Z\subset T$.
If $I$ is a functorial cylinder, we denote by $\{\epsilon\}$ the subobject of $I$ determined by the image of $\partial_\epsilon$, for $\epsilon =0,1$. Therefore, there are canonical inclusions
\[K\simeq \{\epsilon\}\otimes K \subset I\otimes K\]
for $\epsilon = 0,1$. Moreover, when $I$ is an exact cylinder, we have inclusions
\[I\otimes K\cup \{\epsilon\}\otimes L \subset I\otimes L\]
for $\epsilon=0,1$. To conclude, we denote by $\partial I$ the union $\{0\}\cup\{1\}$ of the subobjects $\{0\}$ and $\{1\}$ in  $I$.

\begin{definition}
Given an exact cylinder $I$ on $\hat{A}$, a class $\cat{B}$ of morphism of presheaves is a class of \emph{$I$-anodyne extensions} if it satisfies the following properties:
\begin{enumerate}
\item
There exists a set of monomorphisms $\Lambda$ such that $\cat{B}=\llift(\rlift{\Lambda})$.
\item
For any monomorphism $K\to L$ of presheaves, the induced map 
\[I\otimes K\cup\{\epsilon\}\otimes L\to I\otimes L\]
is in $\cat{B}$ for $\epsilon=0,1$.
\item
For any map $K\to L$ in $\cat{B}$, the induced map 
\[I\otimes K\cup \partial I\otimes L\to I\otimes L\]
is in $\cat{B}$.
\end{enumerate}
\end{definition}

\subsection{}
A \emph{homotopical structure} on $A$ is a pair $(I,\cat{B})$ where $I$ is an exact cylinder on $\hat{A}$ and $\cat{B}$ is a class of $I$-anodyne extensions.
We fix a homotopical structure $(I,\cat{B})$ on $A$.

\subsection{}
Given two morphisms $f_0,f_1\colon K\to X$ of presheaves, an $I$-\emph{homotopy} from $f_0$ to $f_1$ is a map $h\colon I\otimes K\to X$ such that $h(\partial_\epsilon\otimes1_K)=f_\epsilon$ for $\epsilon=0,1$.
In particular, since $I$ is functorial, we can define the smallest equivalence relation generated by $I$-homotopies on the set of morphisms from $K$ to $X$, which is compatible with composition.
We denote by $[K,X]$ the quotient of $\hom_{\hat{A}}(K,X)$ by such equivalence relation.

\subsection{}
We define a \emph{cofibration} to be a monomorphism, a \emph{naive fibration} to be a map in $\rlift{\cat{B}}$, a \emph{fibrant object} to be a presheaf $W$ such that $W\to \terminal$ is a naive fibration and a weak equivalence to be a morphism $f\colon X\to Y$ such that, for every fibrant object $W$, the induced map
\[f^*\colon [Y,W]\to [X,W]\]
is an isomorphism.

\begin{theorem}
\label{thm:model-structure-presheaves}
The class of cofibrations and weak equivalences determines a cofibrantly generated model category structure on the category $\hat{A}$ of presheaves on $A$. Moreover, a morphism $p\colon X\to Y$ between fibrant objects is a fibration, if and only if it is a naive fibration and a monomorphism $j\colon K\to W$ between fibrant objects is a trivial cofibration if and only if it is an $I$-anodyne extension.
\end{theorem}
\begin{proof}
See \cite[Theorem 2.4.19]{cis16}.
\end{proof}
\section{Topological constructs and Cisinski model structures}
\label{ch:2sec:4}
In this section we investigate the relationship between Cisinski model structures on presheaf categories and topological constructs.
Given a functor $u\colon A\to \cat{C}$ from a small category $A$ to a topological construct $\cat{C}$, one can consider the final closure $\icat{C}$ of $uA$ in $\cat{C}$. 
The objects in $\icat{C}$ should be thought of as locally modelled by the ``basic shapes'' in $A$, always keeping in mind $A=\DDelta$ and $\cat{C}=\Topa$ as a fundamental example.
We show that, if the functor $u$ is nice enough, one can endow the category $\icat{C}$ with the structure of an $\hat{A}$-enriched category, which is tensored and cotensored.
Moreover, if $\hat{A}$ is endowed with a Cisinski model structure, we give a conditional result for the category $\icat{C}$ to possess a model structure, transferred from $\hat{A}$.
To conclude the full subcategory of $\icat{C}$ spanned by the objects that are fibrant in the candidate model structure has the structure of a category of fibrant objects, unconditionally.
To our knowledge, the results of this section are not spelled out anywhere else in the literature.

We fix a small category $A$ with a terminal object $\terminal$ and we consider the category $\hat{A}$ of presheaves on $A$. 
\subsection{}
\label{subsec:cisinski-topological}
Let $\cat{C}$ be a well fibered topological construct and assume we have a functor $u\colon A\to \cat{C}$. Then  by Theorem \ref{thm:kan-presheaves}, there exists an adjunction:
\begin{equation}
\label{adjoint:presheaf-topcat}
\Adjoint{\hat{A}}{\cat{C}}{u_!}{u_*}
\end{equation}
Assume now that $u_{!}$ preserves products of representable presheaves, then we have:
\begin{equation}
\label{eq:u-preserves-representable}
u(a)\times u(a')=u_{!}\left(h^{a}\times h^{a'}\right)=\colim_{h^{x}\to h^{a}\times h^{a'}}u(x).
\end{equation}
Assume furthermore that $ua$ is an exponentiable object in $\cat{C}$ for every $a\in A$, let $uA$ be the full subcategory of $\cat{C}$ spanned by $ua$ for every $a\in A$ and let $\icat{C}$ be the final closure of $uA$ in $\cat{C}$.
Then, $\icat{C}$ is Cartesian closed by \eqref{eq:u-preserves-representable} and Theorem \ref{thm:final-closure-cartesian-closed}. 
Moreover, the adjunction \eqref{adjoint:presheaf-topcat} restricts to an adjunction
\begin{equation}
\label{adjoint:presheaf-topcat-gen}
\Adjoint{\hat{A}}{\icat{C}}{u_!}{u_*}
\end{equation}

\begin{proposition}
\label{prop:cisinski-topological-enriched}
The functor $u_{!}\colon \hat{A}\to \icat{C}$ is strong monoidal.
In particular, $\icat{C}$ is a category enriched, tensored and cotensored over $\hat{A}$ and \eqref{adjoint:presheaf-topcat-gen} is an $\hat{A}$-enriched adjunction.
\end{proposition}
\begin{proof}
The functor $u_{!}$ preserves products of representable presheaves by hypothesis, hence it preserves all products since $\icat{C}$ is Cartesian closed.
The second statement follows immediately from Corollary \ref{cor:adjunction-v-adjunction}.
\end{proof}

\subsection{}
More explicitly, given $X$ and $Y$ in $\icat{C}$ the \emph{mapping space} between $X$ and $Y$ is the presheaf defined by 
\[\map(X,Y)=u_*\left(\ihom(X,Y)\right)\]
Moreover, if $X$ is an object in $\icat{C}$ and $K$ is a presheaf on $A$, the \emph{tensor} of $X$ by $K$ is the presheaf defined by
\[K\otimes X= u_{!}K\times X\]
and dually, the \emph{cotensor} of $X$ by $K$ is the presheaf given by:
\[X^K=\ihom(u_!(K), X)\]

\subsection{} 
We now assume that we are given a Cisinski model structure on $\hat{A}$, i.e. a combinatorial model category structure for which the cofibrations are monomorphisms. 
Moreover, we assume that such a model structure is Cartesian closed, meaning that the model structure on $\hat{A}$ is a monoidal model structure with respect to the Cartesian product.

\begin{definition}
\label{def:model-structure-top-cat}
Let $f\colon X\to Y$ be a morphism in $\cat{C}$. 
We say that $f$ is a \emph{weak equivalence} (resp. a \emph{fibration}) if the morphism $u_*(f)$ is a weak equivalence (resp. a fibration) in $\hat{A}$.
We call $f$ an \emph{acyclic fibration} if it is both a weak equivalence and a fibration.
\end{definition}

\begin{definition}
A morphism $i\colon A\to B$ in $\cat{C}$ is said to be a \emph{cofibration} (\resp a trivial cofibration) if it has the left lifting property with respect to all trivial fibration (\resp all fibrations).
We call $i$ an \emph{acyclic cofibration} if it is both a cofibration and a weak equivalence
\footnote{The reader will notice the asymmetry between fibrations and cofibrations.
Indeed, by definition we know that the class of acyclic fibrations is exactly the class of morphisms with the right lifting property with respect to all cofibrations.
On the other hand, acyclic cofibrations and trivial cofibrations do not necessarily coincide, and this is somewhat the source of difficulties in practical applications.
}.
\end{definition}

\begin{definition}
Let $X$ be an object of $\cat{C}$.
We say that $X$ is fibrant if $u_{*}X$ is a fibrant object in $\hat{A}$.
We say that $X$ is cofibrant if the unique morphism $\initial\to X$ is a cofibration.
\end{definition}

\begin{lemma}
\label{lem:top-cat-homotopical-cat}
The category $\cat{C}$ equipped with the class of weak equivalences has the structure of a homotopical category.
\end{lemma}
\begin{proof}
Since $u_{*}$ reflects weak equivalences and $\hat{A}$ is a model category, the claim follows from Remark \ref{rmk:model-cat-homotopical-cat}.
\end{proof}

\begin{lemma}
\label{lem:top-cat-prod-we}
Let $X$ be an object of $\cat{C}$ and let $w\colon Y\to Z$ be a weak equivalence in $\cat{C}$, then the induced map:
\[X\times w\colon X\times Y\to X\times Z.\]
is a weak equivalence.
\end{lemma}
\begin{proof}
Since $u_{*}$ preserves products and preserves and reflects weak equivalences, it is enough to prove that the product of a presheaf in $\hat{A}$ with a weak equivalence  is a weak equivalence. This follows from the fact that every object in $\hat{A}$ is cofibrant together with the dual of Corollary \ref{cor:ken-brown-lemma}.
\end{proof}

\begin{proposition}
\label{prop:cotensor-Quillen-functor}
Let $i\colon K\to L$ be a monomorphism in $\hat{A}$ and let $X\to Y$ be a fibration in $\icat{C}$. Then, the induced map
\[\hat{p^{i}}\colon X^L\to X^K\times_{Y^K}Y^L\]
is a fibration, which is acyclic if $i$ or $p$ is so.
\end{proposition}
\begin{proof}
Since $u_{*}$ reflects fibrations and acyclic fibrations by definition, it is enough to show that $u_{*}\left(\hat{p^{i}}\right)$ is a fibration, which is acyclic when either $i$ or $p$ is so.
Since $(u_{!},u_{*})$ is an enriched adjunction, the functor $u_{*}$ maps cotensors to cotensors, and since $\hat{A}$ is self enriched, the claim is equivalent to show that the map
\[\hat{\ihom}(i,p)\colon \ihom(L,u_{*}X)\to \ihom(K,u_{*}X)\times_{\ihom(K,u_{*}Y)}\ihom(L,u_{*}Y)\]
is a fibration, which is acyclic when either $i$ or $p$ is. This follows from the hypothesis that the model structure on $\hat{A}$ is Cartesian monoidal.
\end{proof}

\begin{corollary}
\label{cor:internal-hom-quillen-top-cat}
Let $f\colon X\to Y$ be a fibration and let $i\colon A\to B$ be a cofibration. Then, the induced map
\[\ihom(B,X)\to \ihom(A,X)\times_{\ihom(A,Y)}\ihom(B,Y)\]
is a fibration, which is acyclic if either $p$ is acyclic or $i$ is a trivial cofibration.
\end{corollary}

\subsection{}
Let $I$ be an interval object for $\hat{A}$. i.e. a presheaf on $A$ together with a factorisation of the codiagonal 
\[\terminal\coprod\terminal\to I\to\terminal\]
such that $(\partial_0,\partial_1)\colon *\coprod *\to I$ is a monomorphism and $\sigma\colon I\to *$ is a weak equivalence. 
Then, for every object $X$ in $\icat{C}$, applying the functor $X^{\blank}$ we get a factorisation
\[X\to X^I\to X\times X\]

\begin{lemma}
\label{lem:path-object-for-top-cat}
For every fibrant object $X$ in $\icat{C}$, the above factorisation induces the structure of a path object $X^I$ for $X$.
\end{lemma}
\begin{proof}
Indeed, by Proposition \ref{prop:cotensor-Quillen-functor} we have that $X^I\to X\times X$ is a fibration, since $\terminal\coprod\terminal\to I$ is a cofibration and $X$ is a fibrant object. 
Moreover, by Corollary \ref{cor:ken-brown-lemma} and Proposition \ref{prop:cotensor-Quillen-functor}, the functor $X^{(-)}\colon \hat{A}^\op\to \icat{C}$ maps weak equivalences to weak equivalences.
\end{proof}

\begin{proposition}
\label{prop:top-cat-fibrant-objects}
Under the hypotheses of \ref{subsec:cisinski-topological}, the full subcategory $\icat{C}^{f}$ of $\icat{C}$ spanned by the fibrant objects is a category of fibrant objects.
\end{proposition}
\begin{proof}
We need to show that $\icat{C}^{f}$ satisfies the hypothesis of Definition \ref{def:category-fibrant-objects}. 
Items  \ref{item:category-fibrant-objects-1}, \ref{item:category-fibrant-objects-2}, \ref{item:category-fibrant-objects-3} and \ref{item:category-fibrant-objects-5} are immediate from the definitions and item \ref{item:category-fibrant-objects-4} follows from Lemma \ref{lem:path-object-for-top-cat}.
\end{proof}

\subsection{}
Let us assume now that we have an object $E$ in $\icat{C}$ together with a factorisation
\begin{equation}
\label{eq:chaotic-interval-top-cat}
\begin{tikzcd}
\terminal\coprod\terminal \ar[r, "{\left(\partial^0,\partial^1\right)}"] &
E \ar[r, "\sigma"] &
\terminal
\end{tikzcd}
\end{equation}
of the codiagonal on the terminal object, such that $\left(\partial^{0}, \partial^{1}\right)$ is a cofibration and $\sigma$ is a weak equivalence.

\subsection{}
For every object $X$ in $\icat{C}$, applying the functor $X\times \blank$ to \eqref{eq:chaotic-interval-top-cat} yields a decomposition of the codiagonal on $X$
\begin{equation}
\label{eq:cylinder-top-cat}
\begin{tikzcd}
X\coprod X \ar[r, "{\left(\partial_{X}^0,\partial_{X}^1\right)}"] &
X\times E \ar[r, "\sigma_{X}"] &
X
\end{tikzcd}
\end{equation}

\begin{lemma}
\label{lem:cylinder-top-cat}
For every object $X$ in $\icat{C}$ the morphism $\sigma_{X}$ in \eqref{eq:cylinder-top-cat} is a weak equivalence.
Moreover, if $X$ is a cofibrant object, the morphism $\left(\partial_{X}^{0}, \partial_{X}^{1}\right)$ is a cofibration.
\footnote{Notice that we don't use the word ``cylinder object'' in this case, since a priori the dual of the category of cofibrant objects in $\icat{C}$ is not a category of fibrant objects in the sense of Definition \ref{def:category-fibrant-objects}}
\end{lemma}
\begin{proof}
The first part of the lemma follows immediately from Lemma \ref{lem:top-cat-prod-we} and from the fact that $\sigma\colon E\to \term$ is a weak equivalence.
The second part of the lemma follows from Corollary \ref{cor:internal-hom-quillen-top-cat} and Lemma \ref{lem:two-variable-lifting}.
\end{proof}

\subsection{}
\label{subsec:E-homotopy}
Given two morphisms $f\colon X\to Y$ and $g\colon X\to Y$ in $\icat{C}$, an $E$-\emph{homotopy} from $f$ to $g$ is a morphism $h\colon X\times E\to Y$ such that $h\circ \partial_{X}^{0}$ is equal to $f$ and $h\circ\partial_{X}^{1}$ is equal to $g$.

\begin{remark}
\label{rem:partial-we-top-cat}
For every object $X$ of $\icat{C}$ and $\epsilon \in \{0,1\}$, the morphism $\partial_{X}^{\epsilon}\colon X\to X\times E$ is a weak equivalence.
Indeed, since $\sigma_{X}$ is a weak equivalence, it is enough to apply the 2-out-of-3 property to the diagram:
\begin{equation}
\begin{tikzcd}
X\ar[d, "{\partial_{X}^{\epsilon}}"']\ar[dr, "1_{X}"]&\\
X\times E\ar[r, "\sigma_{X}"'] &X
\end{tikzcd}
\end{equation}
\end{remark}

\begin{definition}
\label{def:E-homotopy}
Let $f\colon X\to Y$ be a morphism in $\cat{C}$. We say that $f$ is an $E$-\emph{equivalence} if there exists a morphism $g\colon Y\to X$, an $E$-homotopy from $gf$ to $1_{X}$ and an $E$-homotopy from $fg$ to $1_{Y}$.
\end{definition}

\begin{proposition}
\label{prop:E-homotopy-we}
Every $E$-homotopy equivalence is a weak equivalence.
\end{proposition}
\begin{proof}
Let $f\colon X\to Y$ be an $E$-equivalence, then there exists a morphism $g\colon Y\to X$  and $E$-homotopies $h\colon X\times E\to X$ and $k\colon Y\times E\to Y$ fitting in commutative diagrams:
\begin{equation}
\begin{tikzcd}
X\ar[d, "{\partial_{X}^{0}}"']\ar[dr, "gf"]&\\
X\times E \ar[r, "h" description]& X\\
X\ar[u, "{\partial_{X}^{1}}"]\ar[ur, "1_{X}"']&
\end{tikzcd}
\quad
\begin{tikzcd}
&Y\ar[d, "{\partial_{Y}^{0}}"]\ar[dl, "fg"']\\
Y&Y\times E \ar[l, "k" description]\\
& Y\ar[u, "{\partial_{Y}^{1}}"']\ar[ul, "1_{Y}"].
\end{tikzcd}
\end{equation}
By Remark \ref{rem:partial-we-top-cat} and a repeated use of the 2-out-of-3 property, both $fg$ and $gf$ are weak equivalences.
To conclude, it is enough to apply the 2-out-of-6 property to the diagram
\begin{equation}
\begin{tikzcd}
X\ar[r, "f" description]\ar[rr, bend left, "gf"]&Y\ar[r,"g" description]\ar[rr, bend right, "fg"']& X\ar[r, "f" description]& Y
\end{tikzcd}
\end{equation}
\end{proof}

\begin{theorem}
\label{thm:transfer-cartesian-closed}
Under the above hypothesis and notation, let $\cat{I}$  and $\cat{J}$ be sets of generating cofibrations and trivial cofibrations for the model structure on $\hat{A}$. 
Assume that every morphism in $\llift{(\rlift{u_!\cat{J}})}$ is a weak equivalence.
Then, there exists a combinatorial, $\hat{A}$-enriched and cartesian closed model structure on $\icat{C}$ such that the adjunction
\begin{equation*}
\Adjoint{\hat{A}}{\icat{C}}{u_!}{u_*}
\end{equation*}
is an enriched Quillen adjunction.
\end{theorem}
\begin{proof}
The existence of a combinatorial model structure and the fact that the adjunction \eqref{adjoint:presheaf-topcat-gen} is a Quillen adjunction follow from Theorem \ref{thm:crans} and Theorem \ref{thm:C_I-locally-presentable}.
The model structure is $\hat{A}$-enriched by Proposition \ref{prop:cotensor-Quillen-functor} and Remark \ref{rem:monoidal-model-structure}.
The Quillen adjunction is enriched by Proposition \ref{prop:cisinski-topological-enriched}.
\end{proof}

\begin{corollary}
In the above situation, assume that there exists a functor $R\colon \icat{C}\to \icat{C}^{f}$ together with a natural weak equivalence $r_X\colon X\to RX$.
Then, there exists a combinatorial, $\hat{A}$-enriched and Cartesian closed model structure on $\icat{C}$ such that the adjunction
\begin{equation*}
\Adjoint{\hat{A}}{\icat{C}}{u_!}{u_*}
\end{equation*}
is an enriched Quillen adjunction.
\end{corollary}
\begin{proof}
By Theorem \ref{thm:transfer-cartesian-closed} it is enough to show that every trivial cofibration is acyclic. Let $i\colon A\to B$ be a trivial cofibration in $\cat{C}$, then there exists a lift in the following diagram:
\begin{equation*}
\lift{A}{RA}{*}{B}{r_A}{}{}{i}{s}
\end{equation*}
since $RA$ is a fibrant object.
Again, since $i$ is a trivial cofibration, there exists a lift in the solid commutative diagram:
\begin{diagram}
A \ar[d, "i"'] \ar[r, "{r_Bi}"]  &
RB\ar[r] & 
{RB^I}\ar[d] \\
B \ar[r, "{(r_B,s)}"'] \ar[rru, dashed, "h" description ] &
{RB\times RA} \ar[r, "{1_{RB}\times Ri}"'] &
{RB\times RB}
\end{diagram}
in particular $u_*(h)\colon u_*(B)\to u_*RB^I$ defines a right homotopy between $u_*(r_B)$ and $u_*(Ri s)$, so that $Ri s$ is a weak equivalence. But then, looking at the diagram
\begin{diagram}
A \ar[r, "i"] &
B \ar[r, "s"] &
RA \ar[r, "Ri"] &
RB
\end{diagram}
we have that $si=r_A$ and $Ris$ are weak equivalences. 
Hence, by Lemma \ref{lem:top-cat-homotopical-cat} $i$ is a weak equivalence as well.
\end{proof}

\chapter{More on Simplicial sets\label{ch:3}}

\section{Model structures on simplicial sets}
\label{ch:3sec:1}
We recall how to define the Kan-Quillen and Joyal model structures on the category of simplicial sets.
The results of this section can be found in \cite{cis16}.

\subsection{}
Recall that the category of simplicial sets is the category $\sSet$ of presheaves on the category $\DDelta$ of finite non-empty ordinals (see \ref{subsec:simplicial-sets}).

\subsection{} 
\label{subsec:face-degeneracy}
For integers $n\ge 1$ and $0\le i\le n$ we let 
\[\partial_i^n\colon \Delta^{n-1}\to\Delta^n\]
be the $i$-th \emph{face map}, i.e. the morphism corresponding to the unique strictly monotonic map $[n-1]\to[n]$ which skips the value $i$.
Similarly, for all $n\ge 0$ and $0\le i\le n$ we let 
\[\sigma_i^n\colon \Delta^{n+1}\to \Delta^n\]
be the $i$-th \emph{degeneracy map}, i.e. the morphism corresponding to the unique surjective map $[n+1]\to [n]$ which hits the value $i$ twice.
One can show (see \cite[1.1]{JT08}) that $\DDelta$ can be recovered as the free category on the reflexive quiver determined by the morphisms $\partial^{n}_{i}$ and $\sigma^{n}_{i}$, quotiented by the relations:
\begin{align*}
\partial^{n+1}_{j}\partial^{n}_{i} & = 
\partial^{n+1}_{i} \partial^{n}_{j-1} \quad\quad i<j\\
\sigma^{n}_{j} \sigma^{n+1}_{i} & = 
\sigma^{n}_{i} \sigma^{n+1}_{j+1} \quad\quad i\le j\\
\sigma^{n-1}_{j}\partial^{n}_{i} & = \begin{cases}
\partial^{n-1}_{i} \sigma^{n-2}_{j-1} & i<j\\
\id & i=j, j+1\\
 \partial^{n-1}_{i-1} \sigma^{n-2}_{j} &i>j+1
\end{cases}
\end{align*}

\begin{notation}
\label{not:simplicial-sets}
If $X$ is a simplicial set, we write $X_{n}$ for the set $X([n])$ of $n$-\emph{simplices} of $X$. Moreover, if $\alpha\colon [n]\to [m]$ is a morphism in $\DDelta$, we write $\alpha^{*}\colon X_{m}\to X_{n}$ for the induced morphism.
In particular, we use the following notation:
\begin{align*}
d^{i}_{n}=\left(\partial^{n}_{i}\right)^{*}\colon X_{n}&\to X_{n-1}\\
s^{i}_{n}=\left(\sigma^{n}_{i}\right)^{*}\colon X_{n}&\to X_{n+1}.
\end{align*}
\end{notation}

\begin{definition}
Let $X$ be a simplicial set. An $n$-simplex $x$ of $X$ is said to be \emph{degenerate} if there exists a surjection $\eta\colon [n]\to [m]$ and an $m$-simplex $y$ of $X$ such that $\eta^{*}y=x$.
\end{definition}

\begin{proposition}[Eilenberg-Zilber Lemma]
For every $x\in X_{n}$ there exists a unique surjection $\eta\colon [n]\to [m]$ and a unique non-degenerate $m$ simplex $y$ such that $x=\eta^{*}y.$
\end{proposition}

\subsection{}
Let $X$ be a simplicial set. 
Recall that the $n$-skeleton of $X$ (see Example \ref{ex:skeleton-coskeleton}) is the simplicial set given by taking the restriction of $X$ to the category $\DDelta_{\le n}$ of finite ordinals $[m]$, with $m\le n$, and then taking the left Kan extension along the inclusion $\iota_{n}\colon\DDelta_{\le n}\to \DDelta$.
In particular, one can see that the $m$-simplices of $\sk_{n}X$ coincide with the $m$-simplices of $X$ for $m\le n$, while they are all degenerate for $m> n$.
Moreover, the $n$-skeleton functor  induces a filtration
\[\sk_{0}X\subset \sk_{1}X\subset \ldots\subset \sk_{n}X\subset \ldots\bigcup_{n\ge 0}\sk_{n}X=X\]
of $X$ by subobjects (see \cite[1.2]{JT08}). 
We say that $X$ is \emph{finite-dimensional}, if there exists an $n$ such that 
\begin{equation}
\label{eq:n-skeletal}
\sk_{n}X=X.
\end{equation}
If $n$ is the minimal integer for which \eqref{eq:n-skeletal} holds, we will say that $X$ is $n$-\emph{dimensional}.
We say that $X$ is \emph{infinite-dimensional} if there exists no $n$ for which \eqref{eq:n-skeletal} occurs.
Clearly, the adjunction \eqref{eq:n-truncated-simplicial-set} defines an equivalence between the category of $n$-truncated simplicial sets and the category of $n$-dimensional simplicial sets.

\subsection{}
Let $\rho\colon \DDelta\to \DDelta$ be defined as the identity on the objects and on the generators of $\DDelta$ as follows:
\begin{align*}
\rho\left(\partial^{n}_{i}\right) & = \partial^{n}_{n+1-i},\\
\rho\left(\sigma^{n}_{i}\right) & =\sigma^{n}_{n-i}.
\end{align*}
Then, precomposition with $\rho$ determines an involution
\begin{align*}
(\blank)^{\op}\colon \sSet &\to\sSet
\end{align*}
on the category of simplicial sets.
For every simplicial set $X$ we call $X^{\op}$ the \emph{opposite} simplicial set of $X$.
Moreover, if $C$ is a small category, the following natural isomorphism holds (see \cite[Proposition 1.5.8]{cis16}):
\[\Nerv \left(C^{\op}\right)\cong \left(\Nerv C\right)^{\op}\]
where $\Nerv\colon \Cat\to \sSet$ denotes the nerve functor, defined in \ref{ex:nerve-fundamental-category}.
\begin{example}

\label{ex:boundary}
For every finite ordinal $[n]$, the \emph{boundary} of the standard $n$-simplex is the simplicial set $\partial \Delta^n$ defined as the union of the images of all the face maps:
\[\partial\Delta^n=\bigcup_{0\le i\le n} \im(\partial_i^n)\]
Notice that $\partial\Delta^{n}$ is an $n-1$ dimensional simplicial set. Moreover, the following equality holds:
\[\sk_{n-1}\Delta^{n}=\partial\Delta^{n}.\]
\end{example}

\begin{example}
\label{ex:horn}
For $0\le k\le n$, the \emph{$k$-th horn} of the standard $n$-simplex is the simplicial set $\Lambda^n_k$ defined as the union of the images of all except the $k$-th face map:
\[\Lambda^n_k=\bigcup_{i\ne k}\im(\partial_i^n).\]
\end{example}

\begin{example}
\label{ex:n-spine}
For every finite ordinal $[n]$ the $n$-\emph{spine} is the sub-simplicial set $\Sp^{n}$ of $\Delta^{n}$ given by the union:
\[\Sp^n=\bigcup_{i=0}^{n-1}\Delta^{\{i,i+1\}}\]
where $\Delta^{i,i+1}$ denotes the edge from $i$ to $i+1$ in $\Delta^{n}$, \ie the sub-simplicial set of $\Delta^{n}$ spanned by the vertices $i$ and $i+1$ (see \cite[Subsection 1.4.5]{cis16}).
Notice that $\Sp^{1}=\Delta^{1}$ and $\Sp^{2}=\Lambda^{2}_{1}$.
\end{example}

\subsection{}
Let $X$ be a simplicial set.
For every $[n]\in \DDelta$, let us denote by $X_{n}^{\nd}$ the set of non-degenerate $n$-simplices of $X$. 
Then, there exists a pushout diagram
\begin{equation}
\push{\displaystyle\bigcoprod_{X^{\nd}_n}\partial\Delta^{n}}{\displaystyle\bigcoprod_{X^{\nd}_n}\Delta^{n}}{\sk_{n}X}{\sk_{n-1}X}{}{}{}{}
\end{equation}
See \cite[1.2]{JT08}.

\begin{remark}
\label{rmk:realisation-cw}
Let us consider the geometric realisation functor $\abs{\blank}\colon \sSet\to \Topa$ defined in Example \ref{ex:geometric-realisation}.
Since $\abs{\blank}$ is a left adjoint, it preserves colimits. Therefore, for every simplicial set $X$ and every $n\in \NN$ there is a pushout diagram:
\begin{equation}
\label{eq:realisation-cw}
\push{\displaystyle\bigcoprod_{X^{\nd}_n}\partial\abs{\Delta^{n}}}{\displaystyle\bigcoprod_{X^{\nd}_n}\abs{\Delta^{n}}}{\abs{\sk_{n}X}}{\abs{\sk_{n-1}X}}{}{}{}{}
\end{equation}
where $\partial\abs{\Delta^{n}}$ denotes the topological boundary of the realisation of $\Delta^{n}$ and it coincides with $\abs{\partial \Delta^{n}}$ by \cite[1.3]{JT08}.
In particular, $\abs{X}$ has the structure of a CW-complex.
\end{remark}

\subsection{}
Let $\Lambda$ be the set of inclusions $\{\Lambda^n_k\to \Delta^n: n\ge 1, 0\le k\le n\}$, a morphism in $\llift{(\rlift{\Lambda})}$ is called an \emph{anodyne extension}. 
The endofunctor $I=\Delta^1\times(\blank)$ given by taking product with the standard 1-simplex defines an exact cylinder on $\sSet$. 
A classical result of Gabriel and Zisman implies that the class of anodyne extensions coincides with the smallest class of $I$-anodyne extensions. 

\begin{definition}
A morphism $p\colon X\to Y$ is said to be a \emph{Kan fibration} if it belongs to $\rlift{\Lambda}$ and a simplicial set $W$, whose structural map $W\to \Delta^0$ is a Kan fibration, is said to be a \emph{Kan complex}.
\end{definition}

\subsection{}
The model category structure obtained by applying Theorem \ref{thm:model-structure-presheaves} to the exact cylinder $I$ and the class of anodyne extensions is called the \emph{Kan-Quillen model category structure}.
Weak equivalences in this model structure are called \emph{weak homotopy equivalences}.
Moreover, one can show that any anodyne extension is a trivial cofibration and that the fibrations in this model structure are precisely the Kan fibrations.

\begin{theorem}
There exists a cofibrantly generated model category structure on the category $\Topa$ of topological spaces such that the adjunction introduced in Example \ref{ex:geometric-realisation} defines a Quillen equivalence with the Kan-Quillen model structure on $\sSet$.
Moreover, the weak equivalences in this model structure are the weak homotopy equivalences.
\end{theorem}
\begin{proof}
A classical proof using Theorem \ref{thm:crans} can be found in \cite{qui67}.
\end{proof}

\subsection{} 
\label{subsec:morphisms-sset}
Let $X$ be a simplicial set.
A \emph{morphism} $\alpha\colon x\to x'$ in $X$ is a 1-simplex $\alpha\colon \Delta^1\to X$ in $X$ such that $x=\alpha\partial^1_1$ is the source of $\alpha$ and $x'=\alpha\partial^1_0$ is the target of $\alpha$.

\begin{example}
Given a $0$-simplex $x$ in a simplicial set $X$, the \emph{identity} on $x$ is the morphism 
\[1_{x}=s^{0}_{0}(x)\colon x\to x.\]
\end{example}

\subsection{}
\label{subsec:triangles-sset}
A \emph{triangle} in $X$ is a map $t\colon \partial \Delta^{2}\to X$, which can be seen as a triple $(f,g,h)$ of morphisms in $X$ such that the source of $g$ coincides with the target of $f$, while the sources of $f$ and $h$ coincide and dually the targets of $g$ and $h$ coincide.
A triangle $t\colon \partial \Delta^{2}\to X$ is said to be \emph{commutative} if there exists a morphism $c\colon \Delta^{2}\to X$ of simplicial sets such that the restriction of $c$ to $\partial \Delta^{2}$ coincides with $t$.

\subsection{}
Let $\alpha\colon x\to x'$ be a morphism in a simplicial set $X$, a \emph{left inverse} to $\alpha$ is a morphism $\beta\colon x'\to x$ such that the triangle $(\alpha,\beta,1_{x})$ is commutative.
Dually, a right inverse to $\alpha$ is a left inverse to $\alpha$ in the opposite simplicial set $X^{\op}$.

\begin{definition}
\label{def:morphism-invertible-sset}
Let $X$ be a simplicial set. A morphism $\alpha\colon x\to x'$ in $X$ is said to be \emph{invertible}, if $\alpha$ has a left inverse and a right inverse.
Explicitly, $\alpha$ is invertible if there exist morphisms $\beta\colon x'\to x$ and $\gamma\colon x'\to x$ in $X$ such that the triangles:
\begin{equation}
\begin{tikzcd}[row sep= small, column sep=small]
& x'\ar[rdd, "\beta"]&\\
&&\\
x\ar[uur, "\alpha"]\ar[rr, "1_{x}"']& & x
\end{tikzcd}
\quad
\begin{tikzcd}[row sep= small, column sep=small]
& x\ar[rdd, "\alpha"]&\\
&&\\
x'\ar[uur, "\gamma"]\ar[rr, "1_{x'}"']& & x'
\end{tikzcd}
\end{equation}
commute.
\end{definition}

\begin{construction}
\label{con:left-right-inverse}
Let $\alpha\colon x\to x'$ be a morphism in a simplicial set $X$.
Following \cite[3.3.1]{cis16} one can freely add a left inverse to $\alpha$.
We construct the simplicial set $X[\beta\alpha=1]$ as follows:
First, we freely adjoin a morphism $\beta\colon x'\to x$ taking the pushout:
\begin{equation*}
\push{\partial\Delta^{1}}{X}{X[\beta]}{\Delta^{1}}{(x',x)}{}{\beta}{}
\end{equation*}
Then, we glue a $2$-simplex:
\begin{equation*}
\push{\partial\Delta^{2}}{X[\beta]}{X[\beta\alpha=1]}{\Delta^{2}}{(\alpha,\beta,1_{x})}{}{c}{}
\end{equation*}
manifesting the commutativity of the triangle $(\alpha,\beta,1_{x})$.
By applying this procedure to $X^{\op}$ one can freely add a right inverse to $\alpha$, obtaining the simplicial set $X[\alpha\gamma=1]$.
\end{construction}

\begin{example}
\label{ex:walking-retraction-sset}
Let $\Delta^{1}$ be the standard $1$-simplex and let $\alpha\colon 0\to 1$ be the unique morphism from $0$ to $1$ in $\Delta^{1}$.
We denote by $R=\Delta^{1}[\beta\alpha=1]$ the simplicial set given by freely adding a left inverse to $\alpha$.
\end{example}

\begin{definition}
\label{def:localisation-sset}
Let $X$ be a simplicial set and let $\alpha\colon x\to x'$ be a morphism in $X$.
The \emph{localisation} of $X$ at $\alpha$ is the simplicial set $X[\alpha^{-1}]$ defined by freely adding a right and a left inverse to $X$:
\[X[\alpha^{-1}]=X[\beta\alpha=1][\alpha\gamma=1].\]
\end{definition}

\begin{example}
\label{ex:J-interval-sset}
We define the interval $J$ as:
\[J=\Delta^{1}[\alpha^{-1}]\]
where $\alpha$ is the unique non-degenerate $1$-simplex in $\Delta^{1}$.
One can show that taking the product with $J$ defines an exact cylinder on $\sSet$ (See \cite[Chapter 3, Section 3]{cis16}).
\end{example}

\subsection{} 
Let $\Lambda'$ be the set of \emph{inner horn inclusions}. 
This is the set of horn inclusions $\Lambda^n_k\to \Delta^n$ for $n\ge 2$ and $0<k<n$. An \emph{inner anodyne extension} is a map of $\llift{(\rlift{\Lambda'})}$. 
A \emph{categorical anodyne extension} is an element of the smallest class of $J$-anodyne extensions containing the inner anodyne extensions.

\begin{definition}
\label{def:Joyal-model-structure}
The \emph{Joyal model category structure} on $\sSet$ is the model category structure on $\sSet$ obtained by applying Theorem \ref{thm:model-structure-presheaves} to the exact cylinder $J\times(\blank)$ and to the class of categorical anodyne extensions.
\end{definition}

\begin{definition}
Let $p\colon X\to Y$ be a morphism of simplicial sets and let $W$ be a simplicial set
\begin{enumerate}
\item
We say that $p$ is an \emph{inner fibration} if $p$ has the right lifting property with respect to the class of inner anodyne extensions (equivalently, with respect to the inner horn inclusions).
\item
We say that $p$ is an \emph{isofibration}, if $p$ is an inner fibration and has the right lifting property with respect to the inclusion $\{0\}\to J$.
\item
We say that $W$ is an \emph{$\infty$-category}, if the unique map $W\to \Delta^0$ is an inner fibration.
\end{enumerate}
\end{definition}

\begin{theorem}[Joyal]
A simplicial set is a fibrant object in the Joyal model category structure if and only if it is an $\infty$-category. 
A morphism of $\infty$-categories is a fibration in the Joyal model category structure if and only if it is an isofibration.
\end{theorem}
\begin{proof}
See \cite[Theorem 3.6.1]{cis16}.
\end{proof}

\subsection{}
Let $X$ and $Y$ be $\infty$-categories.
A \emph{functor} from $X$ to $Y$ is a morphism $f\colon X\to Y$ of simplicial sets.
A natural transformation from a functor $f\colon X\to Y$ to a functor $g\colon X\to Y$ is a $\Delta^{1}$-homotopy from $f$ to $g$.
A natural transformation is said to be invertible if for every object $x\in X$, the induced morphism $fx\to gx$ is invertible in $Y$.
A functor $f\colon X\to Y$ is said to be an \emph{equivalence of $\infty$-categories} if there exists a functor $g\colon Y\to X$ and invertible natural transformations from $fg$ to $1_{Y}$ and from $gf$ to $1_{X}$.
One can show that a functor is an equivalence of $\infty$-categories if and only if it is a $J$-homotopy equivalence.

\subsection{}
Let $\Cat$ be the category of small categories.
A functor $p\colon C\to D$ in $\Cat$ is said to be an \emph{isofibration} if for every object $c_{0}$ in $C$ and for every isomorphism $\delta\colon d_{0}\to d_{1}$ in $D$ such that $pc_{0}=d_{0}$ there exists a unique isomorphism $\gamma\colon c_{0}\to c_{1}$ in $C$ such that $p\gamma=\delta$.

\begin{theorem}
The category $\Cat$ of small categories admits a cofibrantly generated model category structure, whose weak equivalences are the equivalences of categories, whose cofibrations are the functors which are injective on objects and whose fibrations are the isofibrations.
Moreover, the adjunction
\begin{equation*}
\Adjoint{\sSet}{\Cat}{\tau_{1}}{N}
\end{equation*}
defines a Quillen adjunction with the Joyal model structure.
\end{theorem}
\begin{proof}
See \cite[Theorem 3.3.10, Proposition 3.3.14]{cis16}
\end{proof}

\begin{definition}
Let $X$ be an $\infty$-category. We say that $X$ is an $\infty$-\emph{groupoid} if every morphism in $X$ is invertible in the sense of Definition \ref{def:morphism-invertible-sset}
\end{definition}

\begin{theorem}
An $\infty$-category $X$ is an $\infty$-groupoid if and only if $X$ is a Kan complex.
\end{theorem}
\begin{proof}
See \cite[Theorem 3.5.1]{cis16}
\end{proof}

\subsection{}
Let $f\colon x\to x'$ and $g\colon x\to x'$ be morphisms in an  $\infty$-category $X$.
We say that $f$ is homotopic to $g$ and write $f\sim g$ if there exists a commutative triangle with boundary $(f,1_{x'},g)$.
One can show that $\sim$ defines an equivalence relation on the set of morphisms from $x$ to $x'$, compatible with composition.

\begin{definition}
\label{def:homotopy-category-quasicategory}
Let $X$ be an $\infty$-category. The \emph{homotopy category} of $X$ is the category $\ho X$  whose objects are the $0$-simplices of $X$ and with set of morphisms $\ho X(x,x')$ from $x$ to $x'$ the set of morphisms from $x$ to $x'$ in $X$ modulo the equivalence relation $\sim$.
\end{definition}

\begin{theorem}[Boardmann-Vogt]
\label{thm:boardmann-vogt}
There is a morphism $X\to \Nerv \ho X$ which is the identity on objects and sends a morphism $f\colon x\to x'$ in $X$ to the equivalence class $[f]$ of $f$ in $\ho X(x,x')$.
Moreover this morphisms induces an isomorphism:
\[\tau_{1}X\cong \ho X.\]
\end{theorem}
\begin{proof}
See \cite[Theorem 1.5.14]{cis16}
\end{proof}

\subsection{}
Let $X$ be an $\infty$-category and let $x$ and $x'$ be $0$-simplices in $X$.
Let us consider the following pullback diagram:
\begin{equation}
\pull{X(x,x')}{\ihom(\Delta^{1},X)}{X\times X}{\Delta^{0}}{}{(s,t)}{(x,x')}{}
\end{equation}
where $s$ and $t$ are the evaluation maps at $0$ and $1$, respectively.
Then, $X(x,x')$ is the $\infty$-groupoid (see \cite[Subsection 3.7.1]{cis16}) of morphisms from $x$ to $x'$ in $X$.
Moreover, we have the following result:

\begin{proposition}
Let $X$ be an $\infty$-category, then there is a natural bijection:
\[\pi_{0}\left(X(x,x')\right)\cong \ho X(x,x').\]
\end{proposition}
\begin{proof}
See \cite[Proposition 3.7.2]{cis16}
\end{proof}

\begin{definition}
Let $f\colon X\to Y$ be a morphism of $\infty$-categories. 
We say that $f$ is \emph{fully faithful} if, for every pair of objects $x$ and $x'$ in $X$, the induced morphism:
\[X(x,x')\to Y(fx,fx')\]
is an equivalence of $\infty$-groupoids.
Moreover, we say that $f\colon X\to Y$ is \emph{essentially surjective} if for every $0$-simplex $y\in Y$ there exists a $0$-simplex $x\in X$ and an invertible morphism $fx\to y$ in $Y$.
\end{definition}

\begin{theorem}
A functor between $\infty$-categories is an equivalence of $\infty$-categories if and only if it is fully faithful and essentially surjective.
\end{theorem}
\begin{proof}
See \cite[]{cis16}.
\end{proof}

\section{Covering spaces and simplicial covers}
\label{ch:3sec:2}
Following \cite[Appendix I]{GZ67}, we give an overview of the theory of topological and simplicial coverings.
Our treatment differs from \cite{GZ67} as we do not investigate coverings of groupoids.
Instead, we define based universal covers of simplicial sets and use functor tensor products to describe the correspondence between covering spaces over a base space and representations of the associated fundamental groupoid.

\begin{definition}
\label{def:simplicial-cover}
Let $X$ be a simplicial set. 
A morphism $p\colon E\to X$ is said to be a \emph{ (simplicial) cover} over $X$ if it has the unique right lifting property with respect to the set
\[\Delta_{0}=\bigcoprod_{n\in\NN}\Delta^{n}_{0}=\{i\colon\Delta^{0}\to\Delta^{n}: [n]\in \DDelta\}\]
of vertices of all standard simplices.
We denote by  $\Cover{X}$ the full subcategory of the slice $\overcat{\sSet}{X}$ spanned by all covers over $X$.
\end{definition}

\begin{remark}
\label{rmk:simplicial-cover}
Definition \ref{def:simplicial-cover} and the Yoneda Lemma imply that a map $p\colon E\to X$ is a cover if and only if every map $[0]\to[n]$ induces a natural isomorphism:
\[E_{n}\cong E_{0}\times_{X_{0}}X_{n}.\]
In particular a morphism $E\to \Delta^{0}$ is a cover over $\Delta^{0}$ if and only if $E$ is a $0$-dimensional simplicial set.
Hence there is a natural isomorphism $\Cover{\Delta^{0}}\cong\Set$.
\end{remark}

\subsection{}
Let $\Delta$ be the set of all maps between all standard simplices:
\[\Delta=\bigcup_{n, m}\Delta^{n}_{m}=\{i\colon \Delta^{m}\to \Delta^{n}: [n], [m]\in \DDelta\}\]
and recall that $\Lambda$ is the set of all horn inclusions $\Lambda^{n}_{k}\to\Delta^{n}$. 
Then we have the following equality between classes of maps in $\sSet$:
\[\ullift{(\urlift{\Delta_{0}})}=\ullift{(\urlift{\Delta})}=\ullift{(\urlift{\Lambda})}\]
(see \cite[Lemma 54.2]{res18}).
In particular, we see that every cover is a Kan fibration.

\subsection{}
Since right orthogonal maps are stable under pullback, a morphism $f\colon X'\to X$ of simplicial sets induces a \emph{base change} functor:
\[f^{*}\colon \Cover{X}\to \Cover{X'}\]
in particular, by Remark \ref{rmk:simplicial-cover} a $0$-simplex $x\colon\Delta^{0}\to X$ of $X$ induces a \emph{fibre functor}:
\begin{align*}
\fib_{x}\colon \Cover{X} & \to \Set\\
p\colon E\to X &\mapsto \fib_{x}E 
\end{align*}

\begin{definition}
Let $p\colon E\to X$ be a morphism of simplicial sets. 
We say that $p$ is \emph{locally trivial} if for every $n$-simplex $x\colon \Delta^{n}\to X$ of $X$, there exists a simplicial set $F$ and a morphism $f\colon F\times \Delta^{n}\to E$ making the following diagram cartesian:
\begin{equation}
\pull{F\times \Delta^{n}}{E}{X}{\Delta^{n}}{f}{p}{x}{\pi_{2}}
\end{equation}
\end{definition}

\begin{proposition}
\label{prop:cover-loctriv}
A morphism $p\colon E\to X$ of simplicial sets is a cover if and only if it is a locally trivial morphism with discrete fibres.
\end{proposition}
\begin{proof}
\cite[Appendix I, 2.2]{GZ67}
\end{proof}

\subsection{}
By Corollary \ref{cor:orthogonal-factorisation-theorem}, every morphism $f\colon Y\to X$ can be factored as a composition:
\begin{equation}
\label{eq:factorisation-cover}
\begin{tikzcd}[column sep=small, row sep=small]
Y \arrow[rr, "f"] \arrow[rdd, "j"']
   & & X \\
   & & \\
& E \ar[ruu, "p"'] &
\end{tikzcd}
\end{equation}
where $p\colon E\to X$ is a cover and $j\colon Y\to X$ has the unique left lifting property with respect to any cover.
In particular, this defines a left adjoint to the inclusion functor $\Cover{X}\to \overcat{\sSet}{X}$.
\begin{equation*}
\Adjoint{\overcat{\sSet}{X}}{\Cover{X}}{}{}
\end{equation*}
which displays $\Cover{X}$ as a reflective subcategory of $\overcat{\sSet}{X}$.

\subsection{}
Let $x$ be a $0$-simplex of a simplicial set $X$. Applying the factorisation of \eqref{eq:factorisation-cover} to the morphism $\colon \Delta^{0}\to X$ yields:
\begin{diagram}
\Delta^{0}
	\ar[r, "\tilde{x}"']
	\ar[rr, bend left, "x"]&
\tilde{X}_{x}
	\ar[r,"\tilde{p}"']&
X
\end{diagram}
The cover $\tilde{p}\colon \tilde{X}_{x}\to X$ is called the \emph{universal cover} of $X$ at $x$.
Moreover, we have the following result:

\begin{lemma}
\label{lem:universal-cover-represents-fibers}
Let $X$ be a simplicial set and let $\tilde{X}_{x}$ be the universal cover of $X$ at $x$.
Then, for every cover $E$ of $X$, evaluation at $\tilde{x}$ induces an isomorphism:
\begin{equation}
\label{eq:universal-cover-represents-fibers}
\Cover{X}(\tilde{X}_{x}, E)=\fib_{x}E,
\end{equation}
natural in $E$.
\end{lemma}
\begin{proof}
This follows immediately from the definitions.
\end{proof}

\begin{definition}
Let $X$ be a simplicial set. The \emph{fundamental groupoid} $\Pi_{1}X$ of $X$ is the groupoid associated to $\tau_{1}X$ (see Example \ref{ex:core-gpd})
\[\Pi_{1}X=\gpd{\tau_{1}X}\]
\end{definition}

\begin{proposition}
\label{prop:universal-cover-fibre-functor}
The universal cover construction defines a functor:
\[\tilde{X}_{\bullet}\colon \Pi_1X^{\op}\to \Cover{X}\]
\end{proposition}
\begin{proof}
As in the proof of Proposition \ref{prop:universal-cover-functor} $\tilde{X}_{\bullet}$ defines a functor:
\[\tilde{X}_{\bullet}\colon \tau_{1}X^{\op}\to \Cover{X}.\]
To show that $\tilde{X}_{\bullet}$ defines a functor on $\Pi_{1}{X}^{\op}$ it is enough to prove that for every $1$-simplex $\alpha\colon x\to x'$ in $X$, the induced map:
\[\alpha^{*}\colon \tilde{X}_{x'}\to \tilde{X}_{x}\]
is an isomorphism (see Example \ref{ex:core-gpd}).
By Yoneda Lemma and Lemma \ref{lem:universal-cover-represents-fibers}, this is equivalent to show that for every cover $E$ over $X$, the induced map:
\[\alpha_{*}\colon \fib_{x}E\to \fib_{x'}E\]
is an isomorphism. Therefore, we are done by Proposition \ref{prop:cover-loctriv}.
\end{proof}

\subsection{}
\label{subsec:adjunctions-covers}
Since $\Cover{X}$ is complete and cocomplete, it is tensored and cotensored over $\Set$ as in \ref{ex:locally-small-tensor-cotensor}.
In particular, for every functor $F\colon \Pi_1X\to \Set$ we can form the tensor product:
\begin{align*}
F\otimes_{X}\tilde{X}_{\bullet} 	& = F\otimes_{\Pi_1X^{\op}}\tilde{X}_{\bullet}\\
						& = \int^{x\in \Pi_1X^{\op}}Fx\otimes\tilde{X}_{x}
\end{align*}
In particular we denote by
\begin{align*}
\rec \colon \Set^{\Pi_1X} 	& \to \Cover{X}\\
F						&\mapsto F\otimes_{X}\tilde{X}_{\bullet}
\end{align*}
and call it the \emph{reconstruction} functor.
On the other hand, we denote by
\begin{align*}
\fib\colon \Cover{X} & \to \Set^{\Pi_1X}\\
E &\mapsto \fib_{\bullet} E=\hom(\tilde{X}_{\bullet},E)
\end{align*}
and call it the \emph{fibre} functor. 
By Proposition \ref{prop:tensor-product-functors}, they form an adjunction:
\begin{equation}
\label{eq:fibre-reconstruction-simplicial}
\Adjoint{\Set^{\Pi_1X}}{\Cover{X}}{\rec}{\fib}
\end{equation}

\begin{corollary}
\label{cor:recfib-equiv}
The adjunction:
\begin{equation*}
\Adjoint{\Set^{\Pi_1X}}{\Cover{X}}{\rec}{\fib}
\end{equation*}
defines an adjoint equivalence of categories
\end{corollary}
\begin{proof}
This follows from \cite[Theorem 1.2]{GZ67} and \cite[Theorem 2.4.1]{GZ67}.
\end{proof}

\subsection{} 
Recall that a morphism $p\colon E\to B$ in $\Top$ is said to be \emph{locally trivial} if for every point $b\in B$ there exists an open neighbourhood $i\colon U\subset B$ of $b$ and a topological space $F_{x}$ fitting in a Cartesian square as follows:
\begin{equation}
\pull{F_{x}\times U}{E}{B}{U}{}{p}{i}{\pi_{U}}
\end{equation}
A morphism $p\colon E\to B$ is said to be a \emph{covering space} or \emph{(topological) cover} if it is locally trivial with discrete fibres.
In particular, notice that in this case the pullback can be taken in the category $\Topa$ of all topological spaces.
We let $\Cover{B}$ be the full subcategory of the slice category over $B$ spanned by covering spaces.

\subsection{}
Let $X$ be a simplicial set and let $B$ be a topological space. Then we have adjunctions:
\begin{equation}
\label{eq:sing-realisation-slice}
\Adjoint{\overcat{\sSet}{X}}{\overcat{\Top}{\abs{X}}}{\abs{\blank}}{\Sing_{X}} \quad
\Adjoint{\overcat{\sSet}{\Sing B}}{\overcat{\Top}{B}}{\abs{\blank}_{B}}{\Sing}
\end{equation}
where the functors $\abs{\blank}$ and $\Sing$ are the induced functors on the slice categories
while the functor $\Sing_{X}$ maps a morphism $f\colon Y\to \abs{X}$ to the base change
\begin{equation}
\pull{\Sing_{X}(Y)}{\Sing(Y)}{\Sing(\abs{X})}{X}{}{\Sing(f)}{\eta_{X}}{}
\end{equation}
and $\abs{\blank}_{B}$ takes a map $g\colon A\to \Sing B$ to the composite
\begin{equation}
\begin{tikzcd}
\abs{A}
\ar[r, "\abs{g}"]&
\abs{\Sing B}
\ar[r, "\epsilon_{B}"]
& B
\end{tikzcd}
\end{equation}

\begin{proposition}
\label{prop:sing-realisation-cover}
Let $X$ be a simplicial set and let $B$ be a topological space. 
Then, the adjunctions \eqref{eq:sing-realisation-slice} induce adjunctions
\begin{align}
\Adjoint{\Cover{X}}{\Cover{\abs{X}}}{\abs{\blank}}{\Sing_{X}} \label{eq:sing-realisation-cover-1}\\
\Adjoint{\Cover{\Sing B}}{\Cover{B}}{\abs{\blank}_{B}}{\Sing} \label{eq:sing-realisation-cover-2}
\end{align}
\end{proposition}
\begin{proof}
See \cite[Theorem 3.2]{GZ67}.
\end{proof}

\begin{theorem}
Let $X$ be a simplicial set. Then, the adjunction
\begin{equation*}
\Adjoint{\Cover{X}}{\Cover{\abs{X}}}{\abs{\blank}}{\Sing_{X}}
\end{equation*}
defines an equivalence between the category $\Cover{X}$ of simplicial covers over $X$ and the category $\Cover{\abs{X}}$ of topological covers over its realisation.
\end{theorem}
\begin{proof}
See \cite[Theorem 3.2.1]{GZ67}.
\end{proof}

\subsection{}
Recall that if $B$ is a topological space, we have the following natural isomorphisms
\[\Pi_{1}B=\ho\Sing B=\Pi_1 \Sing B\]
where $\Pi_{1}B$ is the \emph{fundamental groupoid} of $B$.
In particular, composing the adjunctions \eqref{eq:sing-realisation-cover-2} and \eqref{eq:fibre-reconstruction-simplicial} for $X=\Sing B$ we obtain
\begin{equation}
\label{eq:fibre-reconstruction-topological}
\Adjoint{\Set^{\Pi_{1}B}}{\Cover{B}}{\rec}{\fib}
\end{equation}

\begin{theorem}
\label{thm:cover-fiber-equivalence}
Let $B$ be a locally $0$-connected and locally $1$-connected topological space. Then, the adjunction 
\begin{equation*}
\Adjoint{\Set^{\Pi_{1}B}}{\Cover{B}}{\rec}{\fib}
\end{equation*}
is an adjoint equivalence of categories.
\end{theorem}
\begin{proof}
See \cite[Theorem 3.6]{GZ67}.
\end{proof}

\begin{corollary}
\label{cor:cover-sset-top-equivalence}
Let $B$ be a locally $0$-connected and locally $1$-connected topological space. Then, the adjunction 
\begin{equation*}
\Adjoint{\Cover{\Sing B}}{\Cover{B}}{\abs{\blank}_{B}}{\Sing}
\end{equation*}
is an adjoint equivalence of categories.
\end{corollary}
\begin{proof}
This follows from Theorem \ref{thm:cover-fiber-equivalence} and Corollary \ref{cor:recfib-equiv}.
\end{proof}

\section{Left covers and fundamental categories}
\label{ch:3sec:3}
In the first part of this section, we recall an explicit construction of the fundamental category of a simplicial set, and give a few examples.
We define left and right covers of a simplicial set $X$ and  we give a description of the universal left cover $\tilde{X}_{x}$ of $X$ at a $0$-simplex $x$, in terms of orthogonal factorisation systems.
To conclude, we use functor tensor products and universal left covers to show that the category of left covers over $X$ is equivalent to the category of representations of its fundamental category.
Although the main statements of this section can be found in \cite{res18}, we give complete proofs.

\subsection{}
\label{subsec:tau-explicit}
A more explicit description of the fundamental category functor $\tau_1\colon\sSet\to\Cat$ introduced in \ref{ex:nerve-fundamental-category} can be given as follows (see \cite[Section 1.3]{JT08}).
Let $X$ be a simplicial set and let us consider the free category $F\iota_{1}^{*}X$ on the 1-truncation on $X$. 
This is the category with objects the 0-simplices of $X$ and morphisms the paths of consecutive 1-simplices of $X$. 
Then $\tau_1 X$ is isomorphic to the quotient of $F\iota_{1}^{*}X$ by the equivalence relation generated by 
\[x\in X_{2} \quad \Rightarrow \quad d_{2}^{1}(x)\sim d_{2}^{0}(x)\circ d_{2}^{2}(x).\]

\begin{example}
\label{ex:simplicial-circle}
The quotient $S^{1}=\Delta^{1}/\partial\Delta^{1}$ is called the \emph{standard simplicial circle}.
This is the simplicial set with exactly one non degenerate 0-simplex and one non degenerate 1-simplex.
In particular, by \ref{subsec:tau-explicit}, the fundamental category of $S^{1}$ is the free category on the reflexive quiver with one object $0$ and a unique non-identity edge from $0$ to itself.
In other words, $\tau_{1}S^{1}$ is isomorphic to the monoid $\NN$ of the natural numbers, seen as a category with one object.
\end{example}

\begin{example}
\label{ex:walking-retraction-cat}
Let $R$ be the simplicial set defined in Example \ref{ex:walking-retraction-sset}. 
Then, the fundamental category of $R$ is the \emph{walking retraction category}, that we denote as $\Ret$.
The category $\Ret$ has two objects $0$ and $1$ and two non-identity morphisms $i\colon 0\to 1$ and $r\colon 1\to 0$ such that $ri=1_{0}$.
\end{example}

The following corollaries of \ref{subsec:tau-explicit} can be found in \cite[Section 1.3]{JT08} and we recall their proofs for convenience.
\begin{corollary}
\label{cor:tau-nerve-iso}
Let $C\in \Cat$ be a small category, then the counit $\tau_{1}\Nerv C\to C$ is an isomorphism.
\end{corollary}
\begin{proof}
From the explicit description of $\tau_{1}$, we have that $\tau_{1}\Nerv C$ is the free category on the underlying reflexive quiver of $C$, modulo the equivalence relation generated by $gf\sim h$ whenever $h$ is the composition of $f$ and $g$, which recovers $C$.
\end{proof}

\begin{example}
\label{ex:fundamental-category-standard-simplex}
Since $\Delta^{n}=\Nerv [n]$, the fundamental category $\tau_{1}\Delta^{n}$ of the standard $n$-simplex is the finite ordinal $[n]$, by Corollary \ref{cor:tau-nerve-iso}.
\end{example}

\begin{corollary}
\label{cor:tau-preserves-products}
The fundamental category functor $\tau_1\colon \sSet\to \Cat$ preserves products.
\end{corollary}
\begin{proof}
Since $\Cat$ is a Cartesian closed category (see Definition \ref{subsec:cartesian-closed-categories} and Example \ref{ex:cat-cartesian-closed}) and by Corollary \ref{cor:density-theorem}, it is enough to show that $\tau_{1}$ preserves products of representable presheaves. Given finite ordinals $[n]$ and $[m]$ we have
\begin{align*}
\tau_{1}\left(\Delta^{n}\times\Delta^{m}\right)
	&\cong \tau_{1}\left(\Nerv[n]\times\Nerv[m]\right)\\
	&\cong \tau_{1}\left(\Nerv\left([n]\times[m]\right)\right)\\
	&\cong [n]\times [m] \\
	&\cong \tau_{1}\Delta^{n}\times \tau_{1}\Delta^{m}
\end{align*}
where the first isomorphism follows from the definitions, the second from the fact that $\Nerv$ is a right adjoint and the third and fourth from Corollary \ref{cor:tau-nerve-iso}.
\end{proof}

\begin{definition}
\label{def:left-cover}
Let $X$ be a simplicial set. 
A morphism $p\colon E\to X$ is said to be a \emph{left cover} over $X$ if it has the unique right lifting property with respect to the set:
\[\tensor[^{l}]{\Delta}{_{0}}=\{0\colon\Delta^{0}\to\Delta^{n}: [n]\in \DDelta\}\]
of initial vertices of all standard simplices.
Dually, a morphism $p\colon E\to X$ is said to be a \emph{right cover} over $X$ if it has the unique right lifting property with respect to the set:
\[\tensor[^{r}]{\Delta}{_{0}}=\{n\colon \Delta^{0}\to \Delta^{n}: [n] \in \DDelta\}\]
of terminal vertices of all standard simplices.
We denote by  $\LCover{X}$ and by $\RCover{X}$ the full subcategories of the slice $\overcat{\sSet}{X}$ spanned by all left and right covers over $X$ respectively.
\end{definition}

\begin{remark}
\label{rmk:left-cover}
Definition \ref{def:simplicial-cover} and the Yoneda Lemma imply that a map $p\colon E\to X$ is a left cover if and only if the initial map $[0]\to[n]$ induces a natural isomorphism:
\[E_{n}\cong E_{0}\times_{X_{0}}X_{n}.\]
In particular a morphism $E\to \Delta^{0}$ is a left cover over $\Delta^{0}$ if and only if $E$ is a $0$-dimensional simplicial set.
Hence there is a natural isomorphism $\LCover{\Delta^{0}}\cong\Set$.
\end{remark}

\subsection{}
Let $\tensor[^{l}]{\Delta}{}$ and $\tensor[^{r}]{\Delta}{}$ be defined as follows:
\begin{align*}
\tensor[^{l}]{\Delta}{} = \bigcup_{n, m}\tensor[^{l}]{\Delta}{^{n}_{m}} &  = \{i\colon \Delta^{m}\to \Delta^{n}: [n], [m]\in \DDelta, \quad i(0)=0\}\\
\tensor[^{r}]{\Delta}{} = \bigcup_{n, m}\tensor[^{r}]{\Delta}{^{n}_{m}} & = \{i\colon \Delta^{m}\to \Delta^{n}: [n], [m]\in \DDelta, \quad i(m)=n\}
\end{align*}
Moreover, let $\tensor[^{l}]{\Lambda}{}$ and $\tensor[^{r}]{\Lambda}{}$ be the sets of left and right horn inclusions  respectively:
\begin{align*}
\tensor[^{l}]{\Lambda}{} &=\{\Lambda^{n}_{k}\subset \Delta^{n}: [n]\in \DDelta, 0\le k<n\}\\
\tensor[^{r}]{\Lambda}{} &=\{\Lambda^{n}_{k}\subset \Delta^{n}: [n]\in \DDelta, 0<k \le n\}
\end{align*}
Then we have the following equalities between classes of maps in $\sSet$:
\begin{align}
\ullift{\left(\urlift{\tensor[^{l}]{\Delta}{_{0}}}\right)}&=\ullift{\left(\urlift{\tensor[^{l}]{\Delta}{}}\right)}=\ullift{\left(\urlift{\tensor[^{l}]{\Lambda}{}}\right)}\label{eq:left-cover-saturation}\\
\ullift{\left(\urlift{\tensor[^{r}]{\Delta}{_{0}}}\right)}&=\ullift{\left(\urlift{\tensor[^{r}]{\Delta}{}}\right)}=\ullift{\left(\urlift{\tensor[^{r}]{\Lambda}{}}\right)}\label{eq:right-cover-saturation}
\end{align}
(see \cite[Lemma 54.6]{res18}). 

\subsection{}
Since right orthogonal maps are stable under pullback, a morphism $f\colon X'\to X$ of simplicial sets induces a \emph{base change} functor:
\[f^{*}\colon \LCover{X}\to \LCover{X'}\]
in particular, by Remark \ref{rmk:left-cover} a $0$-simplex $x\colon\Delta^{0}\to X$ of $X$ induces a \emph{fibre functor}:
\begin{align*}
\fib_{x}\colon \LCover{X} & \to \Set\\
p\colon E\to X &\mapsto \fib_{x}E 
\end{align*}

\subsection{}
By Corollary \ref{cor:orthogonal-factorisation-theorem}, every morphism $f\colon Y\to X$ can be factored as a composition:
\begin{equation}
\label{eq:factorisation-left-cover}
\begin{tikzcd}[column sep=small, row sep=small]
Y \arrow[rr, "f"] \arrow[rdd, "j"']
   & & X \\
   & & \\
& E \ar[ruu, "p"'] &
\end{tikzcd}
\end{equation}
where $p\colon E\to X$ is a left cover and $j\colon Y\to E$ has the unique left lifting property with respect to any left cover.
In particular, this defines a left adjoint to the forgetful functor $\LCover{X}\to \overcat{\sSet}{X}$
\begin{equation*}
\Adjoint{\overcat{\sSet}{X}}{\LCover{X}}{}{}
\end{equation*}
which displays $\LCover{X}$ as a reflective full subcategory of $\overcat{\sSet}{X}$.

\subsection{}
Let $x$ be a $0$-simplex of a simplicial set $X$, applying the factorisation \eqref{eq:factorisation-left-cover} to the morphism $x\colon \Delta^{0}\to X$ yields a diagram:
\begin{diagram}
\Delta^{0}
	\ar[r, "\tilde{x}"']
	\ar[rr, bend left, "x"]&
\tilde{X}_{x}
	\ar[r,"\tilde{p}"']&
X
\end{diagram}
The morphism $\tilde{p}\colon \tilde{X}_{x}\to X$ is called the \emph{universal left cover} of $X$ at $x$.

\begin{lemma}
\label{lem:universal-left-cover-represents-fibers}
Let $X$ be a simplicial set and let $\tilde{X}_{x}$ be the universal left cover of $X$ at $x$.
Then, for every left cover $E$ of $X$, evaluation at $\tilde{x}$ induces an isomorphism:
\begin{equation}
\label{eq:universal-left-cover-represents-fibers-sset}
\LCover{X}(\tilde{X}_{x}, E)=\fib_{x}E,
\end{equation}
natural in $E$.
\end{lemma}
\begin{proof}
This follows immediately from the definitions.
\end{proof}

\begin{example}
\label{ex:left-cover-standard-simplex}
The universal left cover of $\Delta^{n}$ at a vertex $k$ is the simplicial set $\Delta^{n-k}$ based at its first vertex, together with the map $\epsilon\colon \Delta^{n-k}\to \Delta^{k}$ uniquely determined by the inclusion:
\begin{align*}
\epsilon\colon [n-k]&\to [n]\\
j&\mapsto k+j
\end{align*}
Indeed, by definition the first vertex of $\Delta^{n-k}$ determines a map with the left lifting property with respect to all left covers.
Moreover, a commutative diagram:
\begin{equation}
\csquare{\Delta^{0}}{\Delta^{n-k}}{\Delta^{n}}{\Delta^{m}}{0}{\epsilon}{\alpha}{0}
\end{equation}
is equivalent to a morphism $\alpha\colon [m]\to[n]$ such that $\alpha(0)=k$.
To conclude, since $\alpha$ is weakly monotonic and $\epsilon$ is an injection, $\alpha$ factors through $\epsilon$, which proves the claim.
\end{example}

\begin{example}
\label{ex:infinity-spine}
We define the $\infty$-\emph{spine} $\Sp^{\infty}$ as the colimit of the sequence:
\begin{equation}
\label{eq:filtration-spine}
\Sp^{1}\subset \Sp^{2}\subset \ldots \subset \Sp^{n}\subset \ldots
\end{equation}
where the inclusion $\Sp^{n-1}\subset \Sp^{n}$ is induced by the bottom face map $\partial^{n-1}_{n-1}\colon \Delta^{n-1}\to \Delta^{n}$.
Notice that, for every $n$, we can define a morphism of simplicial sets:
\[e_{n}\colon \Sp^{n}\to S^{1}\]
uniquely determined by mapping every $0$-simplex of $\Sp^{n}$ to the unique $0$-simplex of the circle, and every $1$-simplex of $\Sp^{n}$ to the unique non degenerate $1$-simplex of $S^{1}$.
Moreover, the maps $e_{n}$ are compatible with the filtration \eqref{eq:filtration-spine} so that they determine a unique morphism:
\[e\colon \Sp^{\infty}\to S^{1}.\]
We claim that $e$ is the universal left cover of $S^{1}$ at $0$, based at its first $0$-simplex.
To show that $e$ is a left cover, since $S^{1}$ and $\Sp^{\infty}$ are $1$-dimensional simplicial sets, it is enough to check that $e$ has the right lifting property with respect to the map $0\colon \Delta^{0}\to \Delta^{1}$, which is immediate from the definitions.
Moreover, the inclusion of the first vertex $0\colon \Delta^{0}\to \Sp^{\infty}$ is manifestly a transfinite composition of pushouts of the map $0\colon \Delta^{0}\to \Delta^{1}$.
\end{example}

\begin{construction}
Let $X$ be a simplicial set and let $x$ be a $0$-simplex of $X$.
We construct a simplicial set $\overline{X}_{x}$ as follows. 
The $n$-simplices of $\overline{X}_{x}$ are given by:
\[\left(\overline{X}_{x}\right)_{n}= \bigcoprod_{\sigma\in X_{n}}\tau_{1}X(x,\sigma_{0}).\]
For every morphism $\delta\colon [m]\to[n]$, the evaluation of $\overline{X}_{x}$ at $\delta$ is defined as:
\begin{align*}
\bigcoprod_{\sigma\in X_{n}}\tau_{1}X(x,\sigma_{0}) &\to \bigcoprod_{\beta\in X_{m}}\tau_{1}X(x,\beta_{0})\\
(\sigma,\alpha)& \mapsto \left(\delta^{*}\sigma, \sigma_{0,\delta_{0}}\alpha\right)
\end{align*}
where $\sigma_{{0},\delta_{0}}$ denotes the unique morphism in $\tau_{1}X$ induced by the composition of $\sigma$ with the unique $1$-simplex in $\Delta^{n}$ from the first vertex to $\delta_{0}$.
Moreover $\overline{X}_{x}$ comes equipped with a canonical projection $p\colon\overline{X}_{x}\to X$ and a  basepoint $1_{x}\in \left(\overline{X}_{x}\right)_{0}$ corresponding to the identity on $x$.
\end{construction}

\begin{lemma}
\label{lem:universal-left-cover-representable}
Let $X$ be a simplicial set and let $x$ be a $0$-simplex of $X$.
Then, there exists a basepoint preserving isomorphism:
\[\tilde{X}_{x}\cong\overline{X}_{x}\]
In particular, for every $[n]\in \DDelta$, there exists a canonical isomorphism:
\begin{equation}
\label{eq:explicit-left-cover}
\left(\tilde{X}_{x}\right)_{n}\cong\bigcoprod_{\sigma\in X_{n}}\tau_{1}X(x,\sigma_{0}).
\end{equation}
\end{lemma}
\begin{proof}
By Lemma \ref{lem:universal-left-cover-represents-fibers}, it is enough to show that $p\colon \overline{X}_{x}\to X$ is a left cover and there is a natural isomorphism:
\begin{equation}
\label{eq:lc-u-property}
\LCover{X}\left(\overline{X}_{x},E\right)\cong \fib_{x}E
\end{equation}
for every cover $E$ over $X$.
To show that $p\colon \overline{X}_{x}\to X$ is a left cover notice that, by construction, we have a natural isomorphism:
\[\left(\overline{X}_{x}\right)_{n}\cong \bigcoprod_{\sigma\in X_{n}}\tau_{1}X(x,\sigma_{0})\cong\left(\overline{X}_{x}\right)_{0}\times_{X_{0}}X_{n}\]
Let $q\colon E\to X$ be a left cover over $X$, to show that \eqref{eq:lc-u-property} holds we need to prove that a morphism $\tilde{X}_{x}\to E$ over $X$ is uniquely determined by the image of $1_{x}$.
Let $e$ be a point in the fibre of $E$ at $x$ and let $(\sigma, \alpha\colon x\to \sigma_{0})$ be an $n$-simplex of $\overline{X}_{x}$.
Then, $\alpha$ defines a morphism of fibres:
\[\alpha_{*}\colon \fib_{E}x\to \fib_{E}\sigma_{0}\]
hence, taking the pair $(\sigma, \alpha_{*}e)$ determines a unique $n$-simplex of $E$ by Remark \ref{rmk:left-cover}.
By construction, this assignment defines a unique morphism of simplicial sets $f\colon \overline{X}_{x}\to E$ with the desired property.
\end{proof}

\begin{proposition}
\label{prop:universal-cover-functor}
The universal cover construction defines a functor:
\[\tilde{X}_{\bullet}\colon \tau_1X^{\op}\to \LCover{X}.\]
\end{proposition}
\begin{proof}
By Lemma \ref{lem:universal-left-cover-represents-fibers}, for every two $0$-simplices $x$ and $x'$ in $X$ we have a natural isomorphism:
\[\LCover{X}(\tilde{X}_{x'},\tilde{X}_{x})\cong \fib_{x'}\tilde{X}_{x}.\]
By the unique lifting property of left covers, given a morphism $\alpha\colon y\to y'$ in $X$, and a point $\tilde{y}$ in the fibre of $\tilde{X}_{x}$ at $y$, there exists a unique morphism $\tilde{\alpha}\colon \tilde{y}\to\tilde{\alpha}(1)$ lifting $\alpha$.
In particular, $\tilde{X}_{x}$ defines a functor from the opposite of the free category at the $1$-truncation of $X$ to the category of left covers over $X$.
To conclude the proof, by \ref{subsec:tau-explicit}, we need to show that for every commutative triangle $(\alpha,\beta,\gamma)$ in $X$ of the form
\begin{equation}
\TriangleOver{x}{x'}{x''}{\alpha}{\beta}{\gamma}
\end{equation}
and lifts $\tilde{\alpha}\colon \tilde{x}\to \tilde{\alpha}(1)$, $\tilde{\beta}\colon \tilde{\alpha}(1)\to \tilde{\beta}(1)$ and $\tilde{\gamma}\colon \tilde{x}\to \tilde{\gamma}(1)$ in $\tilde{X}$, the points $\tilde{\gamma}(1)$ and $\tilde{\beta}(1)$ coincide in the fibre of $\tilde{X}_{x}$ at $x''$.
Consider the following lifting problems
\begin{equation}
\lift{\Lambda^{2}_{1}}{\tilde{X}_{x}}{X}{\Delta^{2}}{(\tilde{\alpha}, \tilde{\beta},\blank)}{}{}{}{h}
\quad
\lift{\Lambda^{2}_{0}}{\tilde{X}_{x}}{X}{\Delta^{2}}{(\tilde{\alpha},\blank,\tilde{\gamma})}{}{}{}{k}
\end{equation}
Then, both problems have a unique lift by \eqref{eq:left-cover-saturation}.
To conclude, notice that both $k$ and $h$ are lifts for the diagram
\begin{equation}
\lift{\Delta^{0}}{\tilde{X}_{x}}{X}{\Delta^{2}}{\tilde{x}}{}{}{0}{}
\end{equation}
which implies that $h=k$ by uniqueness of lifts, hence $\tilde{\gamma}(1)=\tilde{\beta}(1)$.
\end{proof}

\begin{corollary}
The functor $\tilde{X}_{\bullet}\colon \tau_{1}X^{\op}\to \Set$ is fully faithful.
\end{corollary}
\begin{proof}
By Lemma \ref{lem:universal-left-cover-represents-fibers} evaluation at the basepoint induces a natural isomorphism:
\[\LCover{X}\left(\tilde{X}_{x'}, \tilde{X}_{x}\right)\cong \fib_{x'}\tilde{X}_{x}\]
Moreover, by Lemma \ref{lem:universal-left-cover-representable} there is a natural isomorphism:
\[\left(\tilde{X}_{x}\right)_{0}\cong \bigcoprod_{x'\in X_{0}}\tau_{1}X(x,x').\]
We conclude that:
\[\tau_{1}X(x,x')\cong \LCover{X}\left(\tilde{X}_{x'},\tilde{X}_{x}\right).\]
\end{proof}

\subsection{}
Since $\LCover{X}$ is complete and cocomplete, it is tensored and cotensored over $\Set$ as in \ref{ex:locally-small-tensor-cotensor}.
Therefore, for every functor $F\colon \tau_1 X \to \Set$ we can form the tensor product (see Definition \ref{def:functor-tensor-product}):
\begin{align*}
F\otimes_{X}\tilde{X}_{\bullet} & = F\otimes_{\tau_1 X^{\op}}\tilde{X}_{\bullet}\\
	& \int^{x\in \tau_1 X^{\op}}Fx\otimes\tilde{X}_{x}
\end{align*}
In particular we obtain a functor:
\begin{align*}
\rec \colon \Set^{\tau_1 X} 	& \to \LCover{X}\\
F					& \mapsto F\otimes_{X}\tilde{X}_{\bullet}
\end{align*}
that we call the \emph{left reconstruction functor}.
Moreover, by Proposition \ref{prop:tensor-product-functors}, $\rec$ has a right adjoint defined as:
\begin{align*}
\fib\colon \LCover{X} & \to \Set^{\tau_1 X}\\
E &\mapsto \fib_{\bullet} E=\hom(\tilde{X}_{\bullet},E)
\end{align*}
that we call the \emph{left fibre functor}.

\begin{proposition}
\label{prop:concrete-left-cover-functor}
Let $F\colon \tau_{1}X\to \Set$ be a functor. Then, for every $[n]\in \DDelta$, there exists a natural isomorphism:
\[\left(F\otimes \tilde{X}_{\bullet}\right)_{n}\cong \bigcoprod_{\alpha\in X_{n}}F(\alpha_{0})\]
where $\alpha_{0}$ denotes the first vertex of $\alpha$.
\end{proposition}
\begin{proof}
By Corollary \ref{cor:density-theorem} we have a natural isomorphism:
\[F\cong \colim_{\overcat{\tau_{1}X}{F}}\tau_{1}X(x,\blank).\]
Applying $\blank\otimes_{X}\tilde{X}_{\bullet}$ and taking $n$-simplices we obtain:
\begin{align*}
\left(F\otimes_{X}\tilde{X}_{\bullet}\right)_{n}&\cong \colim_{\overcat{\tau_{1}X}{F}}\left(\tilde{X}_{x}\right)_{n}\\
	&\cong\colim_{\overcat{\tau_{1}X}{F}}\bigcoprod_{\alpha\in X_{n}}\tau_{1}X(x,\alpha_{0})\\
	&\cong \bigcoprod_{\alpha\in X_{n}}F(\alpha_{0})
\end{align*}
where the first isomorphism follows from Corollary \ref{cor:density-theorem} and the second isomorphism follows from Lemma \ref{lem:universal-left-cover-representable}.
\end{proof}

\begin{theorem}
\label{thm:left-covers-fiber-functors}
The adjunction:
\begin{equation*}
\Adjoint{\Set^{\tau_1 X}}{\LCover{X}}{\rec}{\fib}
\end{equation*}
is an equivalence of categories.
\end{theorem}
\begin{proof}
Let $F$ be a functor. 
Then, by Proposition \ref{prop:concrete-left-cover-functor} we have an isomorphism:
\[Fx\cong \fib_{x}\rec{F},\]
which is natural in $x$ by the explicit description of $\overline{X}_{x}$ in Lemma \ref{lem:universal-left-cover-representable}
On the other hand, given a left cover $E$ of $X$, for every $[n]\in \DDelta$ we have a chain of isomorphisms:
\begin{align*}
\left(\rec\fib E\right)_{n} &= \bigcoprod_{\alpha\in X_{n}}\fib_{\alpha_{0}}E\\
	&\cong E_{0}\times_{X_{0}}X_{n}\\
	&\cong E_{n}
\end{align*}
natural in $[n]$.
\end{proof}
\chapter{Stratified spaces}
\label{ch:4}

\section{Topology and preorder\label{ch:4sec:1}}
Every preordered set $P$ can be endowed with a topology called the \Alexandroff{} topology of $P$.
This assignment defines an equivalence of categories between the category of preordered sets and the category of finitely generated topological spaces (\ref{subsec:equivalence-finitely-generated-preorder})
and we recall the proof for convenience in Proposition \ref{prop:finitely-generated-space}.
We conclude the section with a recollection on preordered topological spaces (Definition \ref{def:preordered-space}) and with a characterisation of the exponentiable preordered spaces (Proposition \ref{prop:ptop-corecompact-exponentiable}).

\subsection{}
Recall that a \emph{relation} $R$ is a set $S$, called the \emph{domain} of $R$, together with a subset $\graph(R)$ of the cartesian product $S\times S$, called the \emph{graph} of $R$.
If $R$ is a relation with domain $S$ we will say that $R$ is a relation on $S$ and write $s\mathrel{R}s'$ whenever the pair $(s,s')$ belongs to the graph of $R$.
A morphism of relations $f\colon R\to R'$ is a function $f\colon S\to S'$ from the domain of $R$ to the domain of $R'$, such that $s\mathrel{R}s'$ implies $fs\mathrel{R'}fs'$.
We denote by $\Rel$ the category of relations and morphisms between them.

\begin{definition}
\label{def:preorder}
A \emph{preordered set}, or \emph{preorder}, is a set $P$ together with a reflexive and transitive relation $\le$ on $P$. 
The category $\Pre$ of preorders is the full subcategory of $\Rel$ spanned by the preordered sets.
A morphism $f\colon P\to P'$ of preorders is also called a \emph{weakly monotonic map}.
\end{definition}

\begin{definition}
Let $R$ be a relation on a set $S$. 
The \emph{reflexive-transitive closure} of $R$ in $S$ is the preorder relation on $S$ with smallest graph containing $\graph(R)$.
\end{definition}

\begin{construction}
\label{con:reflexive-transitive}
Let $R$ be a relation on a set $S$.
We construct a relation $R^{\infty}$ on $S$ as follows.
Let $R^{0}=R$ and $R^{1}$ be the relation on $S$ with graph:
\[\graph(R^{1})=\graph(R)\cup \{(s,s):s\in S\}\] 
For every $n\ge 2$ we define the relation $R^{n}$ on $S$ with graph:
\[\graph(R^{n})=\left\{(s,s'): \exists s''\in S, s \mathrel{R^{n-1}}s'' \mathrel{R}s'\right\}.\]
Then, we define $R^{\infty}$ to be the relation on $S$ with graph:
\[\graph(R^{\infty})=\bigcup_{n\in \NN}\graph(R^{n}).\]
\end{construction}

\begin{lemma}
\label{lem:reflexive-transitive}
Let $R$ be a relation on a set $S$. Then, the the relation $R^{\infty}$ constructed in \ref{con:reflexive-transitive} is the reflexive-transitive closure of $R$ in $S$.
\end{lemma}
\begin{proof}
The relation $R^{\infty}$ is reflexive and transitive by construction.
Let $\le$ be a preoder relation on $S$ whose graph contains the graph of $R$.
Then, if $(s,s')$ is a pair in $R^{\infty}$, there exists an $n+1$-tuple $(s_{0},\ldots, s_{n})$ of elements of $S$ with $s_{0}=s$ and $s_{n}=s'$, such that:
\[s_{0}\mathrel{R}s_{1}\mathrel{R}\ldots \mathrel{R}s_{n}.\]
In particular, since $\le$ is transitive and it contains $R$, the pair $(s_{0},s_{n})$ belongs to the graph of $\le$.
\end{proof}

\subsection{}
\label{subsec:final-initial-preorder}
One can show that the forgetful functor $U\colon \Pre\to \Set$, mapping a preorder to its underlying set, is a topological functor (see \cite[Example 3.1 (1)]{faj08}).
Given a family of preorders $(P_{i},\le_{i})$ and maps $(f_{i}\colon S\to P_{i})_{i\in I}$, the initial structure on $S$ is the preorder relation $\bigwedge_{i\in I}\le_{i}$ with graph: 
\[\graph\left(\bigwedge_{i\in I}\le_{i}\right)=\bigcap_{i\in I}\left(f_{i}\times f_{i}\right)^{-1}\left(\graph(\le_{i})\right).\]
Dually, the final structure on $S$ with respect to a family of maps $(f_{i}\colon P_{i}\to S)$ is the relation $\bigvee_{i\in I}\le_{i}$ given by the reflexive-transitive closure of the relation on $S$ with graph:
\[\bigcup_{i\in I}(f_{i}\times f_{i})(\graph(\le_{i}))\]
By Lemma \ref{lem:reflexive-transitive}, $s\bigvee_{i\in I}\le_{i} s'$ in $S$ if there exists an $n$-tuple $(i_{1},\ldots, i_{n})$ of elements in $I$ and a set $\{p_{j}^{\epsilon}: p_{j}^{\epsilon}\in P_{i_{j}}, \epsilon\in\{0,1\}\}$ such that
\[p_{j}^{0}\le_{i_{j}} p_{j}^{1}\quad f_{i_{j}}\left(p_{j}^{1}\right)=f_{i_{j+1}}\left(p_{j}^{0}\right)\quad f_{i_{1}}\left(p_{1}^{0}\right)=s,\; f\left(p_{n}^{1}\right)=s'.\]

\subsection{}
The right adjoint to the forgetful functor $U\colon \Pre\to \Set$ maps a set $S$ to the \emph{chaotic preorder} $\ind{S}$ on $S$, whose graph is the product $S\times S$.
On the other hand, the left adjoint to $U$ maps a set $S$ to the \emph{discrete preorder} $\disc{S}$ on $S$, whose graph is the image of the diagonal morphism $S\to S\times S$.

\subsection{} 
Every preorder $P$ gives rise to a category with the underlying set of $P$ as a set of objects and a unique morphism from $p$ to $p'$ if and only if $p\le p'$. 
This assignment defines a fully faithful embedding $\iota\colon \Pre\to\Cat$ of the category of preorders in the category of small categories.
A category in the essential image of $\iota$ is called a \emph{thin category}.

\subsection{}
\label{subsec:pre-cartesian-closed}
The inclusion functor $\iota\colon \Pre\to\Cat$ has a left adjoint which assigns to a small category $\cat{A}$ the preorder on the set of objects of $\cat{A}$ given by $a\le a'$ if and only if there exists a morphism from $a$ to $a'$. 
Such a functor manifestly preserves products, which implies that the category $\Pre$ is Cartesian closed.
Explicitly, if $P$ and $P'$ are preorders, the category ${P'}^{P}$ of functors from $P$ to $P'$ is a thin category, since a natural transformation $\eta\colon f\to g$ between functors $f, g\colon P\to P'$ simply manifests that $fp\le gp$ for every $p\in P$.

\begin{notation}
\label{not:preord}
Let $P$ be a preorder and let $p$ be a element of $P$, we use the following notation:
\begin{align*}
P_{\le p}	&= \{p'\in P: p'\le p\}	\\
P_{\ge p}  &= \{p'\in P: p'\ge p\}	\\
P_{< p}	&= \{p'\in P: p'< p\}	\\
P_{> p}	&= \{p'\in P: p'> p\}	\\
P_{p}	&= \{p'\in P: p'\le p \le p'\}.
\end{align*}
We call $P_{\le p}$ the \emph{sieve} generated by $p$ and $P_{\ge p}$ the \emph{cosieve} generated by $p$.
Two elements $p$ and $p'$ of a preorder $P$ are said to be \emph{chaotically equivalent} if $p\le p'\le p$ and we write $p\sim p'$.
If $p$ is an element of $P$, we say that $P_{p}$ is the \emph{chaotic orbit} of $P$ at $p$.
\end{notation}

\begin{definition}
\label{def:partial-order}
A \emph{partially ordered set} or \emph{poset} is a preordered set $(P,\le)$ such that $\le$ is anti-symmetric.
We denote by $\Pos$ the full subcategory of $\Pre$ spanned by the partially ordered sets.
\end{definition}

\subsection{}
\label{subsec:poset-preord-left-adj}
The category $\Pos$ of posets is a coreflective full subcategory of $\Pre$.
The left adjoint $\phi\colon \Pre\to \Pos$  to the inclusion functor sends a preorder $P$ to the quotient 
\[\phi(P)=P/\sim=\{P_{p}:p\in P\}\]
of $P$ by the chaotic equivalence relation on $P$.

\begin{definition}
If $P$ is a preorder, the \emph{\Alexandroff{} topology} on $P$ is defined by declaring a subset $U$ open in $P$ if and only if $U$ is  upward closed in $P$ \ie if $u\in U$ and $p\in P$ with $u\le p$, then $p$ is in $U$.
This defines a functor
	\[ A\colon \Pre\to \Topa\]
which comes equipped with a right adjoint 
	\[S\colon \Topa\to \Pre\]
mapping a topological space $X$ to the \emph{specialisation preorder} on the set of points of $X$, given by $x\le x'$ if and only if $x$ is in the closure of $\{x'\}$.
We will still denote by $P$ the \Alexandroff{} space associated to a preorder $P$, omitting $A$ from the notation.
\end{definition}

\begin{example}
Every finite ordinal $[n]$ (see \ref{subsec:simplicial-sets}) is in particular a poset. The \Alexandroff{} space associated to $[n]$ has open subsets the subchains $[n]_{\ge k}$.
Notice that the \Alexandroff{} space associated to $[1]$ is the topological space with two points $\{0,1\}$ such that $\{1\}$ is open in $[1]$ and $\{0\}$ is closed.
In other words, $[1]$ is the \emph{\Sierpinski{}} space.
\end{example}

\begin{lemma}
Let $P$ be a preorder, then a basis for the \Alexandroff{} topology on $P$ is given by the set of all cosieves of $P$.
\end{lemma}
\begin{proof}
The proof is immediate from the definitions.
\end{proof}

\subsection{} 
Recall that a \emph{finite topological space} is a topological space $F$ with a finite set of points.
Let $\FinTop$ be the full subcategory of $\Topa$ spanned by the finite topological spaces.
A topological space $X$ is said to be \emph{finitely generated}, if it is $\FinTop$-generated in the sense of Definition \ref{def:I-generated}.
We denote by $\fgTop$ the full subcategory of $\Topa$ spanned by the finitely generated spaces.

\begin{proposition}
\label{prop:finitely-generated-space}
Let $X$ be a topological space. The following conditions are equivalent:
\begin{enumerate}
\item
The space $X$ is finitely generated.
\item
Arbitrary intersections of open subsets of $X$ are open in $X$.
\item
The topology on $X$ coincides with the \Alexandroff{} topology associated to the specialisation preorder on $X$.
\end{enumerate}
\end{proposition}
\begin{proof}
Assume that $X$ is finitely generated and let $\{U_i\}_{i\in I}$ be a family of open subsets of $X$. For every finite topological space $F$ and any map $f\colon F\to X$, we have that $\{f^{-1}(U_i)\}_{i\in I}$ is a (finite) family of open subsets in $F$, so that 
\[f^{-1}\left(\bigcap_{i\in I}U_i\right)=\bigcap_{i\in I}f^{-1}\left(U_i\right)\]
is open in $F$. Hence (1) implies (2).
Assume now that $X$ satisfies (2) and let $U$ be a subset of $X$ upward closed with respect to the specialisation preorder, then the following holds:
\begin{equation*}
\label{eq:finitely-gen-proposition}
\bigcap_{y\notin U} \left(X\setminus\overline{\{y\}}\right)=U.
\end{equation*}
In particular, $U$ is open in $X$.
To conclude, assume that (3) holds and that $U$ is not open in $X$. 
Then, there exist $x$ and $x'$ in $X$ such that $x$ is in $U$ and $x$ belongs to the closure of $\{x'\}$ in $X$, but $x'$ is not in $U$.
In particular, we can define a morphism $f\colon [1] \to X$  from the \Sierpinski{} space to $X$ mapping the closed point to $x$ and the open point to $x'$ and we have that $f^{-1}(U)=\{0\}$ is not open in $S$, which proves the claim.
\end{proof}

\subsection{} 
\label{subsec:equivalence-finitely-generated-preorder}
By Proposition \ref{prop:finitely-generated-space}, the \Alexandroff{} topology defines an equivalence between the category of preorders and the category of finitely generated spaces. 
\begin{equation}
\Adjoint{\Pre}{\fgTop}{A}{S}
\end{equation}
In particular, it restricts to an equivalence between the category of finite preorders and the category of finite topological spaces.
Moreover, a preorder $P$ is a poset if and only if the associated topological space is a $T_{0}$-finitely generated space.

\begin{remark}
The Alexandroff space associated to any preorder $P$ is a numerically generated space (see Example \ref{ex:numerically-generated}).
Indeed, every preorder $P$ can be written as a colimit of $[0]$, $[1]$ and $\ind{2}$, where $\ind{2}$ denotes the chaotic preorder on the set with two elements.
Moreover, the \Alexandroff{} spaces associated to $[0]$, $[1]$ and $\ind{2}$ are numerically generated, since they can be written as colimits of points and intervals.
Therefore, by Proposition \ref{prop:C_I-coreflective} we conclude that every preorder defines a numerically generated space.
\end{remark}

\begin{definition}
\label{def:preordered-space}
A \emph{preordered space} is a pair $(X,\le)$ where $X$ is a topological space and $\le$ is a preorder on the underlying set of $X$.
A morphism $f\colon (X,\le)\to (X',\le')$ of preordered spaces is a continuous map $f\colon X\to X'$ such that $f$ is monotonic with respect to the given preorders.
When no confusion will arise, a preordered space $(X,\le)$ will be denoted simply by $X$, omitting the preorder from the notation.
We denote by $\pTop$ the category of preordered spaces and morphisms between them.
\end{definition}

\subsection{}
\label{subsec:ptop-pullback-top-pre}
The category $\pTop$ of preordered spaces can be described categorically as the (strict) pullback:
\begin{equation}
\label{eq:ptop-pullback}
\pull{\pTop}{\Topa}{\Set}{\Pre}{}{}{}{}
\end{equation}
of the forgetful functors $\Pre\to \Set$ and $\Topa\to \Set$.
Indeed, to give a category $\cat{C}$ together with functors $F\colon \cat{C}\to \Topa$ and $G\colon \cat{C}\to \Pre$ making the following square commutative:
\begin{equation*}
\csquare{\cat{C}}{\Topa}{\Set}{\Pre}{F}{}{}{G}
\end{equation*}
is the same as giving a functor $F\colon \cat{C}\to \Set$ such that, for every $c\in \cat{C}$, the set $Fc$ is equipped with a topology and a preorder and for every morphism $f\colon c\to c'$ in $\cat{C}$, the function $Ff$ is continuous and weakly monotonic.
This is exactly the datum of a functor from $\cat{C}$ to the category of preordered spaces.
In particular, $\pTop$ is a well fibered topological construct, since the categories $\Pre$ and $\Topa$ are.
Moreover, all the functors involved in the Cartesian square \eqref{eq:ptop-pullback} are topological.

\begin{example}
Let $X$ be a topological space.
When no confusion arises, we denote the chaotic preordered space associated to $X$ simply as $X$.
In particular, we denote by $\RR$ and $\II$ the real line and the unit interval equipped with the chaotic preorder.
\end{example}

\begin{example}
The real line $\RR$ equipped with the standard order $\le$ is called the \emph{standard ordered real line}.
Similarly $(\II,\le)$ is the interval equipped with the order induced by the standard order on $\RR$.
\end{example}

\subsection{}
\label{subsec:alex-preord-space}
The \Alexandroff{} topology functor $A\colon \Pre\to \Topa$ induces a functor $(1_{\Pre},A)\colon \Pre \to \pTop$ mapping a preorder $P$ to the \Alexandroff{} space of $P$ equipped with its own preorder.
As for the functor $A$ itself, we will blur the distinction between $P$ viewed as a preorder and its associated preordered space.

\begin{construction}
\label{con:exponentiable-preordered-space}
Let $(A,\le_{A})$ be a preordered space whose underlying space $A$ is core-compact (see \ref{subsec:corecompact-exponentiable}) and let $(X,\le_{X})$ be a preordered space.
The set $\pTop\left(A,X\right)$ of weakly monotonic continuous maps is a subset of the topological space $X^{A}$ of continuous maps from $A$ to $X$ and as such can be endowed with the subspace topology.
Moreover, we can define a preorder $\le^{A}_{X}$ on the points of $\pTop\left(A, X\right)$ by setting $f\le^{A}_{X}g$ if and only if $fa\le_{X}ga$ for every $a\in A$.
Let $\left(\pTop\left(A, X\right), \le^{A}_{X}\right)$ be the preordered space so obtained.
\end{construction}

\begin{proposition}
\label{prop:ptop-corecompact-exponentiable}
A preordered space $(A,\le_{A})$ is exponentiable if and only if $A$ is a core-compact topological space.
Moreover, for every preordered space $(X,\le_{X})$ the internal-hom from $(A,\le_{A})$ to $(X,\le_{X})$ is given by the preordered space $\left(\pTop\left(A,X\right), \le^{A}_{X}\right)$.
\end{proposition}
\begin{proof}
Let $(A,\le_{A})$ be an exponentiable preordered space and let us prove that $A$ is core-compact.
Let $X$ be a topological space and let $\ind{X}$ be the chaotically preordered space on $X$.
Then, the underlying set of the internal-hom $\ihom\left((A,\le_{A}),\ind{X}\right)$ is the set of continuous maps $\Topa(A,X)$ from $A$ to $X$.
Moreover, for every topological space $Y$, taking the discrete preorder on $Y$ we get the following chain of isomorphisms:
\begin{align*}
\Topa(Y\times A, X)
	& \cong \pTop\left(\disc{Y}\times (A,\le_{A}), \ind{X}\right)\\
	& \cong \pTop\left(\disc{Y},\ihom\left((A,\le_{A}),\ind{X}\right)\right)\\
	&\cong \Topa\left(Y,U\left(\ihom\left((A,\le_{A}),\ind{X}\right)\right)\right)
\end{align*}
where, in the last line, the functor $U$ denotes the forgetful from $\pTop$ to $\Topa$.
In particular, $U\left(\ihom\left((A,\le_{A}),\ind{X}\right)\right)$ defines an internal-hom in $\Topa$ from $A$ to $X$, which implies that $A$ is a core-compact topological space by Lemma \ref{lem:corecompact-exponentiable}.
Conversely, let us assume that $A$ is a core-compact topological space and let $\left(\pTop\left(A, X\right), \le^{A}_{X}\right)$ be as in Construction \ref{con:exponentiable-preordered-space}.
Then, we have a natural bijection:
\[\pTop\left((Z\times A, \le_{Z}\times \le_{A}),(X,\le_{X})\right)\cong \pTop\left((Z,\le_{Z}),\left(\pTop\left(A, X\right), \le^{A}_{X}\right)\right)\]
that follows from the definitions.
\end{proof}

\section{A convenient category of stratified spaces\label{ch:4sec:2}}
Our treatment of stratified spaces is based on the work of Nand-Lal (see \cite[Definition 6.1.2.1]{nan18}, which in turn is a version of \cite[Definition A.5.1]{lur16} without the requirement of a fixed base poset).
Our point of view on stratified spaces differs from both of these sources for the use of preordered spaces as a natural ambient category for stratified spaces, and is based on earlier work of Woolf on well filtered spaces (see \cite{woo08} and Construction \ref{con:strat-preord}).
More specifically, the main achievement of this section is that we are able to construct a convenient category $\Strat$ of stratified spaces, that we call numerically generated stratified spaces, using the results of Section \ref{ch:1sec:4} applied to the category of preordered spaces.
The main advantage of doing so is that, in addition of being cartesian closed (Theorem \ref{thm:strat-cartesian-closed}), $\Strat$ is also a locally presentable category (Theorem \ref{thm:strat-loc-pres}).
Moreover, we construct an adjunction with the category of simplicial sets, which endows $\Strat$ with the structure of a simplicially enriched tensored and cotensored category (Theorem \ref{thm:strat-enriched-tensored-cotensored}).
Furthermore, we give a characterisation of numerically generated stratified spaces in terms of exit paths (Lemma \ref{lem:numerically-generated-exit-path}) and use it to deduce that the strata of a numerically generated stratified space are path connected (Corollary \ref{cor:strata-path-connected}).
To conclude we set up an adjunction between $\Strat$ and $\Top$, and give an explicit description of the right adjoint (Proposition \ref{prop:delta-prespace-associated-to-delta-space}). 

\begin{definition}
\label{def:strat-space}
A \emph{stratified space} is a continuous surjective map $s\colon X\to P$ from a topological space $X$ to a poset $P$ endowed with the \Alexandroff{} topology.
If $s\colon X\to P$ is a stratified space we call $X$ the \emph{total space}, $P$ the \emph{base poset} and $s$ the \emph{stratification map} of the stratified space.
A morphism of stratified spaces $(f,\alpha)\colon (s\colon X\to P)\to (s'\colon X'\to P')$ is a commutative square: 
\begin{equation*}
\csquare{X}{X'}{P'}{P}{f}{s'}{\alpha}{s}
\end{equation*}
where $f\colon X\to X'$ is a continuous map and $\alpha\colon P\to P'$ is a weakly monotonic map.
Let $\Strata$ be the category of stratified spaces. Since the \Alexandroff{} topology functor is fully faithful, $\Strata$ is a full subcategory of the arrow category $\arr{\Topa}$.
\end{definition}

\begin{notation}
Let $s\colon X\to P$ be a stratified space. If $p$ is a point of $P$, we use the following standard notation:
\begin{align*}
X_{\le p}	&= s^{-1}(P_{\le p})	\\
X_{\ge p}  &= s^{-1}(P_{\ge p})	\\
X_{< p}	&= s^{-1}(P_{< p})	\\
X_{> p}	&= s^{-1}(P_{> p})	\\
X_{p}	&= s^{-1}(p).
\end{align*}
Moreover, the fibre $X_{p}$ of $s$ at $p$ is called the \emph{stratum} of $X$ at $p$.
Notice that $X_{p}=X_{\le p}\cap X_{\ge p}$ is locally closed in $X$ since $X_{\le p}$ is closed and $X_{\ge p}$ is open.
\end{notation}

\begin{construction}
\label{con:strat-preord}
Let $s\colon X\to P$ be a stratified space. 
Since the functor $\pTop\to \Topa$ is topological, we can endow $X$ with the initial structure with respect to $s$, where $P$ is equipped with the structure of a preordered space as in \ref{subsec:alex-preord-space}.
Concretely, we endow $X$ with the preorder $\le_{s}$ defined by:
\[x\le_{s} x' \quad \Leftrightarrow \quad sx\le sx' \text{ in } P.\]
Conversely, let $(X,\le)$ be a preordered space, then the quotient $X/\sim$ of $X$ by the chaotic equivalence relation inherits the structure of a partially ordered topological space and we write $P_{X}$ for the underlying poset of $X/\sim$ equipped with the \Alexandroff{} topology.
Let $UX$ be the underlying set of $X$, we denote by $\overline{X}$ the topological space on $UX$ equipped with the initial topology with respect to the maps:
\[1_{UX}\colon UX\to UX, \quad q\colon UX \to UP.\]
Then, by construction, the quotient map defines a stratification $q\colon \overline{X}\to P_{X}$ of $\overline{X}$ by $P_{X}$.
Notice that the topology on $\overline{X}$ is generated by the union of the topology on $X$ and the \Alexandroff{} topology on the underlying preorder of $(X,\le)$.
\end{construction}

\begin{lemma}
\label{lem:strat_closed_colimits}
The constructions described in \ref{con:strat-preord} define an adjunction:
\begin{equation}
\Adjoint{\Strata}{\pTop}{\iota}{\phi}
\end{equation}
Moreover, $\iota$ is a full embedding and the composition $\phi\iota$ equals the identity on $\Strata$.
\end{lemma}
\begin{proof}
We need to show that for every stratified space $s\colon X\to P$ and every preordered topological space $(Y,\le)$, there exists a natural isomorphism:
\[\pTop(\iota(X,s,P), (Y,\le))\cong \Strata((X,s,P),\phi(Y, \le)).\]
Following the notation of \ref{con:strat-preord} let $f\colon (X,\le_{s})\to (Y,\le)$ be a morphism of preordered spaces.
As $f$ is weakly monotonic, for every upward closed subset $A$ in $Y$, the preimage $f^{{-1}}(A)$ is upward closed in $X$. In particular, $f^{{-1}}(A)$ is open in $X$, which implies that $f$ defines a map $f\colon X\to \overline{Y}$ fitting in the diagram:
\begin{diagram}
X
	\ar[r, "f"]
	\ar[d, "s"']&
\overline{Y}	\ar[d, "q"]\\
P\ar[r, dashed]&
P_{Y}.
\end{diagram}
We define a map $\alpha\colon P\to P_{Y}$ as follows: let $p\in P$, as $s$ is surjective there exists an element $x\in X$ such that $sx=p$ and we define $\alpha(p)=qfx$.
If $x'$ is another representative for $p$, we have that $x\le_{s} x'\le_{s} x$ in $X$. In particular, since $f$ is a map of preordered spaces and $P_{Y}$ is the quotient of $Y$ by the chaotic relation, we have that $qfx=qfx'$ in $P_{Y}$, hence $\alpha$ is well defined. Uniqueness of $\alpha$ follows from surjectivity of $s$.
To conclude the proof we need to show that  $(X,s,P)=\phi\iota(X,s,P)$ for every stratified space.
To do so, note that the topology of $\overline{X}$ on the underlying set of $X$ coincides with the topology of $X$, since every upward closed subset with respect to the preorder $\le_{s}$ is open in $X$.
Moreover, the quotient poset $P_{X}$ coincides with $P$ since the map $s\colon X\to P$ is surjective.
\end{proof}

\subsection{}
\label{subsec:limits-colimits-strat}
Lemma \ref{lem:strat_closed_colimits} in particular implies that the category of stratified spaces is isomorphic to a full subcategory of the category of preordered spaces and that the inclusion functor $\iota\colon \Strata\to \pTop$ creates colimits.
However, in practical applications it is convenient to have an explicit description of colimits and finite products, computed in the category $\Strata$ of stratified spaces.
If $F\colon \cat{C}\to \Strata$ is a small diagram of stratified spaces, with $Fc=X_{c}\to P_{c}$, the colimit of $F$ can be computed as the unique map 
\[s\colon \colim_{c\in \cat{C}}X_{c}\to \colim_{c\in \cat{C}}P_{c}\]
from the colimit of the total spaces to the colimit of the base spaces.
Limits in $\Strata$ are in general slightly more complicated, and we refer to \cite[Proposition 6.1.4.1]{nan18} for a complete discussion. Nonetheless, given stratified spaces $s\colon X\to P$ and $s'\colon X'\to P'$, the product of $s$ and $s'$ in $\Strata$ is given by:
\[s\times s'\colon X\times X'\to P\times P'.\]
Put more succinctly, finite products and colimits in $\Strata$ are created by the embedding $\Strata\to \arr{\Topa}$ of $\Strata$ in the arrow category of $\Topa$.

\subsection{} 
The \emph{stratified interval} is the topological interval $\II$ equipped with the map $\Omega_{0}\colon \II\to [1]$ to the \Sierpinski{} space that classifies the complement of $0$ in $\II$.
In other words, $\Omega_{0}(0)=0$ and $\Omega_{0}(t)=1$ for $t$ different from $0$.
A morphism 
\begin{equation*}
\csquare{\II}{X}{P}{[1]}{f}{s}{\alpha}{\Omega_{0}}
\end{equation*}
from the standard stratified interval to an arbitrary stratified space is said to be an \emph{elementary exit path} in $X$.
This is given by the choice of a point $x_{0}$ in a stratum $X_{p}$ of $X$ and a continuous path $f\colon \II\to X$ from $x_{0}$ to a point $x_{1}$ in a stratum $X_{q}$ such that $q\ge p$ and the image of the complement of $0$ is contained in $X_{q}$.

\subsection{}
\label{subsec:left-cone-strat}
Recall that both $\II$ and $[1]$ have the structure of a segment object (see Definition \ref{subsec:segment-object} and Examples \ref{ex:cone-top} and \ref{ex:cone-cat}).
The stratified interval inherits the structure of a segment object, since the following diagrams are commutative
\begin{equation*}
\begin{tikzcd}[ampersand replacement = \&]
{\term}\ar[r, "{i_{0}}", shift left]
	\ar[r, "{i_{1}}"', shift right]
	\ar[d]\&
{\II}	\ar[r]
	\ar[d]\&
{\term}\ar[d]\\
{[0]}	\ar[r, "{\partial_{1}}", shift left]
	\ar[r, "{\partial_{0}}"', shift right]\&
{[1]}	\ar[r]\&
{[0]}
\end{tikzcd}\quad
\begin{tikzcd}[ampersand replacement =\&]
{\II\times\II}\ar[r, "{\min{}}"]
		\ar[d, "{\Omega_{0}\times\Omega_{0}}" description] \&
{\II}		\ar[d, "{\Omega_{0}}"]\\
{[1]\times[1]}\ar[r, "{\min{}}"'] \&
{[1]}
\end{tikzcd}
\end{equation*}
In particular, for every stratified space $(X,s,P)$ we can define its \emph{left cone} (see \ref{subsec:left-cone}).
By \ref{subsec:limits-colimits-strat} the left cone $\lcone{(X,s,P)}$ has total space the $\II$-based left cone $\lcone{X}$ on $X$ and base poset the $[1]$-based left cone $\lcone{P}$ on $P$.
Notice that, the base segment objects of $\lcone{X}$ and $\lcone{P}$ are different.

\subsection{}
\label{subsec:stratified-realisation}
For every finite ordinal $[n]$ we define the \emph{stratified standard} $n$-\emph{simplex} inductively as follows.
We set $\sabs{\Delta^{0}}$ to be the terminal object $\abs{\Delta^{0}}\to[0]$ in $\Strata$ and we define
\[\sabs{\Delta^{n}}=\lcone{\sabs{\Delta^{n-1}}}\]
Notice that the total space of $\sabs{\Delta^{n}}$ is the geometric standard $n$-simplex $\abs{\Delta^{n}}$ and the base poset is $[n]$.
Explicitly, the stratification map can be given in barycentric coordinates as:
\begin{align*}
s\colon \abs{\Delta^n}	&\to [n]\\
(t_0,\ldots,t_n) & \mapsto \max_{i:t_i\ne 0}i.
\end{align*}
In particular, for $n=1$ we recover the stratified interval, and we will use the compact notation $\sabs{\Delta^{1}}$ henceforth.

\begin{remark}
Classical definitions of stratified spaces, including Thom-Mather stratified spaces, topological stratified spaces (see \cite{tho69}, \cite{mat73}, \cite{mat12}), Whitney stratified spaces (see \cite{whi92}) and homotopically stratified sets (see \cite{qui88}) are all examples of stratified spaces in the sense of Definition \ref{def:strat-space}.
More precisely, there is a chain of implications starting from Whitney stratified spaces and ending with homotopically stratified sets (see \cite[Remark 5.1.0.13]{nan18}) and every homotopically stratified set is a stratified space in the sense of Definition \ref{def:strat-space}.
Among all these notions, we are mostly interested in homotopically stratified sets as they possess pleasant homotopical properties (in particular they are fibrant stratified spaces in the sense of Definition \ref{def:model-structure-strat}, see Theorem \ref{thm:hoss-fibrant}) while they don't require any smooth structure (in contrast with Whitney and Thom-Mather stratified spaces).
A taxonomy of the different notions of stratified spaces can be found in \cite{HW01}.
\end{remark}

\subsection{} 
Let $\cat{I}$ be the full subcategory of $\pTop$ spanned by the stratified standard $n$-simplices. 
We denote by $\Strat$ the final closure of $\cat{I}$ in $\pTop$ and call it the category of numerically generated stratified spaces.
Notice that we have coreflective embeddings.
\[\Strat\to\Strata\to\pTop\]
In particular, colimits in $\Strat$ can be computed as in $\Strata$ or $\pTop$, and we will use both descriptions interchangeably according to which of them is most natural.
We now investigate the categorical properties of the category of numerically generated stratified spaces.

\begin{theorem}
\label{thm:strat-loc-pres}
The category $\Strat$ is a locally presentable, coreflective subcategory of $\pTop$.
\end{theorem}
\begin{proof}
The claim follows from Proposition \ref{prop:C_I-coreflective}, \ref{subsec:ptop-pullback-top-pre} and Theorem \ref{thm:C_I-locally-presentable}.
\end{proof}

\begin{lemma}
\label{lem:product-stratified-simplices}
For finite ordinals $[n]$ and $[m]$, the natural map:
\[\sabs{\Delta^n\times \Delta^m}\to \sabs{\Delta^n}\times\sabs{\Delta^m}\]
is an isomorphism.
\end{lemma}
\begin{proof}
By \cite[Chapter III, 3.4]{GZ67} we have a natural isomorphism: 
\[\abs{\Delta^{n}\times\Delta^{m}}\cong\abs{\Delta^{n}}\times\abs{\Delta^{m}}\]
for every finite ordinals $[n]$ and $[m]$.
Corollary \ref{cor:tau-preserves-products} implies that we have the following isomorphisms
\begin{align*}
[n]\times [m]
	&= \tau_{1}\left(\Delta^{n}\times\Delta^{m}\right)\\
	&= \colim_{\Delta^{k}\to\Delta^{n}\times\Delta^{m}}\tau_{1}\Delta^{k}\\
	&= \colim_{\Delta^{k}\to\Delta^{n}\times\Delta^{m}}[k]
\end{align*}
in the category $\Cat$ of small categories. Since $[n]\times[m]$ is a poset and $\Pos\to \Cat$ is a fully faithful functor, the isomorphisms hold in the category of posets as well.
The claim then follows from \ref{subsec:limits-colimits-strat}.
\end{proof}

\begin{lemma}
\label{lem:stratified-standard-simplex-exponentiable}
The stratified standard $n$-simplex $\sabs{\Delta^{n}}$ is an exponentiable preordered space.
\end{lemma}
\begin{proof}
This follows immediately from Proposition \ref{prop:ptop-corecompact-exponentiable}.
\end{proof}

\begin{theorem}
\label{thm:strat-cartesian-closed}
The category $\Strat$ of numerically generated stratified spaces is a Cartesian closed category.
\end{theorem}
\begin{proof}
The claim follows immediately from Lemma \ref{lem:stratified-standard-simplex-exponentiable} and Theorem \ref{thm:final-closure-cartesian-closed}.
\end{proof}

\subsection{}
The assignment that sends every finite ordinal $[n]$ to the stratified $n$-simplex $\sabs{\Delta^{n}}$ defines a functor $\DDelta\to \Strata$, which induces an adjunction:
\begin{equation}
\label{eq:sset-strata-adjunction}
\Adjoint{\sSet}{\Strata}{\sabs{\blank}}{\sSing}
\end{equation}
by Theorem \ref{thm:kan-presheaves}. 
The left adjoint is called the \emph{stratified realisation functor} and the right adjoint is called the \emph{stratified singular complex} functor.
Moreover \eqref{eq:sset-strata-adjunction} restricts to an adjunction
\begin{equation}
\label{eq:sset-strat-adjunction}
\Adjoint{\sSet}{\Strat}{\sabs{\blank}}{\sSing}
\end{equation}
between the category of simplicial sets and the category of numerically generated stratified spaces.

\begin{theorem}
\label{thm:strat-enriched-tensored-cotensored}
The category $\Strat$ is naturally an $\sSet$-enriched tensored and cotensored category via \ref{eq:sset-strat-adjunction}.
Moreover, the adjunction \ref{eq:sset-strat-adjunction} is an $\sSet$-enriched adjunction.
\end{theorem}
\begin{proof}
The claim follows from Proposition \ref{prop:cisinski-topological-enriched}.
\end{proof}

\begin{example}
The stratified realisation $\sabs{S^1}$ of the simplicial circle (see Example \ref{ex:simplicial-circle}) is the standard circle $\SS^1$ equipped with the chaotic stratification.
In particular, every elementary exit path $a\colon \sabs{\Delta^1}\to X$ in a stratified space $X$ such that $a(0)=a(1)$ is contained in a unique stratum.
\end{example}

\begin{example}
Recall that the $n$-spine is the sub-simplicial set $\Sp^{n}$ of $\Delta^{n}$ given by the union of the consecutive edges (see Example \ref{ex:n-spine}).
The \emph{stratified} $n$-\emph{spine} is the stratified space $\sabs{\Sp^{n}}$ given by the realisation of $\Sp^{n}$. 
As the functor $\sabs{\blank}$ preserves colimits, the total space of the stratified $n$-spine is the interval $[0,n]$, while the base poset is given by the finite ordinal $[n]$.
If $X$ is a stratified space, a morphism $f\colon \sabs{\Sp^{n}}\to X$ from the stratified $n$-spine to $X$ will be called an \emph{exit path} from $f(0)$ to $f(n)$.
\end{example}

\begin{lemma}
\label{lem:numerically-generated-exit-path}
Let $(X,\le)$ be a preordered space and assume that $X$ is a numerically generated topological space.
Then, $X$ is a numerically generated stratified space if and only if for any two points $x$ and $x'$ in $X$, we have that $x\le x'$ precisely when there exists an exit path from $x$ to $x'$ in $X$.
\end{lemma}
\begin{proof}
Assume $(X,\le)$ is numerically generated. By the explicit description of final structures in $\Pre$ given in \ref{subsec:final-initial-preorder}, there exist integers $(i_{1},\ldots, i_{n})$ and a set
\[\left\{p_{j}^{\epsilon} : \epsilon \in \{0,1\}, p_{j}^{\epsilon}\in \sabs{\Delta^{i_{j}}}\right\}\]
together with continuous weakly monotonic maps $f_{j}\colon \sabs{\Delta^{i_{j}}}\to X$ such that 
\begin{align}
p_{j}^{0}			&\le_{i_{j}} p_{j}^{1}				\label{eq:preord-final-structure-1}\\
f_{j}\left(p_{j}^{1}\right)	&=f_{j+1}\left(p_{j_{1}}^{0}\right)		\label{eq:preord-final-structure-2}\\
f_{1}\left(p_{1}^{0}\right)	=x &,\; f_{n}\left(p_{n}^{1}\right)=x'.			\label{eq:preord-final-structure-3}
\end{align}
where $\le_{i_{j}}$ denotes the preorder relation of $\sabs{\Delta^{i_{j}}}$. In particular, taking the line segment from $p_{j}^{0}$ to $p_{j}^{1}$ in $\sabs{\Delta^{i_{j}}}$ defines an elementary exit path:
\[\alpha_{j}\colon\sabs{\Delta^{1}}\to \sabs{\Delta^{i_{j}}}\to X\]
in $X$, for every $j$.
Then, \eqref{eq:preord-final-structure-2} and  \eqref{eq:preord-final-structure-3} imply that the $\alpha_{j}'s$ glue to an exit path $\sabs{\Sp^{n}}\to X$ from $x$ to $x'$.

Conversely assume that the preorder on $X$ is given by exit paths and let us consider a cocone $(g_{\alpha}\colon \sabs{\Delta^{n^{\alpha}}}\to (X',\le'))$ indexing over the cocone of all stratified $n$-simplices mapping to $X$. 
Assume that we have a continuous map $f\colon X\to X'$ such that, for every $\alpha\colon \sabs{\Delta^{n^{\alpha}}}\to X$, the following diagram of topological spaces is commutative:
\begin{equation}
\label{eq:lifting-numerically-generated-strat}
\begin{tikzcd}
\abs{\Delta^{n^{\alpha}}}
	\ar[r, "\alpha"]
	\ar[dr, "g_{\alpha}"']&
X	\ar[d, "f"]\\
&
X'
\end{tikzcd}
\end{equation}
Then, we need to show that $f$ defines a map of preordered spaces.
Let $x\le x'$ in $X$, then by hypothesis there exists an exit path $s\colon\sabs{\Sp^{n}}\to X$ from $x$ to $x'$.
In particular, this defines elementary exit paths $\alpha_{k}\colon \sabs{\Delta^{1}}\to X$ for $k=1,\ldots, n$ such that 
\[\alpha_{1}(0)=x,\; \alpha_{n}(1)=x'\quad \text{ and }\quad \alpha_{k}(1)=\alpha_{k-1}(0).\]
Therefore, specialising \eqref{eq:lifting-numerically-generated-strat} to the $\alpha_{k}$ we have commutative diagrams
\begin{equation}
\label{eq:lifting-numerically-generated-strat-2}
\begin{tikzcd}
\abs{\Delta^{1}}
	\ar[r, "\alpha_{k}"]
	\ar[dr, "g_{\alpha_{k}}"']&
X	\ar[d, "f"]\\
&
X'
\end{tikzcd}
\end{equation}
that imply there is a chain of inequalities 
\[fx\le' g_{\alpha_{1}}(1)= g_{\alpha_{2}}(0)\le'\ldots \le' g_{\alpha_{n-1}}(1)=g_{\alpha_{n}}(0)\le' fx'.\]
Hence, the claim follows from transitivity of $\le'$.
\end{proof}

\begin{corollary}
\label{cor:strata-path-connected}
Let $X$ be a numerically generated stratified space. Then, for every point $x\in X$, the stratum $X_{x}$ of $X$ at $x$ is a path connected topological space.
\end{corollary}
\begin{proof}
Let $x'$ be a point of $X_{x}$
By Lemma \ref{lem:numerically-generated-exit-path}, since $x\le x'$, there exists an exit path $\alpha\colon \sabs{\Sp^{n}}\to X$ from $x$ to $x'$ in $X$.
As $x'\le x$, the image of $\alpha$ is contained in $X_{x}$.
Hence the underlying map of topological spaces $\alpha\colon [0,n]\to UX$ defines a path in $X_{x}$ from $x$ to $x'$.
\end{proof}

\subsection{}
\label{subsec:coreflection-top-strat}
The underlying topological space of any numerically generated stratified space is a numerically generated space, since the forgetful functor $U\colon \pTop\to \Topa$ maps stratified standard simplices to geometric standard simplices and it preserves colimits. In particular, the restriction of $U$ defines a functor
\[U\colon \Strat \to \Top\]
which comes equipped with a right adjoint:
\[\pi\colon \Top\to \Strat\]
given by the composition of the indiscrete object functor $\Top\to \pTop$ with the coreflection $\pTop\to \Strat$. 

\begin{remark}
It is worth pointing out the difference between numerically generated stratified spaces, and stratified spaces with a numerically generated total space and arbitrary stratification map (as used for example in \cite{hai18}).
\footnote{The punchline is that ``numerically generated'' and ``stratified'' do not commute with each other}
Corollary \ref{cor:strata-path-connected} gives a clear measure of this difference, as one can easily find examples of the second kind which don't have path-connected strata.
For a concrete example, take $P$ to be a (non discrete) poset and $\disc{P}=\coprod_{p\in P}\{p\}$ to be the discrete topological space on $P$. Then, the identity on $P$ defines a stratification map $\disc{P}\to P$ with disconnected strata.
\end{remark}

\subsection{} 
Let $X$ be a numerically generated space and let $\pi_{0}X$ be the set of path-connected components of $X$. Notice that, if we endow $\pi_{0}X$ with the discrete poset structure, the Alexandroff topology on $\pi_{0}X$ coincides with the discrete topology. In particular, the quotient map $\pi\colon X\to \pi_{0}X$ defines a stratified space $(X,\pi,\pi_{0}X)$.
In terms of preordered spaces $(X,\pi,\pi_{0}X)$ is the space $X$ equipped with the preorder $\sim$ defined by: $x_0\sim x_1$ if and only if there exists a path $\alpha\colon \II\to X$ from $x_0$ to $x_1$. 
Notice that, by adjunction, the identity on $X$ defines a unique map $\eta_X\colon (X,\sim)\to \ind{X}$ to the codiscrete preordered space on $X$.

\begin{proposition}
\label{prop:delta-prespace-associated-to-delta-space}
Let $X$ be a numerically generated topological space, then there is a natural isomorphism:
\[\pi(X)\cong(X,\sim).\]
In particular, if $X$ is a path connected numerically generated space, the codiscrete preordered space $\ind{X}$ associated to $X$ is a numerically generated stratified space.
\end{proposition}
\begin{proof}
Since $\pi(X)$ is given by the coreflection $\pTop\to\Strat$ applied to the codiscrete object $\ind{X}$ on $X$, it is enough to show that $(X,\sim)$ satisfies the same universal property.
By Lemma \ref{lem:coreflection-I-generated} it is enough to show that for every $[n]\in \DDelta$ there exists a natural isomorphism:
\begin{equation}
\label{eq:path-generated-space}
\Strat(\sabs{\Delta^n},(X,\sim))\cong \pTop(\sabs{\Delta^n}, \ind{X})
\end{equation}
and that $\pi(X)$ is a numerically generated stratified space.
The isomorphism \eqref{eq:path-generated-space} follows from the fact that $\abs{\Delta^n}$ is path connected, hence a morphism $\sabs{\Delta^{n}}\to (X,\sim)$ is the same as an unstratified morphism $\abs{\Delta^{n}}\to X$. In other words, we have natural isomorphisms:
\begin{equation*}
\Strat\left(\sabs{\Delta^n},(X,\sim)\right)\cong\Top\left(\abs{\Delta^{n}},X\right)\cong \pTop\left(\sabs{\Delta^n}, \ind{X}\right),
\end{equation*}
where the second isomorphism follows from the adjunction between the codiscrete object functor and the forgetful functor.
To prove that $(X,\sim)$ is a numerically generated stratified space we apply Lemma \ref{lem:numerically-generated-exit-path}.
We claim that $x\sim x'$ in $X$ if and only if there exists an exit path $f\colon \sabs{\Sp^{n}}\to (X,\sim)$ from $x$ to $x'$.
Let $x, x'$ be points in $X$ such that $x\sim x'$ then, by definition, there exists a path $\alpha\colon\II\to X$ from $x$ to $x'$.
The path $\alpha$ defines a unique elementary exit path $\alpha\colon \sabs{\Delta^{1}}\to (X,\sim)$ from $x$ to $x'$, with same underlying path. The first implication then follows observing that $\Delta^{1}=\Sp^{1}$.
Conversely, the strata in $(X,\sim)$ are the connected components of $X$, and so every exit path $f\colon \sabs{\Sp^{n}}\to (X,\sim)$ is forced to stay in a unique stratum.
In particular $f$ defines a unique chaotic path $\alpha\colon \II\to (X,\sim)$ whose underlying path is given by the composition:
\begin{diagram}
{[0,1]}\ar[r, "n\cdot"]& {[0,n]} \ar[r, "f"] & X.
\end{diagram}
where $n\cdot\colon [0,1]\to[0,n]$ is given by multiplication by $n$.
\end{proof}
\section{The homotopy theory of stratified spaces}
\label{ch:4sec:3}
We conclude the chapter with a recollection on the modern homotopy theory of stratified spaces, following \cite{nan18}.
We apply the results of Section \ref{ch:4sec:2} to the category $\Strat$ of numerically generated stratified spaces.
In particular, Theorem \ref{thm:strat-fibrant-objects} asserts that the category of fibrant stratified spaces (Definition \ref{def:model-structure-strat}) has the structure of a category with fibrant objects (see Definition \ref{def:category-fibrant-objects}).
We recall that every stratum-preserving homotopy equivalence is a weak equivalence of stratified spaces (Proposition \ref{prop:stratum-preserving-equivalence-we}) and the converse is true when we restrict to the category of cofibrant-fibrant stratified spaces. (see \cite{nan18})
Moreover, we recall the definition of homotopically stratified spaces and how they are related to cofibrant and fibrant objects (Theorem \ref{thm:cofibrant-fibrant-hoss} and \ref{thm:hoss-fibrant}).
We conclude the section with the definition of the exit path category of a stratified space (Definition \ref{def:exit-path-category}) and its relationship with stratified covers (Theorem \ref{thm:woolf}).
The main results of this section can be found in \cite{nan18} and \cite{woo08}.

\begin{definition}
\label{def:model-structure-strat}
A morphism $f\colon X\to Y$ of stratified spaces is said to be a \emph{fibration} (\resp a weak equivalence) if the induced map:
\[\sSing(f)\colon \sSing(X)\to \sSing(Y)\]
is a fibration (\resp a weak equivalence) in the Joyal model structure.
We call $f$ a \emph{trivial fibration} if it is both a fibration and a weak equivalence.
\end{definition}

\begin{definition}
A morphism $i\colon A\to B$ between stratified spaces is said to be a \emph{cofibration} (\resp a trivial cofibration) if it has the left lifting property with respect to all trivial fibration (\resp all fibrations).
We call $i$ an \emph{acyclic cofibration} if it is both a cofibration and a weak equivalence.
\end{definition}

\begin{definition}
Let $X$ be a stratified space.
We say that $X$ is \emph{fibrant} if $\sSing(X)$ is an $\infty$-category.
Dually, we call $X$ \emph{cofibrant} if the unique morphism $\initial\to A$ from the initial object is a cofibration.
\end{definition}

\begin{example}
\label{poset-fibrant}
Let $P$ be a poset equipped with the structure of a stratified space via the identity map $1_{P}\colon P\to P$.
Then $P$ is a fibrant stratified space.
Indeed, the following natural isomorphism holds in $\sSet$:
\begin{equation}
\sSing(P)\cong \Nerv P
\end{equation}
since a morphism $\sabs{\Delta^{n}}\to P$ is uniquely determined by a weakly monotonic map $[n]\to {P}$.
In particular the stratified space $1_{[n]}\colon [n]\to [n]$ associated to every finite ordinal $[n]$ is a fibrant stratified space.
\end{example}

\begin{example}
\label{ex:realisation-sing-poset-co-fibrant}
Let $P$ be a poset and let us consider the stratified geometric realisation $\sabs{\sSing P}$ of the nerve of $P$.
Then, $\sabs{\sSing P}$ is a fibrant stratified space by \cite[Remark 9.4.0.12]{nan18} and it is also cofibrant, since $\sabs{\blank}$ preserves cofibrant objects.
In particular, for every $n$ in $\DDelta$, the stratified standard $n$-simplex $\sabs{\Delta^{n}}$ is a fibrant and cofibrant stratified space.
\end{example}

\begin{remark}
By Lemma \ref{lem:product-stratified-simplices} and Lemma \ref{lem:stratified-standard-simplex-exponentiable} the functor $\sabs{\blank}\colon \DDelta\to \pTop$ satisfies the hypothesis described in \ref{subsec:cisinski-topological}.
In particular, we can harvest the results of Section \ref{ch:2sec:4} with no effort, and we highlight them in what follows.
\end{remark}

\begin{lemma}
\label{lem:strat-homotopical-cat}
The category $\Strata$ equipped with the class of weak equivalences has the structure of a homotopical category.
\end{lemma}
\begin{proof}
Since $\sSing$ reflects weak equivalences and $\sSet$ is a model category, the claim follows from Remark \ref{rmk:model-cat-homotopical-cat}.
\end{proof}

\begin{lemma}
\label{lem:strat-prod-we}
Let $X$ be a stratified space and let $w\colon Y\to Z$ be a weak equivalence between stratified spaces, then the induced map:
\[X\times w\colon X\times Y\to X\times Z.\]
is a weak equivalence.
\end{lemma}
\begin{proof}
See Lemma \ref{lem:top-cat-prod-we}.
\end{proof}

\begin{proposition}
\label{prop:cotensor-Quillen-functor-strat}
Let $i\colon K\to L$ be a monomorphism between simplicial sets and let $p\colon X\to Y$ be a fibration in $\Strat$. Then, the induced map
\[ X^L\to X^K\times_{Y^K}Y^L\]
is a fibration, which is acyclic if $i$ or $p$ is so.
\end{proposition}
\begin{proof}
See Proposition \ref{prop:cotensor-Quillen-functor}.
\end{proof}

\begin{corollary}
\label{cor:internal-hom-quillen-strat}
Let $f\colon X\to Y$ be a fibration in $\Strat$ and let $i\colon A\to B$ be a cofibration. Then, the induced map
\[\ihom(B,X)\to \ihom(A,X)\times_{\ihom(A,Y)}\ihom(B,Y)\]
is a fibration, which is acyclic if either $p$ is acyclic or $i$ is a trivial cofibration.
\end{corollary}
\begin{proof}
The claim follows from Lemma \ref{lem:two-variable-lifting} and Proposition \ref{prop:cotensor-Quillen-functor-strat}.
\end{proof}

\subsection{}
Recall that $J$ is an interval object for the Joyal model structure on $\sSet$ (see Example \ref{ex:J-interval-sset}).
In other words, $J$ comes equipped with a factorisation of the codiagonal:
\[\terminal\coprod\terminal\to J\to\terminal\]
such that $(\partial_0,\partial_1)\colon *\coprod *\to J$ is a monomorphism and $\sigma\colon J\to *$ is a Joyal weak equivalence. 
Then, for every numerically generated stratified space $X$, applying the functor $X^{(\blank)}$ we get a factorisation
\[X\to X^J\to X\times X\]

\begin{lemma}
\label{lem:path-object-for-strat}
For every fibrant stratified space $X$ in $\Strat$, the above factorisation induces the structure of a path object $X^{J}$ for $X$.
\end{lemma}
\begin{proof}
This follows immediately from Lemma \ref{lem:path-object-for-top-cat}.
\end{proof}

\begin{theorem}
\label{thm:strat-fibrant-objects}
The full subcategory $\Strat_{f}$ of $\Strat$ spanned by the fibrant stratified spaces is a category of fibrant objects.
\end{theorem}
\begin{proof}
This follows from Proposition \ref{prop:top-cat-fibrant-objects} and Lemma \ref{lem:path-object-for-strat}.
\end{proof}

\subsection{}
Notice that, by Proposition \ref{prop:delta-prespace-associated-to-delta-space}, the codiscrete interval $\II$ is a numerically generated stratified space.
Moreover, $\II$ fits into a diagram:
\begin{equation}
\label{eq:chaotic-interval-strat}
\begin{tikzcd}
\terminal\coprod\terminal \ar[r, "{\left(\partial^0,\partial^1\right)}"] &
\II \ar[r, "\sigma"] &
\terminal
\end{tikzcd}
\end{equation}
whose composition is the codiagonal $\left(1_\terminal,1_\terminal\right)\colon\terminal\coprod\terminal\to \terminal$.
As $\sSing(\II)$ is isomorphic to the singular simplicial set associated to the underlying space of $\II$, the morphism $\sSing(\II)\to \Delta^0$ is a trivial Kan fibration, and in particular, $\sigma$ is a trivial fibration in $\Strata$.

\subsection{}
Recall that the \emph{topological $n$-disk} $\DD^{n}$ is the subspace of $\RR^{n+1}$ given by:
\[\DD^{n}=\{x\in \RR^{n+1}: \norm{x}_{2}\le 1\}\]
where $\norm{\blank}_{2}$ is the standard Euclidean norm.

\begin{lemma}
\label{lem:strat-J-trivial-disk}
The stratified space $\sabs{J}$ is a chaotically stratified 2-disk $\ind{\DD^{2}}$.
\end{lemma}
\begin{proof}
It is clear from the construction of $J$ that the underlying topological space of $\sabs{J}$ is a 2-disk, given by attaching two half disks along the diameter.
To see that the stratification of $\sabs{J}$ is chaotic, it is enough to notice that $0$ is a minimum and a maximum for the preorder on $\sabs{J}$.
\end{proof}

\begin{corollary}
The morphism $\left(\partial^0,\partial^1\right)\colon \terminal\coprod\terminal\to \II$ is a retract of $\terminal\coprod\terminal \to \sabs{J}$. In particular, $\left(\partial^0,\partial^1\right)$ is a cofibration.
\end{corollary}
\begin{proof}
The diameter $i\colon\II\to \DD^{2}$ from $0$ to $1$ defines a morphism of chaotic stratified spaces, with a retraction $r\colon\ind{\DD^{2}}\to \ind{\II}$ given by vertical projection.
The second part of the claim follows from the fact that $\sabs{\blank}$ preserves cofibrations and that cofibrations are stable under retracts.
\end{proof}

\subsection{}
For every stratified space $X$, applying the functor $X\times(\blank)$ to \eqref{eq:chaotic-interval-strat} yields a decomposition of the codiagonal on $X$
\begin{equation}
\label{eq:cylinder-strat}
\begin{tikzcd}
X\coprod X \ar[r, "{\left(\partial_{X}^0,\partial_{X}^1\right)}"] &
X\times \II \ar[r, "\sigma_{X}"] &
X
\end{tikzcd}
\end{equation}
where $\sigma_{X}$ is a weak equivalence by Lemma \ref{lem:cylinder-top-cat}.
Moreover, if $X$ is a cofibrant stratified space, the morphism $\left(\partial_{X}^{0}, \partial_{X}^{1}\right)$ is a cofibration, again by Lemma \ref{lem:cylinder-top-cat}.

\subsection{}
\label{subsec:stratum-preserving-homotopy}
Given $f\colon X\to Y$ and $g\colon X\to Y$ morphisms between stratified spaces, a \emph{stratum preserving homotopy} from $f$ to $g$ is a morphism $h\colon X\times\II\to Y$ such that $h$ restricted to $X\times\{0\}$ equals $f$ and $X$ restricted to $X\times \{1\}$ equals $g$.
In the presence of a stratum preserving homotopy from $f$ to $g$ we say that $f$ is \emph{stratum preserving homotopic} to $g$ and write $f\sim g$.
Notice that the relation $\sim$ defines an equivalence relation on $\Strata(X,Y)$ since $\II$ fits in a cocartesian diagram:
\begin{equation*}
\push{\term}{\II}{\II}{\II}{\partial^{1}}{}{}{\partial^{0}}
\end{equation*}
of stratified spaces.

\begin{definition}
Let $f\colon X\to Y$ be a morphism between stratified spaces. We say that $f$ is a \emph{stratum preserving homotopy equivalence} if there exists a morphism $g\colon Y\to X$ such that $gf$ is stratum preserving homotopic to the identity on $X$ and $fg$ is stratum preserving homotopic to the identity on $Y$.
\end{definition}

\begin{proposition}
\label{prop:stratum-preserving-equivalence-we}
Every stratum preserving homotopy equivalence is a weak equivalence.
\end{proposition}
\begin{proof}
The claim follows from Proposition \ref{prop:E-homotopy-we}.
\end{proof}

\begin{definition}
We say that a stratified space $X$ satisfies the \emph{frontier condition} if for every $p$ and $q$ in the base poset of $X$, whenever $X_{p}$ intersects the closure of $X_{q}$ the stratum $X_{p}$ is contained in the closure of $X_{q}$
\end{definition}

\begin{proposition}
Every cofibrant and fibrant stratified space satisfies the frontier condition.
\end{proposition}
\begin{proof}
See \cite[Corollary 9.3.0.5]{nan18}
\end{proof}

\begin{definition}
Let $X$ be a stratified space and let $p$ and $q$ be elements in the base poset of $X$.
The \emph{homotopy link} from $X_{p}$ to $X_{q}$ is the pullback
\begin{equation}
\pull{\holink(X_{p},X_{q})}{X^{\Delta^{1}}}{X\times X}{X_{p}\times X_{q}}{}{}{}{}
\end{equation}
taken in the category $\Top$.
\end{definition}

\subsection{}
Recall that a subspace $i\colon A\subset X$ of a topological space $X$ is a \emph{strong deformation retract} if there exists a retraction $r\colon X\to A$ together with a homotopy $h\colon \II\times X\to X$ from $ir$ to the identity on $X$ such that the restriction of $h$ to $\II\times A$ coincides with the projection onto $A$.

\begin{definition}
Let $i\colon A\to X$ be a union of strata in a stratified space $X$. We say that $i$ is an \emph{almost stratum preserving strong deformation retract} if the underlying topological map of $i$ is a strong deformation retract with deformation $h$ such that, for every point $x\in X$, the restriction $h_{x}$ of $h$ to $\II\times \{x\}$ defines an elementary exit path in $X$.
\end{definition}

\begin{definition}
Let $A\subset B\subset X$ be embeddings of union of strata in $X$. We say that $A$ is \emph{tame} in $B$ if there exists a neighbourhood $N$ of $A$ in $B$ such that the inclusion $A\to B$ is an almost stratum preserving strong deformation retract. 
\end{definition}

\begin{definition}
\label{def:homotopically-sset}
Let $X$ be a stratified space. We say that $X$ is a \emph{homotopically stratified space} if, for every $p$ and $q$ in the base poset of $X$, with $p\le q$, the following properties hold
\begin{enumerate}
\item
The stratum $X_{p}$ is tame in $X_{p}\cup X_{q}$.
\item 
The morphism 
\[\partial_{0}^{*}\colon\holink(X_{p},X_{q})\to X_{p}\]
induced by the evaluation at $0$ is a Serre fibration.
\end{enumerate}
\end{definition}

\begin{theorem}
\label{thm:cofibrant-fibrant-hoss}
Every cofibrant and fibrant stratified space is a homotopically stratified space.
\end{theorem}
\begin{proof}
See \cite[Theorem 9.3.0.2]{nan18}
\end{proof}

\begin{theorem}
\label{thm:hoss-fibrant}
Let $X$ be a homotopically stratified space with finite stratification and assume that the topology on $X$ is induced by a metric. Then,  $X$ is a fibrant stratified space.
\end{theorem}
\begin{proof}
See \cite[Proposition 8.1.2.6]{nan18}
\end{proof}

\begin{theorem}
\label{thm:we-stratum-preserving-equivalence}
A morphism $f\colon X\to Y$ between cofibrant and fibrant stratified spaces is a weak equivalence if and only if it is a stratum preserving homotopy equivalence.
\end{theorem}
\begin{proof}
See \cite[Theorem 9.3.0.7]{nan18}
\end{proof}

\begin{definition}
\label{def:exit-path-category}
Let $X$ be a fibrant stratified space. The \emph{exit-path category} of $X$ is the homotopy category of the $\infty$-category associated to $X$:
\[\Exit X=\ho\Sing X\]
\end{definition}

\begin{lemma}
\label{lem:exit-path-I-homotopy}
Let $X$ be a fibrant stratified space, let $x$ and $x'$ be points of $X$ and let $f$ and $g$ be exit paths in $X$ from $x$ to $x'$.
Then $f$ and $g$ define the same morphism in $\Exit X$ if and only if there exists an $\II$-homotopy $h\colon \sabs{\Delta^{1}}\times \II\to X$ from $f$ to $g$, whose restriction to $\partial\Delta^{1}\times \II$ is constant.
\end{lemma}
\begin{proof}
By definition of the homotopy category, $f$ and $g$ define the same morphism if and only if there exists a map $k\colon\sabs{\Delta^{2}}\to X$ which restricts to $(f,g,1_{x'})$ on $\partial\sabs{\Delta^{2}}$.
On the level of underlying topological spaces, it is clear how a morphism $k$ determines a morphism $h$ and viceversa.
To conclude the proof, it is enough to notice that the images of $h$ and $k$ as above are 2-disks stratified by a point and its complement.
\end{proof}

\begin{example}
\label{ex:exit-path-standard-simplex}
The stratified standard $n$-simplex $\sabs{\Delta^{n}}$ is a fibrant stratified space by Example \ref{ex:realisation-sing-poset-co-fibrant}, hence it possesses an exit-path category.
Moreover, we have the following equivalence of categories:
\[\Exit \sabs{\Delta^{n}}\cong [n]\]
Indeed, by Lemma \ref{lem:exit-path-I-homotopy} it is enough to observe that every exit path $\sabs{\Delta^{1}}\to \sabs{\Delta^{n}}$ from a stratum $k$ to a stratum $j$ is stratum preserving homotopic to the linear exit-path from the vertex $k$ to the vertex $j$.
\end{example}

\begin{definition}
\label{def:stratified-cover}
Let $X$ be a homotopically stratified space, and assume that the strata of $X$ are locally $0$-connected and locally $1$-connected.
A \emph{stratified cover} of $X$ is a morphism $p\colon E\to X$ of stratified spaces, such that $p$ is a local homeomorphism and the restriction of $p$ to every stratum of $X$ is a topological cover.
We denote by $\sCover{X}$ the category of stratified covers over $X$.
\end{definition}

\begin{theorem}[Woolf]
\label{thm:woolf}
Let $X$ be a fibrant homotopically stratified space with locally $0$-connected and locally $1$-connected strata.
Then, there exists an equivalence of categories:
\[\sCover{X}\cong \Set^{\Exit X}.\]
\end{theorem}
\begin{proof}
See \cite[Theorem 4.3]{woo08}
\end{proof}

\begin{example}
The stratified standard simplex $\sabs{\Delta^{n}}$ satisfies the hypothesis of Theorem \ref{thm:woolf} by Example \ref{ex:realisation-sing-poset-co-fibrant} and Theorem \ref{thm:cofibrant-fibrant-hoss}.
Therefore, there is an equivalence of categories:
\[\sCover{\sabs{\Delta^{n}}}\cong \Set^{[n]}\]
by Example \ref{ex:exit-path-standard-simplex}.
\end{example}

\chapter{Locally stratified spaces}
\label{ch:5}

\section{Streams\label{ch:5sec:1}}
A prestream is a pair $(X,\preceq)$, where $X$ is a topological space and $\preceq$ is a compatible assignment of preorders on the open subsets of $X$, called precirculation.
Following \cite{kri09} a stream is a prestream $(X, \preceq)$ such that, the value of $\preceq$ on an open subset $U$ of $X$ is uniquely determined by its value on every element of every open cover of $U$ (see Definition \ref{def:circulation}).
The definition of streams we give here (Definition \ref{def:stream}) differs slightly from the original definition, and it coincides with the notion of a Haucourt stream (see \cite{hau12}).
A Haucourt stream is a prestream whose local preorder can be detected by directed paths and, up to isomorphism, is given by the prestream associated to a d-space (see \ref{subsec:dspace} and \ref{subsec:idempotent-comonad-streams}).
In \ref{subsec:preorder-prestream} we set up an adjunction between the category of preordered topological spaces and the category of prestreams.
The only original contribution of this section is Remark \ref{rmk:sierpinski-stream} and, as a consequence, Theorem \ref{thm:stream-exponentiable} which is based on \cite[Theorem 6.2]{gou14} and gives a characterisation of exponentiable objects in the category of Haucourt streams.

\begin{definition}
Let $X$ be a topological space. 
A \emph{precirculation} $\preceq$ on $X$ is a function that assigns to every open subset $U$ in $X$ a preorder relation $\preceq_U$ on $U$ such that, given open subsets $U\subset U'$ in $X$ and two points $x_0$ and $x_1$ in $U$, if $x_0\preceq_U x_1$ then $x_0\preceq_{U'} x_1$ in $U'$.
A \emph{prestream} is a pair $(X,\preceq)$ where $X$ is a topological space and $\preceq$ is a precirculation on $X$.
\end{definition}

\subsection{}
A morphism of prestreams $f\colon (X,\preceq)\to (X',\preceq')$ is a continuous map $f\colon X\to X'$ such that, for every open subset $U'$ in $X'$ and every two points $x_0$ and $x_1$ in $f^{-1}(U')$ one has $f(x_0)\preceq_{U'} f(x_1)$ whenever $x_0\preceq_{f^{-1}(U')} x_1$.
We denote by $\pStream$ the category of prestreams.
When no confusion will arise, a prestream $(X,\preceq)$ will simply be denoted by $X$, making the precirculation implicit.

\subsection{}
\label{subsec:stream-initial-final}
Given a continuous map $f\colon X\to X'$ of topological spaces and a precirculation $\preceq'$ on $X'$, we define the \emph{pullback} precirculation $(\preceq')^{f^{-1}}$ on $X$, with graph:
\[\graph\left((\preceq')^{f^{-1}}_U\right)=\bigcap_{f(U)\subset V \in \Open(X')}(f\times f)^{-1}(\graph(\preceq'_V))\cap(U\times U).\]
On the other hand, if $\preceq$ is a precirculation on $X$ and $f\colon X\to X'$ is a continuous map, we define the \emph{pushforward} $\preceq^{f_*}$ of $\preceq$ along $f$ as the precirculation on $X'$ assigning to each open subset $U$ of $X'$ the reflexive-transitive closure of the relation on $U$ with graph:
\[(f\times f)\left(\graph\left(\preceq_{f^{-1}(U)}\right)\right).\]

\subsection{} 
\label{subsec:prestream-topological-over-top}
The forgetful functor $U\colon \pStream\to \Topa$ is a topological functor (see \cite[Proposition 4.3]{gou14}). 
Indeed, given a family of prestreams $(X_i,\preceq^i)$ and a family of continuous maps $(f_i\colon X_i\to T)$ to a topological space $T$, a final lift of this data is given by the precirculation $\bigvee_{i\in I}\preceq^{f_{i,*}}$ on $T$, where $\bigvee_{i\in I}$ is defined sectionwise.
Dually, if $(f_{i}\colon T\to X_{i})$ is a family of continuous maps with source $T$, an initial structure of this data is given by the precirculation $\bigwedge_{i\in I}\preceq^{f^{-1}_{i}}$.
In particular, we have the following:

\begin{proposition}
\label{prop:pstream-wf-top-construct}
The category $\pStream$ of prestreams is a well fibered topological construct.
\end{proposition}
\begin{proof}
Since $\pStream\to \Topa$ is topological and $\Topa$ is a topological construct, $\pStream$ is a topological construct.
To conclude, it follows from the definitions that $\pStream$ has discrete terminal object and small fibres.
\end{proof}

\subsection{}
By Proposition \ref{prop:top-functor-adjoints}, the forgetful functor $U\colon \pStream\to \Topa$ has a left and a right adjoint.
If $X$ is a topological space, the \emph{discrete prestream} on $X$ is the prestream $(X,=)$ given by $x=_{U} x'$ if and only if $x=x'$, for every $U$ open subset of $X$.
Dually, the \emph{chaotic prestream} on $X$ is the prestream $(X,\equiv)$ defined by $x\equiv_{U}x'$ for every $(x,x')\in U\times U$ and every $U$ open in $X$.

\subsection{}
\label{subsec:preorder-prestream}
Every preordered space $(X,\le)$ naturally defines a precirculation $\restr{\le}{(\blank)}$ on $X$ by restriction. For every $U$ open in $X$, the graph of $\restr{\le}{U}$ is defined as:
\[\graph(\restr{\le}{U})=\graph(\le)\cap\left(U\times U\right).\]
This induces a full embedding $\iota\colon \pTop\to \pStream$ of the category of preordered spaces into the category of prestreams, which comes equipped with a left adjoint:
\[\Gamma\colon \pStream\to \pTop\]
defined by taking a prestream $(X,\preceq)$ to the \emph{global preorder} $(X,\preceq_X)$ on $X$.

\begin{example}
Let $(\II,\le)$ be the interval with the natural preorder structure inherited by the real numbers.
By the above paragraph we can see it as a prestream where $t\le_U t'$ if and only if $t\le t'$.
On the other hand, we denote by $\dI$ the \emph{directed interval}, which is the prestream $(\II, \direct{\le})$ where $t {\direct{\le}}_U t'$ if and only if $t\le t'$ and $[t,t']$ is contained in $U$.
Clearly, the identity map on $\II$ defines a morphism of prestreams $\dI\to(\II,\le)$. 
\end{example}

\begin{definition}
Let $X$ be a prestream and let $x_0$ and $x_1$ be points in $X$. A \emph{directed path} in $X$ from $x_0$ to $x_1$ is a morphism of prestreams $\alpha\colon \dI\to X$ such that $\alpha(i)=x_{i}$ for $i\in \{0,1\}$.
\end{definition}

\subsection{}
\label{subsec:dspace}
Recall that a \emph{d-space} is a pair $(X,dX)$, where $X$ is a topological space and $dX\subset \Topa(\II,X)$ is a set of paths in $X$, called \emph{d-paths}, which contains all constant paths and is closed under concatenation and monotonic reparametrisation (see \cite{gra03}). A morphism of d-spaces or \emph{d-map} $f\colon (X,dX)\to (X',dX')$ is a continuous map $f\colon X\to X'$ which preserves directed paths. 
The category $\dTop$ of directed spaces is a well fibered topological construct equipped with a topological functor to $\Topa$ (see for example \cite[Theorem 4.2]{faj08}).

\subsection{}
\label{subsec:idempotent-comonad-streams}
Let $(X,dX)$ be a d-space, we define a prestream $(X,\preceq^{dX})$ by the rule $x\preceq^{dX}_{U}x'$ if and only if there exists a d-path in $U$ from $x$ to $x'$.
Dually, if $(X,\preceq)$ is a prestream, we define a d-space $(X, d^{\preceq}X)$ with set of $d$-paths the directed paths in $(X,\preceq)$.
These constructions are functorial and define an adjunction (see \cite[Section 5]{hau12}):
\begin{equation}
\Adjoint{\dTop}{\pStream}{P}{D}
\end{equation}
In particular, by \ref{subsec:adjunction-monad}, the composition 
\[S=PD\colon \pStream\to \pStream\] 
comes equipped with the structure of a comonad $\mathbb{S}=(S,\nu,\epsilon)$.

\subsection{}
Given a prestream $(X,\preceq)$ the free coalgebra of $(X,\preceq)$ over $\mathbb{S}$ can be described explicitly as the prestream $SX=(X,\direct{\preceq})$ where $x_0 {\direct{\preceq}}_U x_1$ if and only if there exists a directed path in $(X,\preceq)$ from $x_0$ to $x_1$, whose image is contained in $U$, for every open subset $U$ in $X$ and every pair of points $x_0$ and $x_1$ in $U$. 
Notice that the underlying continuous map to the structure morphism $SX\to X$ of the free coalgebra on $X$  is the identity on $X$.

\begin{proposition}
The comonad $\mathbb{S}$ is an idempotent comonad. In particular, the category of coalgebras over $\mathbb{S}$ is a coreflective full subcategory of the category of prestreams.
\end{proposition}
\begin{proof}
See \cite[Proposition 5.25]{hau12}.
\end{proof}

\begin{definition}
\label{def:stream}
A  \emph{(Haucourt) stream} is a coalgebra over the comonad $\mathbb{S}$.
More explicitly, a prestream $(X,\preceq)$ is said to be a stream if, for every open subset $U$ of $X$ and every pair of points $x_0$ and $x_1$ in $U$ the inequality $x_0\preceq_{U} x_1$ holds if and only if there exists a directed path $\alpha\colon \dI\to (X,\preceq)$ from $x_0$ to $x_1$ whose image lies in $U$.
We denote by $\Stream$ the full subcategory of the category of prestreams spanned by streams.
\end{definition}

\begin{example}
The \emph{directed circle} $\dS^{1}$ is the quotient in $\Stream$ of the directed interval $\dI$ by its boundary $\partial\dI$. 
The underlying topological space of $\dS^{1}$ is the standard unit-circle $\SS^{1}$ with a distinguished basepoint.
The precirculation is given by counterclockwise directed paths (See \cite[Example 3.19]{kri09}, see also \cite[Examples 4.9-4.11]{gou14} for different models of the directed circle).
\end{example}

\begin{definition}
\label{def:circulation}
Let $X$ be a topological space. A precirculation $\preceq$ on $X$ is said to be a \emph{circulation} if, for every open subset $U$ in $X$ and every open cover $\{U_i\}_{i\in I}$ of $U$ the following identity holds:
\begin{equation}
\label{eq:circulation}
\graph\left(\preceq_U\right)=\graph\left(\bigvee_{i\in I}\preceq_{U_i}\right).
\end{equation}
\end{definition}

\begin{proposition}
The category $\Stream$ of streams is a well fibered topological construct.
\end{proposition}
\begin{proof}
This follows from Proposition \ref{prop:coreflective-subcat-top-cat}.
\end{proof}

\begin{lemma}
The precirculation of every Haucourt stream $(X,\preceq)$ is a circulation.
\end{lemma}
\begin{proof}
See \cite[Lemma 4.15]{gou14}
\end{proof}

\begin{remark}
A pair $(X, \preceq)$ where $X$ is a topological space and $\preceq$ is a circulation is called a (Krishnan) stream and it was first introduced by Sanjeevi Krishnan in \cite{kri09}.
We have a chain of coreflective full embeddings:
\[\Stream\subset \Stream_{\mathrm{Kri}}\subset \pStream\]
where we denote by $\Stream_{\mathrm{Kri}}$ the category of Krishnan streams.
Although \cite{kri09} is the original source on streams, in this text we decided to take a slighlty different approach since we will take $\Stream$ as our ``ambient category'' for the category of locally stratified spaces, in Section \ref{ch:5sec:2}.
\footnote{The reason of this choice lies in the fact that, by definition, Haucourt streams have the ``homotopical flavour'' of d-spaces, as well as the ``order-theoretic flavour'' of prestreams.}
\end{remark}

We conclude the section with a brief recollection on exponentiable streams, following \cite{gou14}.

\subsection{}
\label{subsec:stream-exponential}
Following \cite{gou14}, a stream $X$ is said to be \emph{core-compact} if the topological space underlying $X$ is core-compact.
For a core-compact stream $X$ and a stream $Y$, we can consider the set $\Stream(X,Y)$ of stream morphisms from $X$ to $Y$, endowed with the subspace topology induced by the core-open topology on $\Topa(X,Y)$ (see \ref{subsec:top-exponential}).
Then, we define a precirculation $\sqsubseteq$ on $\Stream(X,Y)$ as follows. If $W$ is an open subset of ${\Stream}(X,Y)$ we say that $f\sqsubseteq_{W} g$ if and only if for all open subsets $U$ of $X$ and $V$ of $Y$ such that $W\times U$ is contained in  $\ev_{X,Y}^{-1}(V)$, given any two points $x,x'\in U$ such that $x\preceq^X_U x'$, one has $f(x)\preceq^Y_V g(x')$.

\begin{remark}
\label{rmk:sierpinski-stream}
Following \cite{gou14} we define the (chaotic) \emph{\Sierpinski{} prestream} as the chaotic prestream $\ind{[1]}$ associated to the \Sierpinski{} space (see the paragraph before \cite[Lemma 5.1]{gou14}).
In contrast with the remark made by the author of \cite{gou14} in the paragraph before \cite[Theorem 6.2]{gou14}, the \Sierpinski{} prestream is actually a Haucourt stream.
Indeed, the quotient map $p\colon [0,1]\to[1]$ defined by 
\begin{align*}
p\colon [0,1]&\to [1]\\
t &\mapsto \begin{cases}
0 & t=0\\
1 & \text{ otherwise.}
\end{cases}
\end{align*}
and the composition $pr$ of $p$ with the reverse path $r\colon [0,1]\to[0,1]$ mapping $t$ to $1-t$, define morphisms of prestreams $p\colon \dI \to \ind{[1]}$.
\end{remark}

The next result is a strengthening of \cite[Theorem 6.2]{gou14} that follows from the previous remark.

\begin{theorem}
\label{thm:stream-exponentiable}
A stream is core-compact if and only if it is exponentiable in $\Stream$.
Moreover, for a core-compact stream $X$ and a stream $Y$, the exponential $Y^X$ from $X$ to $Y$ is given by the stream associated to the prestream defined in \ref{subsec:stream-exponential}. 
\end{theorem}
\begin{proof}
By Remark \ref{rmk:sierpinski-stream} we can apply \cite[Lemma 5.1]{gou14} for one implication.
The other implication is \cite[Theorem 6.2]{gou14}.
\end{proof}

\begin{corollary}
\label{cor:stream-core-compact-cartesian-closed}
Let $\cat{I}$ be a class of core-compact streams such that any binary product of objects in $\cat{I}$ is $\cat{I}$-generated.
Then, the category $\Stream_{\cat{I}}$ is Cartesian closed.
\end{corollary}
\begin{proof}
The proof follows immediately from Theorem \ref{thm:stream-exponentiable} and Theorem \ref{thm:final-closure-cartesian-closed}.
\end{proof}

\section{Locally stratified spaces: definitions and examples\label{ch:5sec:2}}
We define the category of locally stratified spaces as a full subcategory of the category of streams.
We show that the stream associated to the stratified standard $n$-simplex is locally stratified.
Since locally stratified spaces are closed under colimits, this assignment induces an adjunction between the category of locally stratified spaces and the category of simplicial sets (\ref{subsec:locstrat-sset-adjunction}).
Moreover, we construct a locally presentable and Cartesian closed category of locally stratified spaces (Theorems \ref{thm:locstrat-locally-presentable} and \ref{thm:locstrat-cartesian-closed}), that we call the category of numerically generated locally stratified spaces.
As a consequence, the category of numerically generated locally stratified spaces is $\sSet$-enriched, tensored and cotensored (Theorem \ref{thm:locstrat-enriched-tensored-cotensored}).
We conclude the section by giving a characterisation of numerically generated locally stratified spaces in terms of local exit paths (Lemma \ref{lem:numerically-generated-local-exit-path}), and we give a concrete description of the adjunctions with the categories of numerically generated stratified and topological spaces (Proposition \ref{prop:delta-stream-associated-to-delta-prespace} and Corollary \ref{cor:forgetful-codiscrete-adjunction-locstrat}).

\begin{definition}
A \emph{locally stratified space} is a stream $(X,\preceq)$ such that for every open subset $U$ in $X$ and every subset $A$ of $U$, which is upward closed with respect to the preorder $\preceq_U$, the subset $A$ is open in $X$.
We denote by $\LocStrata$ the full subcategory of the category of streams spanned by the locally stratified spaces.
\end{definition}

\begin{lemma}
\label{lem:stream-upward-closed}
Let $f\colon X\to Y$ be a morphism of streams, let $U$ be an open subset of $Y$ and assume that $A$ is upward closed in $U$. Then, $f^{-1}(A)$ is upward closed in $f^{-1}(U)$.
\end{lemma}
\begin{proof}
Let $x_0$ be a point in $f^{-1}(A)$ and assume that $x_0\preceq_{f^{-1}(U)} x_1$ for a point $x_1\in f^{-1}(U)$. 
Then, we have that $f(x_0)\preceq_{U} f(x_1)$ which implies that $f(x_1)$ belongs to $A$, since $A$ is upward closed.
Therefore, $x_1$ belongs to $f^{-1}(A)$ and the proof is complete.
\end{proof}

\subsection{}
The \emph{locally stratified standard $n$-simplex} is the stream $\dabs{\Delta^n}$ associated to the stratified standard $n$-simplex via the composition 
\begin{equation*}
\begin{tikzcd}
\pTop \ar[r, "\iota"]&
\pStream \ar[r, "S"]&
\Stream 
\end{tikzcd}
\end{equation*}
where $\iota$ is the embedding defined in \ref{subsec:preorder-prestream} and $S$ is the coreflection functor. 
Concretely, $\dabs{\Delta^{n}}$ is the stream $(\abs{\Delta^{n}}, \direct{\le})$ with underlying space the geometric standard $n$-simplex and circulation given by: $x\direct{\le}_{U}x'$ if and only if there exists a directed path $\dI\to \sabs{\Delta^{n}}$ from $x$ to $x'$ whose image is contained in $U$.

\begin{lemma}
\label{lem:convex-preorder-path-preorder}
For every convex open subset $C$ of $\dabs{\Delta^n}$, the preorder $\direct{\le}_{C}$ coincides with the restriction $\restr{\le}{C}$ of the preorder $\le$ on $C$.
\end{lemma}
\begin{proof}
It is enough to show that, given any two points $x$ and $y$ in a subset $C$ convex in $\dabs{\Delta^n}$ with $x\le y$, there exists a directed path $\alpha$ lying entirely in $C$ and connecting $x$ and $y$. 
By describing $\dabs{\Delta^n}$ in barycentric coordinates, we see that if $x$ and $y$ are points in $\dabs{\Delta^n}$, the line segment connecting $x$ and $y$ is a directed path in $\dabs{\Delta^n}$.
In particular, since $C$ is convex, if $x$ and $y$ belong to $C$, such a line segment is contained in $C$, as required.
\end{proof}

\begin{proposition}
\label{prop:delta-locally-stratified}
For every finite ordinal $[n]$, the stream $\dabs{\Delta^n}$ is a locally stratified space.
\end{proposition}
\begin{proof}
In order to show that $\dabs{\Delta^{n}}$ is a locally stratified space it is enough to show that for every open subset $U$ in $\dabs{\Delta^{n}}$ and every point $x\in U$, the subset 
\[U_{\succeq x}=\{y\in U: x\preceq_{U}y\}\]
is open in $U$. 
Let $U$ be an open subset in $\dabs{\Delta^n}$ and write $U$ as a union $U=\bigcup_{i\in I} B_i$ of open balls.
For every open ball $B$ in $\dabs{\Delta^n}$ and every point $x\in B$, as $B$ is convex in $\dabs{\Delta^n}$, the subset $B_{\succeq x}$ is open in $B$  by Lemma \ref{lem:convex-preorder-path-preorder}.
In particular, $B_{\succeq x}$ is equal to the set $B_{\ge k}$, where $x$ belongs to the $k$-th stratum of $B$.
Moreover, the path preorder $\preceq_{U}$ on $U$ coincides with the join $\bigvee_{i\in I}\preceq_{B_i}$ of the path preorders on each ball.
We claim that, for every point $x\in U$, the following equality holds:
\begin{equation}
\label{eq:loc-strat-1}
U_{\succeq x}=\bigcup_{i\in I} (B_i)_{\ge k_i}
\end{equation}
for some integers $k_i\in [n+1]$, where we set $(B_i)_{\ge n+1}=\emptyset$ for every $i\in I$.
For every $i\in I$ if there is no directed path in $U$ joining $x$ with a point of $B_i$, we take $k_i$ to be $n+1$.
On the other hand, if there exists a path as above, we take $k_i$ to be the minimum such that $x \preceq_U y_{k_i}$, for some $y_{k_i}$ in the stratum $B_{k_i}$.
To show that the right hand side of $\eqref{eq:loc-strat-1}$ is contained in $U_{\succeq x}$ notice that, given a point $y$ in $B_{\ge k_i}$, we have $y_{k_i}\preceq_{B_i} y$ and $x\preceq_{U}y_{k_i}$, which implies $x\preceq_U y$.
To show the opposite inclusion, given $y\in U_{\succeq x}$ there exists some $B_i$ such that $y\in B_i$ and, by our choice of $k_i$, we must have that $y$ belongs to a shallower stratum than $(B_i)_{k_i}$.
\end{proof}

\begin{proposition}
\label{prop:colimit-loc-strat}
The inclusion functor $\iota\colon\LocStrata\to\Stream$ creates colimits.
In particular, $\LocStrata$ is a cocomplete category.
\end{proposition}
\begin{proof}
Since $\LocStrata\to \Stream$ is a fully faithful functor it suffices to show that a colimit of a diagram of locally stratified spaces in $\Stream$ is locally stratified.
Let $X$ be a stream and assume that $X$ has the final structure with respect to a family of maps $(f_{i}\colon X_{i}\to X)_{i\in I}$ such that $X_{i}$ is locally stratified for every $i\in I$.
Let $U$ be an open subset of $X$ and let $A$ be an upward closed subset in $U$ with respect to $\preceq_U$.
By definition of the final structure, it is enough to show that $f_{i}^{-1}(A)$ is open in $X_{i}$ for every $i\in I$.
Since $X_{i}$ is a locally stratified space it is enough to show that $f^{-1}(A)$ is upward closed in $f^{-1}(U)$ with respect to the preorder $\preceq_{f^{-1}(U)}$, which is true by Lemma \ref{lem:stream-upward-closed}.
\end{proof}

\subsection{}
Let $\cat{I}$ be the full subcategory of $\Stream$ spanned by locally stratified standard simplices.
An object of the final closure $\Stream_{\cat{I}}$ of $\cat{I}$ in $\Stream$ is a locally stratified space by Propositions \ref{prop:delta-locally-stratified} and \ref{prop:colimit-loc-strat}.

\begin{definition}
A locally stratified space $X$ is said to be \emph{numerically generated} if it belongs to the final closure of $\cat{I}$ in $\Stream$.
We denote by $\LocStrat$ the full subcategory of numerically generated locally stratified spaces.
\end{definition}

\begin{theorem}
\label{thm:locstrat-locally-presentable}
The category $\LocStrat$ is a locally presentable, coreflective full subcategory of the category $\Stream$ of streams.
\end{theorem}
\begin{proof}
This follows from Theorem \ref{thm:C_I-locally-presentable} and Proposition \ref{prop:C_I-coreflective}.
\end{proof}

\subsection{} 
\label{subsec:locstrat-sset-adjunction}
By Proposition \ref{prop:colimit-loc-strat}, the functor:
\begin{align*}
\DDelta	& \to \LocStrata\\
[n] 		&\mapsto \dabs{\Delta^{n}}
\end{align*}
induces an adjunction:
\begin{equation}
\label{eq:sset-locstrata-adjunction}
\Adjoint{\sSet}{\LocStrata}{\dabs{\blank}}{\dSing}
\end{equation}
whose left adjoint we call \emph{locally stratified geometric realisation functor} and the right adjoint we call the \emph{locally stratified singular simplicial set functor}.
Moreover, \eqref{eq:sset-locstrata-adjunction} restricts to an adjunction
\begin{equation}
\label{eq:sset-locstrat-adjunction}
\Adjoint{\sSet}{\LocStrat}{\dabs{\blank}}{\dSing}
\end{equation}
with the category of numerically generated locally stratified spaces.

\begin{proposition}
\label{prop:global-section-realisation}
Let $\Gamma\colon \Stream\to \pTop$ be the global preorder functor and let $K$ be a simplicial set. 
Then, there exists a natural isomorphism:
\[\Gamma\left(\dabs{K}\right)\simeq \sabs{K}.\]
\end{proposition}
\begin{proof}
The claim is true for the realisation of standard simplices, by Lemma \ref{lem:convex-preorder-path-preorder}. 
The general statement follows from the fact that $\Gamma$ commutes with colimits.
\end{proof}

\begin{example}
The \emph{locally stratified $n$-spine} is the realisation $\dabs{\Sp^n}$ of the $n$-spine (see Example \ref{ex:n-spine}). 
Since realisation preserves colimits, $\dabs{\Sp^n}$ is naturally isomorphic to the colimit given by glueing $n$ ordered copies of $\dabs{\Delta^1}$, where the endpoint of a copy of $\dabs{\Delta^1}$ is glued to the starting point of the next copy.
Equivalently, by the description of colimits in prestreams (see \ref{subsec:prestream-topological-over-top}), $\dabs{\Sp^{n}}$ is the stream associated to the stratified space $\sabs{\Sp^{n}}$. 
In particular, multiplication by $n$:
\begin{align*}
n \colon [0,1] &\to [0,n]\\
t\mapsto n\cdot t
\end{align*}
determines a morphism  
\[n \colon \dI\to \dabs{\Sp^{n}}\]
of streams.
\end{example}

\begin{definition}
Let $X$ be a prestream and let $x$ and $x'$ be points in $X$. 
An \emph{elementary local exit path} in $X$ from $x$ to $x'$ is a morphism of prestreams $\alpha\colon\dabs{\Delta^1}\to X$ such that $\alpha(0)=x$ and $\alpha(1)=x'$. 
A \emph{local exit path} in $X$ from $x$ to $x'$ is a morphism $\alpha\colon \dabs{\Sp^n}\to X$ such that $\alpha(0)=x$ and $\alpha(n)=x'$. 
\end{definition}

\begin{example}
\label{ex:locally-strat-circle}
The locally stratified realisation $\dabs{S^{1}}$ of the simplicial circle (see Example \ref{ex:simplicial-circle}) is called the \emph{locally stratified circle} pointed at $0$.
The local preorder around every point $x\in \dabs{S^{1}}$ distinct from the origin is the chaotic preorder.
On the other hand, directed paths crossing $0$ can only do so counterclockwise.
In other words, for any open arc $U_{\epsilon}=(-\epsilon,\epsilon)$ around $0$, the local preorder $\preceq_{U_{\epsilon}}$ on $U_{\epsilon}$ corresponds to the stratification $U_\epsilon\to [1]$ mapping $(-\epsilon, 0]$ to $0$ and $(0,\epsilon)$ to 1.
\begin{figure}[h]
\centering
\begin{tikzpicture}

\draw [thick] (0,0) circle [radius=1];

\draw [thick, fill] (-1,0) circle [radius=0.05];
\node at (-0.7,0) {$0$};

\draw [thick, ->] (1,-0.1) -- (1,0.1);
\node at (0.7,0) {$f$};

\draw [gray,  dashed] (-2,0) circle [radius=0.5];
\draw[gray, dashed] (-1.5,0) -- (-1,0);

\draw[red, thick] (-2,0.5) -- (-2,0);
\draw[green, thick] (-2,0) -- (-2,-0.5);
\draw [red, thick, fill] (-2,0) circle [radius=0.05];

\draw [gray, dashed] (2,0) circle [radius=0.5];
\draw[gray, dashed] (1,0) -- (1.5,0);

\draw[red, thick] (2,-0.5) -- (2,0.5);

\end{tikzpicture}
\caption{The locally stratified circle}\label{figure}
\end{figure}
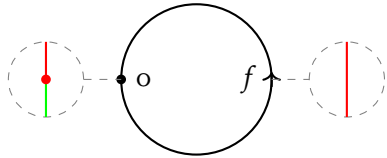
\end{example}

\begin{lemma}
\label{lem:product-realisation-stream}
The canonical morphism of streams 
\[\dabs{\Delta^n\times\Delta^m}\to \dabs{\Delta^n}\times\dabs{\Delta^m}\]
is an isomorphism.
\end{lemma}
\begin{proof}
Both functors $\iota\colon\pTop\to \pStream$ and $S\colon \pStream\to \Stream$ preserve products, since they are right adjoints. Therefore, we conclude by Lemma \ref{lem:product-stratified-simplices}.
\end{proof}

\begin{lemma}
\label{lem:locstratified-standard-simplex-exponentiable}
The locally stratified strandard $n$-simplex $\dabs{\Delta^{n}}$ is an exponentiable stream.
\end{lemma}
\begin{proof}
This follows immediately from Theorem \ref{thm:stream-exponentiable}.
\end{proof}

\begin{theorem}
\label{thm:locstrat-cartesian-closed}
The category $\LocStrat$ of numerically generated locally stratified spaces is a Cartesian closed category.
\end{theorem}
\begin{proof}
The class of realisations of standard simplices satisfies the hypotheses of Corollary \ref{cor:stream-core-compact-cartesian-closed} by Lemma \ref{lem:product-realisation-stream}.
\end{proof}

\begin{theorem}
\label{thm:locstrat-enriched-tensored-cotensored}
The category $\LocStrat$ is naturally an $\sSet$-enriched tensored and cotensored category via \ref{eq:sset-locstrat-adjunction}.
Moreover, the adjunction \ref{eq:sset-locstrat-adjunction} is an $\sSet$-enriched adjunction.
\end{theorem}
\begin{proof}
The claim follows from Proposition \ref{prop:cisinski-topological-enriched}.
\end{proof}

\subsection{} 
Concretely, for every pair of numerically generated locally stratified spaces $X$ and $Y$ the mapping space of $X$ and $Y$ is the simplicial set:
\[\map(X,Y)=\dSing\ihom(X,Y),\]
where $\ihom(X,Y)$ denotes the internal-hom in $\LocStrat$.
For every simplicial set $K$, the tensor of $X$ with $K$ is the locally stratified space:
\[X\otimes K=X\times\dabs{K}\]
and dually, the cotensor of $Y$ with $K$ is given by:
\[Y^K=\ihom(\dabs{K},Y).\]

\begin{lemma}
\label{lem:numerically-generated-local-exit-path}
Let $(X,\preceq)$ be a preordered space and assume that $X$ is a numerically generated topological space.
Then, $(X,\preceq)$ is a numerically generated locally stratified space if and only if for any two points $x$ and $x'$ in an open subset $U$ of $X$, we have that $x\preceq_{U} x'$ if and only if there exists a local exit path from $x$ to $x'$ in $X$ whose image is contained in $U$.
\end{lemma}
\begin{proof}
Assume $(X,\preceq)$ is numerically generated. 
By \ref{subsec:stream-initial-final}, there exist integers $(i_{1},\ldots, i_{n})$ and a set
\[\left\{p_{j}^{\epsilon} : \epsilon \in \{0,1\}, p_{j}^{\epsilon}\in \dabs{\Delta^{i_{j}}}\right\}\]
together with morphisms of streams $f_{j}\colon \dabs{\Delta^{i_{j}}}\to (X,\preceq)$ such that 
\begin{align}
p_{j}^{0}			&\le_{f_{j}^{{-1}}U} p_{j}^{1}				\label{eq:circulation-final-structure-1}\\
f_{j}\left(p_{j}^{1}\right)	&=f_{j+1}\left(p_{j+1}^{0}\right)		\label{eq:circulation-final-structure-2}\\
f_{1}\left(p_{1}^{0}\right)	=x &,\; f_{n}\left(p_{n}^{1}\right)=x'.			\label{eq:circulation-final-structure-3}
\end{align}
where $\le_{f_{j}^{-1}U}$ denotes the precirculation of $\dabs{\Delta^{i_{j}}}$ evaluated at $f^{-1}_{j}U$. 
In particular, there exists a local exit path from $p_{j}^{0}$ to $p_{j}^{1}$ in $\dabs{\Delta^{i_{j}}}$ whose image is contained in $U$, which defines a local exit path:
\[\alpha_{j}\colon\dabs{\Sp^{k_{j}}}\to \dabs{\Delta^{i_{j}}}\to X\]
in $X$, for every $j$.
Then, \eqref{eq:circulation-final-structure-2} and  \eqref{eq:circulation-final-structure-3} imply that the $\alpha_{j}'s$ glue to an exit path $\sabs{\Sp^{m}}\to X$ from $x$ to $x'$.

Conversely assume that the circulation $\preceq$ on $X$ is given by local exit paths and let us consider a cocone $(g_{\alpha}\colon \dabs{\Delta^{n^{\alpha}}}\to (X',\preceq'))$ indexed over the cocone of all locally stratified $n$-simplices mapping to $X$. 
Let $f\colon X\to X'$ be a continuous map such that for every morphism $\alpha\colon \dabs{\Delta^{n^{\alpha}}}\to (X,\preceq)$, the following diagram of topological spaces is commutative:
\begin{equation}
\label{eq:lifting-numerically-generated-locstrat}
\begin{tikzcd}
\abs{\Delta^{n^{\alpha}}}
	\ar[r, "\alpha"]
	\ar[dr, "g_{\alpha}"']&
X	\ar[d, "f"]\\
&
X'
\end{tikzcd}.
\end{equation}
We need to show that $f$ defines a map of streams.
Let $U'$ be an open subset of $X'$ and assume that $x\preceq_{f^{{-1}}U'} x'$ in $X$, then by assumption there exists a local exit path $s\colon\dabs{\Sp^{n}}\to (X,\preceq)$ from $x$ to $x'$ whose image is contained in $f^{-1}(U')$.
In particular, this defines elementary local exit paths $g_{\alpha_{k}}\colon \dabs{\Delta^{1}}\to (X',\preceq')$ for $k=1,\ldots, n$ with images contained in $U'$ and such that 
\[\alpha_{1}(0)=x,\; \alpha_{n}(1)=x'\quad \text{ and }\quad \alpha_{k}(1)=\alpha_{k-1}(0).\]
Hence, there is a chain of inequalities:
\[fx\preceq'_{U'} g_{\alpha_{1}}(1)= g_{\alpha_{2}}(0)\preceq_{U'}'\ldots \preceq_{U'}' g_{\alpha_{n-1}}(1)=g_{\alpha_{n}}(0)\le' fx'.\]
Which implies that $fx\preceq_{U'}fx'$ and we are done.
\end{proof}

\begin{proposition}
\label{prop:open-numerically-generated}
Let $(X,\preceq)$ be a numerically generated locally stratified space and let $i\colon U\subset X$ be an open subset of $X$.
Then, $U$ endowed with the subspace structure is a numerically generated locally stratified space.
\end{proposition}
\begin{proof}
By \cite[Proposition 1.18]{dug03} $U$ is a numerically generated topological space. Therefore, by Lemma \ref{lem:numerically-generated-local-exit-path} it is enough to show that the local order of $U$ can be detected by local exit paths.
Since $i\colon U\to X$ is an open inclusion, the precirculation $\preceq^{i^{-1}}$ coincides with the restriction of $\preceq$ to $U$.
As the underlying topological space of $X$ is numerically generated, the claim follows from Lemma \ref{lem:numerically-generated-local-exit-path}.
\end{proof}

\subsection{}
\label{subsec:global-strat-space}
Since the global preorder on $\dabs{\Delta^n}$ coincides with $\sabs{\Delta^n}$, the global preorder functor, restricted to the category of numerically generated locally stratified spaces, defines a functor:
\[\Gamma\colon\LocStrat\to \Strat\]
which has a right adjoint $R\colon \Strat\to \LocStrat$, given by the composition of the functor $\pTop\to \Stream$ with the coreflection $\Stream\to \LocStrat$.
We now give a more explicit description of $R$.

\subsection{} Given a numerically generated stratified space $(X,\le)$, consider the prestream $(X, \le^\Delta)$ with $x\le^\Delta_U x'$ if and only if there exist an exit path $\alpha\colon \sabs{\Sp^n}\to (X,\le)$ whose image lies in $U$.
Notice that the identity on $X$ defines a morphism $\epsilon_X\colon(X,\le^\Delta)\to (X,\direct{\le})$ of prestreams, to the stream associated to $(X,\le)$.
In particular, we have that the identity on $X$ induces morphisms of prestreams:
\[(X,\le^\Delta)\to (X,\direct{\le})\to (X,\restr{\le}{\blank})\]

\begin{proposition}
\label{prop:delta-stream-associated-to-delta-prespace}
Let $(X,\le)$ be a numerically generated stratified space, and let 
\[R\colon \Strat\to \LocStrat\]
be the right adjoint to the global preorder functor. 
Then 
\[R(X,\le)=(X,\le^\Delta).\]
\end{proposition}
\begin{proof}
We show the claim by applying Lemma \ref{lem:coreflection-I-generated}.
For every $[n]\in \DDelta$, we have natural isomorphisms:
\begin{align}
\label{eq:delta-ptop-delta-str}
\Stream\left(\dabs{\Delta^n}, (X,\le^\Delta)\right)
& \simeq \Stream\left(\dabs{\Delta^n}, (X,\direct{\le})\right)\\
& \simeq \Strat\left(\sabs{\Delta^n},(X,\le)\right)
\end{align}
where the second isomorphism holds by adjunction. 
To prove the first isomorphism,  a morphism of streams $\alpha\colon\dabs{\Delta^n}\to (X,\direct{\le})$ defines a morphism of prestreams $\alpha\colon\dabs{\Delta^n}\to (X,\restr{\le}{(\blank)})$ to the prestream associated to the preordered space $(X,\le)$.
Now, for every open subset $U$ in $X$, given any two points $x$ and $x'$ in $\alpha^{-1}(U)$ such that $x\preceq_{\alpha^{-1}(U)}x'$ we can find a local exit path $\gamma \colon \dabs{\Sp^n}\to \dabs{\Delta^n}$ from $x$ to $x'$ whose image is contained in $\alpha^{-1}(U)$. 
Hence, $\alpha\gamma\colon \dabs{\Sp^n}\to \left(X,\restr{\le}{(\blank)}\right)$ defines an exit path in $(X,\le)$ from $\alpha(x)$ to $\alpha(x')$ whose image is contained in $U$.
Therefore $\alpha(x)\le^\Delta_U\alpha(x')$, which proves the claim.
To conclude, the stream $(X,\le^{\Delta})$ is a numerically generated locally stratified space by Lemma \ref{lem:numerically-generated-local-exit-path}.
\end{proof}

\begin{corollary}
\label{cor:forgetful-codiscrete-adjunction-locstrat}
The forgetful functor $U\colon \Stream\to \Top$ together with the codiscrete object functor $\ind{(\blank)}\colon\Top\to \Stream$ restrict to an adjunction
\begin{equation}
\Adjoint{\LocStrat}{\Top}{U}{\ind{(\blank)}}
\end{equation}
\end{corollary}
\begin{proof}
Since the forgetful functor $U\colon \Str\to \Topa$ factors through $\pTop$, the composition $R\pi\colon \Top\to \LocStrat$ is right adjoint to the forgetful $\LocStrat\to \Top$, where $\pi$ is the coreflection functor defined in \ref{subsec:coreflection-top-strat}.
Therefore, by Proposition \ref{prop:delta-prespace-associated-to-delta-space} and Proposition \ref{prop:delta-stream-associated-to-delta-prespace}, it is enough to show that the stream $(X,\sim^\Delta)$ is isomorphic to the codiscrete stream $\ind{X}$. 
To conclude, notice that $x\sim^{\Delta}_{U}x'$ if and only if there exists a path $\alpha\colon \II\to X$ from $x$ to $x'$ whose image lies in $U$.
\end{proof}

\begin{remark}
The reader is invited to compare Corollary \ref{cor:forgetful-codiscrete-adjunction-locstrat} with the analogous result for stratified spaces (Proposition \ref{prop:delta-prespace-associated-to-delta-space}).
The difference between the two results lies in the fact that local preorder in streams can be detected by directed paths.
Notice indeed that, if we take $\LocStrat$ as a coreflective subcategory of the category of prestreams, the right adjoint $\Strat\to \LocStrat$ does not coincide with the codiscrete object functor.
\end{remark}

\subsection{} Given a numerically generated topological space $X$, the codiscrete stream $\ind{X}$ associated to $X$ will be simply denoted by $X$, when no confusion arises. For example, we denote by $\II$ the codiscrete stream associated to the interval in $\Top$ and call it the \emph{chaotic interval}.
\section{The homotopy theory of locally stratified spaces}\label{ch:5sec:3}
Using the adjunction between simplicial sets and locally stratified spaces, we define a candidate model structure on the category of locally stratified spaces.
By the results in Section \ref{ch:2sec:4} we are able to show that the category of locally stratified spaces that are fibrant and numerically generated, has the structure of a category of fibrant objects (Theorem \ref{thm:loc-strat-fibrant-objects}),
where path objects are defined using the locally stratified realisation of the interval $J$.
We define a chaotic homotopy equivalence to be an $\II$-equivalence, where $\II$ is the chaotic interval and we show that chaotic homotopy equivalences are weak equivalences (Proposition \ref{prop:chaotic-equivalence-weak-equivalence}).

\begin{definition}
A morphism $f\colon X\to Y$ between locally stratified spaces is said to be a \emph{fibration} (resp. a \emph{weak equivalence}) if the induced map
\[\dSing(f)\colon \dSing(X)\to \dSing(Y)\]
is a fibration (resp. a weak equivalence) in the Joyal model structure.
We say that $f$ is a trivial fibration if it is both a fibration and a weak equivalence.
\end{definition}

\begin{definition}
Let $i\colon A\to B$ be a morphism of locally stratified spaces.
We say that $i$ is be a \emph{cofibration} (\resp a \emph{trivial cofibration}) if it has the left lifting property with respect to all trivial fibration (\resp all fibrations).
We call $i$ an \emph{acyclic cofibration} if it is both a cofibration and a weak equivalence.
\end{definition}

\begin{definition}
Let $X$ be a locally stratified space.
We say that $X$ is \emph{fibrant} if $\dSing(X)$ is an $\infty$-category.
We call $X$ \emph{cofibrant} if the unique morphism $\initial\to X$ from the initial object is a cofibration.
\end{definition}

\begin{lemma}
Let $X$ be a locally stratified space and let $\Gamma X$ be the stratified space associated to $X$.
Then, there exists a natural map of simplicial sets:
\[\iota_X\colon \dSing(X)\to \sSing(\Gamma X)\]
which is a monomorphism, for every locally stratified space $X$.
\end{lemma}
\begin{proof}
Level-wise, the map $\iota_X$ is defined as:
\[\LocStrata\left(\dabs{\Delta^n}, X\right)\to \Strata\left(\Gamma\dabs{\Delta^n}, \Gamma X\right)\cong\Strata\left(\sabs{\Delta^n}, \Gamma X\right).\]
To conclude, notice that the functor $\Gamma$ is faithful, and so $\iota_X$ is a monomorphism.
\end{proof}

\begin{proposition}
\label{prop:locstrat-standard-simplex-strat}
For every $[n]\in \DDelta$, the induced map:
\[\iota_{\dabs{\Delta^n}}\colon \dSing\dabs{\Delta^n}\to \sSing\sabs{\Delta^n}\]
is an isomorphism of simplicial sets. 
In particular, the locally stratified realisation of every standard simplex is a fibrant locally stratified space.
\end{proposition}
\begin{proof}
It is enough to prove that the map induced by $\Gamma$:
\[\LocStrata(\dabs{\Delta^m},\dabs{\Delta^n})\to \Strata(\sabs{\Delta^m},\sabs{\Delta^n})\]
is an isomorphism, for every $[m]\in \DDelta$. This is true by adjunction.
\end{proof}

\begin{remark}
Thanks to Lemma \ref{lem:product-realisation-stream} and Lemma \ref{lem:locstratified-standard-simplex-exponentiable} the functor $\dabs{\blank}\colon \DDelta\to \pTop$ satisfies the hypothesis described in \ref{subsec:cisinski-topological}.
In particular, as in Section \ref{ch:4sec:3} we can apply the results of Section \ref{ch:2sec:4} as we do in the following.
\end{remark}

\begin{lemma}
\label{lem:locstrat-homotopical-cat}
The category $\LocStrata$ equipped with the class of weak equivalences has the structure of a homotopical category.
\end{lemma}
\begin{proof}
See Lemma \ref{lem:top-cat-homotopical-cat}.
\end{proof}

\begin{lemma}
\label{lem:locstrat-prod-we}
Let $X$ be a locally stratified space and let $w\colon Y\to Z$ be a weak equivalence between locally stratified spaces, then the induced map:
\[X\times w\colon X\times Y\to X\times Z.\]
is a weak equivalence.
\end{lemma}
\begin{proof}
See Lemma \ref{lem:top-cat-prod-we}.
\end{proof}

\begin{proposition}
\label{prop:cotensor-Quillen-functor-streams}
Let $i\colon K\to L$ be a monomorphism between simplicial sets and let $p\colon X\to Y$ be a fibration in $\LocStrat$. Then, the induced map
\[ X^L\to X^K\times_{Y^K}Y^L\]
is a fibration, which is acyclic if $i$ or $p$ is so.
\end{proposition}
\begin{proof}
See Proposition \ref{prop:cotensor-Quillen-functor}.
\end{proof}

\begin{corollary}
Let $p\colon X\to Y$ be a fibration in $\LocStrat$ and let $i\colon A\to B$ be a cofibration.
Then, the induced map
\[\ihom(B,X)\to \ihom(A,X)\times_{\ihom(A,Y)}\ihom(B,Y)\]
is a fibration, which is acyclic if either $p$ is acyclic or $i$ is a trivial cofibration.
\end{corollary}
\begin{proof}
The claim follows from Lemma \ref{lem:two-variable-lifting} and Proposition \ref{prop:cotensor-Quillen-functor-streams}.
\end{proof}

\subsection{}
Recall that $J$ is an interval object for the Joyal model structure on $\sSet$ (see Example \ref{ex:J-interval-sset}).
In other words, $J$ comes equpped with a factorisation of the codiagonal:
\[\terminal\coprod\terminal\to J\to\terminal\]
such that $(\partial_0,\partial_1)\colon *\coprod *\to J$ is a monomorphism and $\sigma\colon J\to *$ is a Joyal weak equivalence. 
Then, for every numerically generated locally stratified space $X$, applying the functor $X^{(\blank)}$ we get a factorisation
\[X\to X^J\to X\times X\]

\begin{lemma}
\label{lem:path-object-for-loc-strat}
For every fibrant locally stratified space $X$ in $\LocStrat$, the above factorisation induces the structure of a path object $X^J$ for $X$.
\end{lemma}
\begin{proof}
This follows immediately from Lemma \ref{lem:path-object-for-top-cat}.
\end{proof}

\begin{theorem}
\label{thm:loc-strat-fibrant-objects}
The full subcategory $\LocStrat_{f}$ of $\LocStrat$ spanned by the fibrant locally stratified spaces is a category of fibrant objects.
\end{theorem}
\begin{proof}
This follows from Proposition \ref{prop:top-cat-fibrant-objects} and Lemma \ref{lem:path-object-for-loc-strat}.
\end{proof}

\subsection{}
The chaotic interval $\II$ is a numerically generated locally stratified space that fits into a diagram:
\begin{equation}
\label{eq:chaotic-interval-locstrat}
\begin{tikzcd}
\terminal\coprod\terminal \ar[r, "{(\partial^0,\partial^1)}"] &
\II \ar[r, "\sigma"] &
\terminal
\end{tikzcd}
\end{equation}
whose composition is the codiagonal $(1_\terminal,1_\terminal)\colon\terminal\coprod\terminal\to \terminal$.
Notice that, since $\dSing(\II)$ is isomorphic to the singular simplicial set associated to the underlying space of $\II$, the morphism $\dSing(\sigma)\colon \dSing(\II)\to \Delta^0$ is a trivial Kan fibration, and in particular, $\sigma$ is a trivial fibration of locally stratified spaces.

\begin{example}
The locally stratified realisation of $J$ is a quotient of $\dabs{\Delta^2}\coprod \dabs{\Delta^2}$ where we glue together the edges from $0$ to $1$ of both simplices and we contract the edges from $0$ to $2$. 
Recall that, the underlying stratified space of $\dabs{J}$ is the chaotic 2-disk $\ind{\DD^{2}}$ (see Lemma \ref{lem:strat-J-trivial-disk}).
However, differently from the stratified case, the local stratification of $\dabs{J}$ is nontrivial. 
For example, locally around the vertex $1$, every local exit path leaving $1$ counterclockwise cannot come back to $1$.
\end{example}

\begin{lemma}
The morphism $(\partial^0,\partial^1)\colon \terminal\coprod\terminal\to \II$ is a retract of $\terminal\coprod\terminal \to \dabs{J}$. In particular, $(\partial^0,\partial^1)$ is a cofibration.
\end{lemma}
\begin{proof}
We show that the circulation of $\dabs{J}$ is trivial locally around $0$, so that the image of the canonical map $\dabs{\Delta^1}\to \dabs{J}$ is indeed trivially locally stratified, hence it extends to a morphism $i\colon \II\to \dabs{J}$.
To do this, notice that for every open subset $U$ of $\dabs{J}$ containing $0$ and for every point $x\in U$, we can take a directed path from $x$ to $2$ in a copy of $\dabs{\Delta^2}$ containing the preimage of $x$, via the quotient map $\dabs{\Delta^2}\coprod\dabs{\Delta^2}\to \dabs{J}$, which then defines a directed path from $x$ to $0$ in $U$. Therefore, $i$ fits in a commutative diagram:
\begin{diagram}
\terminal\coprod\terminal\ar[r]\ar[d] &
\terminal\coprod\terminal\ar[r]\ar[d] &
\terminal \coprod\terminal\ar[d] \\
\II\ar[r, "i"] &
\dabs{J} \ar[r, "r"] &
\II
\end{diagram}
where $r\colon \dabs{J}\to \II$ is the unique map in $\Strat$ corresponding to the retraction $r\colon \abs{J}\to \II$ and the horizontal compositions are identities.
\end{proof}

\subsection{}
For every locally stratified space $X$, applying the functor $X\times(\blank)$ to \eqref{eq:chaotic-interval-locstrat} yields a decomposition of the codiagonal on $X$
\begin{equation}
\label{eq:cylinder-locstrat}
\begin{tikzcd}
X\coprod X \ar[r, "{\left(\partial_{X}^0,\partial_{X}^1\right)}"] &
X\times \II \ar[r, "\sigma_{X}"] &
X
\end{tikzcd}
\end{equation}
where $\sigma_{X}$ is a weak equivalence by Lemma \ref{lem:cylinder-top-cat}.
Moreover, if $X$ is a cofibrant stratified space, the morphism $\left(\partial_{X}^{0}, \partial_{X}^{1}\right)$ is a cofibration, again by Lemma \ref{lem:cylinder-top-cat}.

\begin{lemma}
\label{lem:push-chaotic-interval-locstrat}
The chaotic interval $\II$ fits in a pushout diagram:
\begin{equation}
\label{eq:push-chaotic-interval-locstrat}
\push{\term}{\II}{\II}{\II}{\partial^{1}}{}{}{\partial^{0}}
\end{equation}
of locally stratified spaces.
\end{lemma}
\begin{proof}
Notice that the underlying topological space of the pushout \eqref{eq:push-chaotic-interval-locstrat} is the space $[0,2]\cong \II$ and  the chaotic structure on $\II$ is finer than any other stream structure.
Therefore, it is enough to prove that for every open subset $U$ in $[0,2]$, if $t\sim_{U}t'$ then $t\preceq_{U}t'$, where $\sim$ denotes the chaotic circulation on $\ind{\II}$ and we use $\preceq$ for the stream structure of the pushout.
Without loss of generality, we can assume that $t$ is less than $1$ and $t'$ is greater than $1$, for if $t$ and $t'$ both lie on the same half interval, the claim follows from the fact that $\partial^{\epsilon}\colon \II\to \II \coprod_{\term}\II$ is a stream morphism.
Therefore, under the above assumptions, we have that $[t,t']$ is contained in $U$. 
In particular, taking the intersections $[t,t']\cap [0,1]$ and $[t,t']\cap[1,2]$ and by \ref{subsec:final-initial-preorder} the claim follows.
\end{proof}

\subsection{}
\label{subsec:local-stratum-preserving-homotopy}
Let $f\colon X\to Y$ and $g\colon X\to Y$ be morphisms of stratified spaces.
A \emph{chaotic homotopy} from $f$ to $g$ is an $\II$-homotopy in the sense of definition \ref{def:E-homotopy}.
In the presence of a chaotic homotopy from $f$ to $g$ we say that $f$ is \emph{chaotically homotopic} to $g$ and write $f\sim g$.
By Lemma \ref{lem:push-chaotic-interval-locstrat}, $\sim$ defines an equivalence relation on the set of morphisms from $X$ to $Y$, which is compatible with composition.
In particular, when $X$ is the terminal locally stratified space, two points $x$ and $x'$ of $X$ are chaotically homotopic if and only if there exists a chaotic path $\alpha\colon \II\to X$ from $x$ to $x'$.

\begin{definition}
For every locally stratified space $X$ we define the set of \emph{chaotic path components} of $X$, denoted by $\cpi_{0}X$ as the quotient of the set of points of $X$ by the chaotic homotopy relation.
\end{definition}

\begin{example}
If $X$ is the locally stratified space associated to a stratified space, two points in $X$ are chaotically homotopic if and only if they belong to the same stratum.
Hence, the chaotic path components of $X$ correspond to the strata of its underlying space, and we will blur the difference between them.
In particular, the set of chaotic path components $\cpi_{0}\dabs{\Delta^{n}}$ of the locally stratified standard $n$-simplex is the set with $n+1$ elements.
Notice that $\cpi_{0}\dabs{\Delta^{n}}$ is different from $\pi_{0}\Sing \dabs{\Delta^{n}}=\term$. 
\end{example}

\begin{definition}
Let $f\colon X\to Y$ be a morphism between locally stratified spaces. 
We say that $f$ is a \emph{chaotic homotopy equivalence} if there exists a morphism $g\colon Y\to X$ such that $gf$ is chaotically homotopic to the identity on $X$ and $fg$ is chaotically homotopic to the identity on $Y$.
\end{definition}

\begin{proposition}
\label{prop:chaotic-equivalence-weak-equivalence}
Every chaotic homotopy equivalence is a weak equivalence.
\end{proposition}
\begin{proof}
The claim follows from Proposition \ref{prop:E-homotopy-we}.
\end{proof}
\section{Left covers and fundamental categories}
\label{ch:5sec:4}
In this section, we define the fundamental category $\dPi_{1}X$ of a locally stratified space $X$.
We show that there exists a canonical essentially surjective functor $\theta_{A}\colon \tau_{1}A\to \dPi_{1}\dabs{A}$ from the fundamental category of a simplicial set to the fundamental category of its realisation (Corollary \ref{cor:tau-dpi-essentially-surjective}).
We illustrate a few examples of simplicial sets $A$ with $\theta_{A}$ an equivalence of categories (Example \ref{ex:fundamental-cat-standard-locstrat} and Proposition \ref{prop:fund-cat-locstrat-circle}) and some counterexamples (Remark \ref{rmk:fundamental-category-R-unequivalent} and Corollary \ref{cor:fund-cat-retract-unequivalent}).
When $Q$ is an $\infty$-category, the functor $\theta_{Q}$ maps split monomorphisms to isomorphisms (Corollary \ref{cor:infty-category-split-iso}) and we conjecture that $\dPi_{1}\dabs{Q}$ is the localisation of $\tau_{1}Q$ at the class of split monomorphisms.
We define the category of left covers of a locally stratified space (Definition \ref{def:left-cover-locstrat}) and its universal left covers.
We conclude the section showing through a series of Examples, that the realisation $\dabs{E}$ of a left cover $E$ over a simplicial set $X$ is not necessarily a left cover.

\begin{definition}
Let $X$ be a locally stratified space. We define the \emph{fundamental category} of $X$ and denote it as $\dPi_{1} X$ as the fundamental category
\[\dPi_{1} X=\tau_{1}\dSing X\]
of the locally stratified singular simplicial set associated to $X$.
\end{definition}

\begin{example}
\label{ex:fundamental-cat-standard-locstrat}
Let $\dabs{\Delta^{n}}$ be the locally stratified standard $n$-simplex. Its fundamental category $\dPi_{1}\dabs{\Delta^{n}}$ is naturally isomorphic to the exit path category $\Exit\sabs{\Delta^{n}}$ of its underlying stratified space.
Indeed, we have that $\dSing\dabs{\Delta^{n}}$ is isomorphic to $\sSing\sabs{\Delta^{n}}$ by Proposition \ref{prop:locstrat-standard-simplex-strat} and the result follows from Theorem \ref{thm:boardmann-vogt}.
In particular, the unit map $\Delta^{n}\to \dSing\dabs{\Delta^{n}}$ induces a natural isomorphism:
\[[n]=\tau_{1}\Delta^{n}\to \tau_{1}\dSing\dabs{\Delta^{n}}=\dPi_{1} \dabs{\Delta^{n}}\cong \Exit\sabs{\Delta^{n}}=[n].\]
\end{example}

\begin{proposition}
\label{prop:fund-cat-locstrat-circle}
The fundamental category $\dPi_{1}\dabs{S^{1}}$ of the locally stratified circle $\dabs{S^{1}}$ (see Example \ref{ex:locally-strat-circle}) is isomorphic to the monoid $\NN$ of the natural numbers.
Moreover, the unit map 
\[\eta_{S^{1}}\colon S^{1}\to \dSing\dabs{S^{1}}\]
induces an isomorphism of fundamental categories.
\end{proposition}
\begin{proof}
Notice that an elementary local exit path $\gamma \colon \dabs{\Delta^{1}}\to \dabs{S^{1}}$ in $\dabs{S^{1}}$ can wrap around the circle at most once. 
For if $\gamma\colon\dabs{\Delta^{1}}\to \dabs{S^{1}}$ is an elementary local exit path crossing the basepoint counterclockwise, taking a small enough open arc containing $0$ in $\dabs{S^{1}}$, we see that the preimage of $U$ has at least one connected component which is chaotically preordered yielding a contradiction, since $U$ is not chaotically preordered.
As a consequence, every morphism $\alpha\colon\dabs{\Delta^{n}}\to \dabs{S^{1}}$ from an $n$-dimensional locally stratified standard simplex to $\dabs{S^{1}}$ can also wrap around the circle at most once, as the restriction of $\alpha$ to the edge $\dabs{\Delta^{{0,n}}}$ from the first to the last vertex of $\Delta^{n}$, defines an elementary local exit path in $\dabs{S^{1}}$.
Moreover, every non trivial elementary local exit path in $\dabs{S^{1}}$ starting and ending at $0$ is easily seen to be chaotically homotopic to the canonical projection $q\colon \dabs{\Delta^{1}}\to \dabs{S^{1}}$.
Therefore, the fundamental category $\dPi_{1}\dabs{S^{1}}$ of the locally stratified circle is naturally equivalent to the free category generated by the quotient map $q\colon \dabs{\Delta^{1}}\to \dabs{S^{1}}$ and the morphism
\[\NN=\tau_{1}S^{1}\to \dPi_{1}\dabs{S^{1}}\]
sends the generator $1$ of $\NN$ to the generator $q$ of $\dPi_{1}\dabs{S^{1}}$.
\end{proof}

\begin{lemma}
\label{lem:chaotic-path-iso-dpi}
Let $X$ be a locally stratified space, $x$ and $x'$ be points in $X$ and assume that there exists a chaotic path $\alpha$ in $X$ from $x$ to $x'$.
Then, $\alpha$ defines an isomorphism in $\dPi_{1}X$ from $x$ to $x'$.
\end{lemma}
\begin{proof}
The path $\alpha$ and its reverse path $\alpha^{{-1}}$ define elementary local exit paths in $X$. 
Moreover, points in the image of $\alpha$ are chaotically equivalent and we can define a  morphism $\dabs{\Delta^{2}}\to X$ whose boundary is given by $(\alpha,1_{x},\alpha^{{-1}})$ and dually a morphism $\dabs{\Delta^{2}}\to X$ with boundary $(\alpha^{{-1}}, 1_{x'},\alpha)$.
\end{proof}

\subsection{}
\label{subsec:last-vertex-geometric}
Let $A$ be a simplicial set and let $\abs{A}$ be its topological realisation.
Since $\abs{A}$ is a CW-complex (see Remark \ref{rmk:realisation-cw}), for every point $x\in \abs{A}$ there exists a unique non-degenerate $n$-simplex $\sigma_{x}\colon \Delta^{n}\to A$ such that $x$ is the image via $\abs{\sigma_{x}}$ of a unique point in the interior of $\abs{\Delta^{n}}$.
In particular, we can define the \emph{geometric last vertex map}:
\begin{align*}
\lv\colon \abs{A}&\to A_{0}\\
x &\mapsto \sigma_{x}(n)
\end{align*}
mapping every point $x$ in $\abs{A}$ to the image of the last vertex of $\Delta^{n}$ via $\sigma_{x}$.
Moreover, we use the notation $\lv(x)$  also for the image of $\lv(x)$ in $\abs{A}$.

\begin{lemma}
\label{lem:last-vertex-chaotic-path}
Let $A$ be a simplicial set and let $x$ be a point of $\dabs{A}$. Then, there exists a chaotic path $\II\to \dabs{A}$ from $x$ to $\lv(x)$
\end{lemma}
\begin{proof}
By \ref{subsec:last-vertex-geometric} there exists a unique simplex $\sigma_{x}\colon \Delta^{n}\to A$ of $A$ such that $x$ is the image of a unique point $p$ in the interior of $\dabs{\Delta^{n}}$.
In particular, since the complement of the $n$-th face of $\dabs{\Delta^{n}}$ is trivially locally stratified, taking the straight line from $p$ to $n$ defines a chaotic path $\gamma\colon\II\to \dabs{\Delta^{n}}$ in $\dabs{\Delta^{n}}$, which gives a chaotic path $\dabs{\sigma_{x}}\gamma\colon \II\to \dabs{A}$ in $\dabs{A}$ from $x$ to $\lv(x)$.
\end{proof}

\begin{corollary}
\label{cor:tau-dpi-essentially-surjective}
Let $A$ be a simplicial set. Then, the canonical functor
\[\theta_{A}\colon \tau_{1}A\to \dPi_{1}\dabs{A}\]
from the fundamental category of $A$ to the fundamental category of its realisation is an essentially surjective functor.
\end{corollary}
\begin{proof}
The claim follows from Lemma \ref{lem:last-vertex-chaotic-path} and Lemma \ref{lem:chaotic-path-iso-dpi}.
\end{proof}

\begin{example}
\label{ex:walking-retraction-simplicial-set}
Let $\Delta^{1}$ be the standard 1-simplex.
Recall that the simplicial set $R$ is defined by freely adding a right inverse $\beta$ to the unique non-degenerate 1-simplex $\alpha$ from $0$ to $1$ in $\Delta^{1}$ (see Example \ref{ex:walking-retraction-sset}).
The locally stratified realisation of $R$ is the locally stratified space whose underlying topological space is the 2-disk, with local stratification given as in Figure \ref{figure2}.

\begin{figure}[ht]
\centering
\begin{tikzpicture}

\draw [thick, fill=gray!20] (0,0) circle [radius=1];

\draw [thick, fill] (-1,0) circle [radius=0.05];
\node at (-0.7,0) {$0$};

\draw [thick, fill] (1,0) circle [radius=0.05];
\node at (0.7,0) {$1$};

\draw [thick, ->] (0.1,1) -- (-0.1,1);
\node at (0,0.6) {$\beta$};

\draw [thick, ->] (-0.1,-1) -- (0.1,-1);
\node at (0,-0.6) {$\alpha$};

\draw [gray,  dashed] (-2,0) circle [radius=0.5];
\draw[gray, dashed] (-1.5,0) -- (-1,0);

\fill[red!30] (-2,-0.5) arc [radius=0.5, start angle = -90, end angle=90];
\draw[red, thick] (-2,0.5) -- (-2,-0.5);
\draw [red, thick, fill] (-2,0) circle [radius=0.05];

\draw [gray, dashed] (2,0) circle [radius=0.5];
\draw[gray, dashed] (1,0) -- (1.5,0);

\fill[green!30] (2,0.5) arc [radius=0.5, start angle = 90, end angle=270];
\draw[red, thick] (2,-0.5) -- (2,0);
\draw[green, thick] (2,0) -- (2,0.5);
\draw [red, thick, fill] (2,0) circle [radius=0.05];

\draw [gray,  dashed] (0,2) circle [radius=0.5];
\draw[gray, dashed] (0,1) -- (0,1.5);

\fill[red!30] (-0.5,2) arc [radius=0.5, start angle = 180, end angle=360];
\draw[red, thick] (-0.5,2) -- (0.5,2);

\draw [gray,  dashed] (0,-2) circle [radius=0.5];
\draw[gray, dashed] (0,-1) -- (0,-1.5);

\fill[green!30] (-0.5,-2) arc [radius=0.5, start angle = 180, end angle=0];
\draw[red, thick] (-0.5,-2) -- (0.5,-2);
\end{tikzpicture}
\caption{The locally stratified realisation of $R$}\label{figure2}
\end{figure}
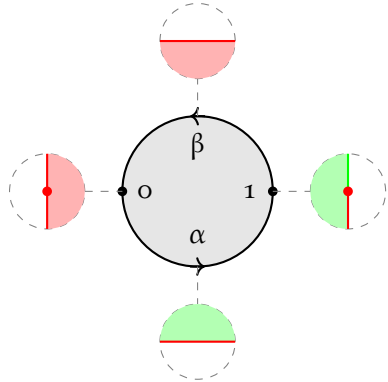
In particular, notice that the underlying path of the elementary local exit path $\dabs{\alpha}$ defines a chaotic path $\gamma\colon \II\to \dabs{R}$.
Indeed, it is enough to show that for every connected open neighborhood $U$ of $0$ in $\dabs{R}$ and every point $x$ in $\im \dabs{\alpha} \cap U$, we have that $x\preceq_{U}0$, where $\preceq$ denotes the circulation of $\dabs{R}$.
This follows from the fact that $x \le_{q^{{-1}}U} 2$, where $q\colon \dabs{\Delta^{2}}\to \dabs{R}$ denotes the quotient map given by collapsing the edge from $0$ to $2$.
\end{example}

\begin{corollary}
\label{cor:fundamental-category-retract}
The fundamental category $\dPi_{1}\dabs{R}$ of the realisation of $R$ is equivalent to the terminal category.
\end{corollary}
\begin{proof}
Following Example \ref{ex:walking-retraction-simplicial-set} and by Lemma \ref{lem:chaotic-path-iso-dpi} the chaotic path $\gamma\colon \II\to \dabs{R}$ defines an isomorphism from $0$ to $1$ in $\dPi_{1}\dabs{R}$.
In particular, since $0$ and $1$ are the only vertices of $R$, by Corollary \ref{cor:tau-dpi-essentially-surjective} we see that all the objects of $\dPi_{1}\dabs{R}$ are isomorphic to each other.
To show that every morphism in $\dPi_{1}\dabs{R}$ is invertible it is enough to show that the edge $\dabs{\beta}\colon 1\to 0$ is invertible in $\dPi_{1}\dabs{R}$, since any other elementary local exit path starting from the lower edge of $\dabs{R}$ is easily seen to be chaotically homotopic to $\dabs{\beta}$.
To conclude since $\dabs{\beta}$ is a right inverse to $\dabs{\alpha}$ in $\dPi_{1}\dabs{R}$ and since the class of $\dabs{\alpha}$ is an isomorphism, the class of $\dabs{\beta}$ is an isomorphism as well.
\end{proof}

\begin{remark}
\label{rmk:fundamental-category-R-unequivalent}
Corollary \ref{cor:fundamental-category-retract} implies that the canonical morphism:
\[\theta_{R}\colon \tau_{1}R\to \dPi_{1}\dabs{R}\]
is not an equivalence of categories. 
Indeed, the fundamental category $\tau_{1}R$ of $R$ is the walking retraction category $\Ret$ (see Example \ref{ex:walking-retraction-cat}).
\end{remark}

\begin{proposition}
\label{prop:sset-split-iso}
Let $A$ be a simplicial set and consider the functor: 
\[\theta_{A}\colon\tau_{1}A\to \dPi_{1}\dabs{A}\]
given by applying $\tau_{1}$ to the map:
\[\eta_{A}\colon A\to \dSing\dabs{A}.\]
Assume that we have $1$-simplices $r\colon a\to a'$ and $i\colon a'\to a$ in $A$ such that $ri$ is homotopic to $1_{a}$ via a 2-simplex $\Delta^{2}\to A$ with boundary $(r,1_{a},i)$.
Then $\theta_{A}r$ defines an isomorphism in $\dPi_{1}\dabs{A}$ with inverse $\theta_{A}i$.
\end{proposition}
\begin{proof}
By hypothesis, we can find a morphism $f\colon R\to A$ mapping the unique $1$-simplex from $0$ to $1$ to $i$ and its right inverse to $r$.
Then, by Example \ref{ex:walking-retraction-simplicial-set} the realisation of $i$ defines a chaotic path $\gamma \colon \II\to \dabs{A}$ from $a$ to $a'$.
In particular, $\gamma$ defines an isomorphism in $\dPi_{1}\dabs{A}$ with inverse $\gamma^{-1}$ which implies the claim.
\end{proof}

\begin{corollary}
\label{cor:infty-category-split-iso}
Let $Q$ be an $\infty$-category and let us consider the functor:
\[\theta_{Q}\colon\tau_{1}Q\to \dPi_{1}\dabs{Q}.\]
Then, $\theta_{Q}$ maps split epimorphisms and split monomorphisms to isomorphisms in $\dPi_{1}\dabs{Q}$.
\end{corollary}
\begin{proof}
This follows from the fact that $\tau_{1}Q\simeq \ho Q$ for an $\infty$-category $Q$, together with Proposition \ref{prop:sset-split-iso}.
\end{proof}

\begin{conjecture}
\label{conj:fundamental-cat-localisation-split}
Let $Q$ be an $\infty$-category. Then the functor:
\[\theta_{Q}\colon \tau_{1}Q\to \dPi_{1}\dabs{Q}\]
is the localisation of $\tau_{1}Q$ at the class of split monomorphisms.
Equivalently, $\theta_{Q}$ is the localisation of $\tau_{1}Q$ at the class of split epimorphisms.
\end{conjecture}

We highlight the following interesting consequence of Conjecture \ref{conj:fundamental-cat-localisation-split}.

\begin{conjecture}
Let $M$ be an abelian monoid, considered as a category with one object. Then, the natural map
\[\Nerv M\to \dSing\dabs{\Nerv M}\]
induces a natural isomorphism:
\[M\to \dPi\dabs{\Nerv M}.\]
\end{conjecture}

\begin{example}
\label{ex:walking-retraction-category}
Let $\Ret$ be the walking retraction category, defined in Example \ref{ex:walking-retraction-cat}.
The nerve $\Nerv \Ret$ of the walking retraction category is an infinite dimensional simplicial set whose $2$-skeleton is the  simplicial set with two $0$-simplices $0$ and $1$, three non degenerate $1$-simplices $i\colon 0\to 1$, $r\colon 1\to 0$ and $ir\colon 1\to 1$ and two non degenerate $2$-simplices $\alpha$ and $\beta$ with boundaries $(r,1_{a}, i)$ and $(i,ir,r)$ respectively.
The locally stratified realisation of the $2$-skeleton $\sk_{2}\Nerv \Ret$ can depicted as follows:

\begin{figure}[h]
\centering
\begin{tikzpicture}

\fill [gray!10] (1,0) ellipse (2.25 and 1.5);
\draw [thick, fill=gray!20] (0,0) circle [radius=1];
\draw [thick, fill=white] (2,0) circle [radius=1];

\draw [thick, fill] (-1,0) circle [radius=0.05];
\node at (-0.7,0) {$a$};

\draw [thick, fill] (1,0) circle [radius=0.05];
\node at (0.7,0) {$b$};

\draw [thick, ->] (0.1,1) -- (-0.1,1);
\node at (0,0.6) {$f$};

\draw [thick, ->] (-0.1,-1) -- (0.1,-1);
\node at (0,-0.6) {$g$};

\draw [thick, ->] (3,0.1) -- (3,-0.1);
\node at (2.7,0) {$gf$};

\draw [gray,  dashed, fill=red!30] (-2,0) circle [radius=0.5];
\draw[gray, dashed] (-1.5,0) -- (-1,0);

\draw[red, thick] (-2,0.5) -- (-2,-0.5);
\draw [red, thick, fill] (-2,0) circle [radius=0.05];

\draw [gray, dashed] (2,2) circle [radius=0.5];
\draw[gray, dashed] (1,0) -- (2,0) -- (2,1.5);

\fill[green!30] (2,2) -- ([shift=(45:0.5)] 2,2) arc [radius=0.5, start angle = 45, end angle=135] -- (2,2);
\draw [green, thick] (2,2) -- ([shift=(45:0.5)] 2,2);

\fill[yellow!30] (2,2) -- ([shift=(135:0.5)] 2,2) arc [radius=0.5, start angle = 135, end angle=225] -- (2,2);
\draw [yellow, thick] (2,2) -- ([shift=(135:0.5)] 2,2);

\fill[red!30] (2,2) -- ([shift=(225:0.5)] 2,2) arc [radius=0.5, start angle = 225, end angle=315] -- (2,2);\draw [red, thick] (2,2) -- ([shift=(315:0.5)] 2,2);
\draw [red, thick] (2,2) -- ([shift=(225:0.5)] 2,2);
\draw [red, thick, fill] (2,2) circle [radius=0.05];

\draw [gray,  dashed] (0,2) circle [radius=0.5];
\draw[gray, dashed] (0,1) -- (0,1.5);

\fill[red!30] (-0.5,2) arc [radius=0.5, start angle = 180, end angle=360];
\draw[red, thick] (-0.5,2) -- (0.5,2);

\fill[green!30] (-0.5,2) arc [radius=0.5, start angle = 180, end angle=0];

\draw [gray,  dashed] (0,-2) circle [radius=0.5];
\draw[gray, dashed] (0,-1) -- (0,-1.5);

\fill[red!30] (-0.5,-2) arc [radius=0.5, start angle = 180, end angle=360];
\draw[red, thick] (-0.5,-2) -- (0.5,-2);

\fill[green!30] (-0.5,-2) arc [radius=0.5, start angle = 180, end angle=0];

\draw [gray,  dashed] (2,-2) circle [radius=0.5];
\draw[gray, dashed] (2,-1) -- (2,-1.5);

\fill[red!30] (1.5,-2) arc [radius=0.5, start angle = 180, end angle=360];
\draw[red, thick] (1.5,-2) -- (2.5,-2);

\end{tikzpicture}
\caption{The realisation of the $2$-skeleton of ${\Nerv \Ret}$}\label{figure3}
\end{figure}
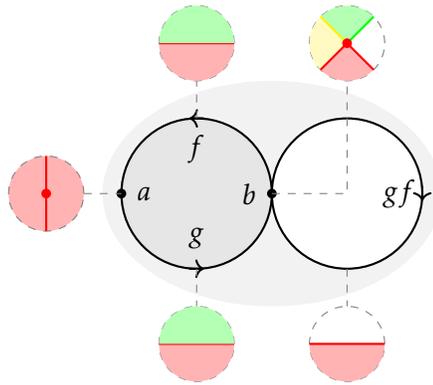

Moreover, the unit of the adjunction $(\tau_{1},\Nerv)$ at $R$ defines a canonical map $R\to \Nerv\Ret$ which factors through the $2$-skeleton of $\Nerv\Ret$, since $R$ is a $2$-dimensional simplicial set.
Therefore, the edge $i$ from $0$ to $1$ defines a chaotic path in $\dabs{\Nerv\Ret}$ and in particular in $\dabs{\sk_{2}\Nerv \Ret}$ by Example \ref{ex:walking-retraction-simplicial-set}.
\end{example}

\begin{corollary}
\label{cor:fund-cat-retract-unequivalent}
The fundamental category $\dPi_{1}\dabs{\Nerv \Ret}$ has a unique isomorphism class of objects. 
In particular, the canonical morphism:
\[\theta_{\Ret}\colon \Ret\cong \tau_{1}\Nerv\Ret\to \dPi_{1}\dabs{\Nerv\Ret}\]
is not an equivalence of categories.
\end{corollary}
\begin{proof}
The claim follows from the fact that the images of the vertices $0$ and $1$ in $\dabs{\Nerv\Ret}$ are isomorphic via a chaotic path, together with Corollary \ref{cor:tau-dpi-essentially-surjective}.
\end{proof}

\subsection{}
Recall that (see Definition \ref{def:left-cover}), if $X$ is a simplicial set, a left cover over $X$ is a morphism $p\colon E\to X$ which has the unique right lifting property with respect to the set:
\[\tensor[^{l}]{\Delta}{_{0}}=\{0\colon\Delta^{0}\to\Delta^{n}: [n]\in \DDelta\}\]
of initial vertices of all standard simplices.

\begin{definition}
\label{def:left-cover-locstrat}
Let $X$ be an object in $\LocStrat$. 
A morphism $p\colon E\to X$ is said to be a \emph{left cover} over $X$ if $p$ has the unique right lifting property with respect to the set:
\[\dabs{\tensor[^{l}]{\Delta}{_{0}}}=\left\{0\colon\term\to\dabs{\Delta^{n}}: [n]\in \DDelta\right\}.\]
In other words, $p$ is a left cover in $\LocStrat$ if and only if $\dSing p\colon \dSing E\to \dSing X$ is a left cover in $\sSet$.
We denote by  $\LCover{X}$ the full subcategory of the slice $\overcat{\LocStrat}{X}$ spanned by the left covers over $X$.
\end{definition}

\begin{remark}
\label{rmk:left-cover-locstrat}
By Remark \ref{rmk:left-cover}, a morphism $p\colon E\to X$ is a left cover if and only if the initial map $[0]\to[n]$ induces a natural isomorphism:
\[\dSing(E)_{n}\cong \dSing(E)_{0}\times_{\dSing(X)_{0}}\dSing(X)_{n}.\]
In particular a left cover $p\colon E\to \term$ over the terminal locally stratified space is a discrete locally stratified space.
Hence there is a natural isomorphism $\LCover{\term}\cong\Set$.
\end{remark}

\subsection{}
Since right orthogonal maps are stable under pullback, a morphism $f\colon X'\to X$ of locally stratified spaces induces a \emph{base change} functor:
\[f^{*}\colon \LCover{X}\to \LCover{X'}\]
in particular, by Remark \ref{rmk:left-cover-locstrat} a point $x\colon\term\to X$ of $X$ induces a \emph{fibre functor}:
\begin{align*}
\fib_{x}\colon \LCover{X} & \to \Set\\
p\colon E\to X &\mapsto \fib_{x}E 
\end{align*}

\subsection{}
Since $\LocStrat$ is a locally presentable category, by Corollary \ref{cor:orthogonal-factorisation-theorem}, every morphism $f\colon Y\to X$ can be factored as a composition:
\begin{equation}
\label{eq:factorisation-left-cover-locstrat}
\begin{tikzcd}[column sep=small, row sep=small]
Y \arrow[rr, "f"] \arrow[rdd, "j"']
   & & X \\
   & & \\
& E \ar[ruu, "p"'] &
\end{tikzcd}
\end{equation}
where $p\colon E\to X$ is a left cover and $j\colon Y\to X$ has the unique left lifting property with respect to any left cover.
In particular, this defines a left adjoint to the inclusion functor $\LCover{X}\to \overcat{\LocStrat}{X}$
\begin{equation}
\label{eq:adjoint-left-cover-locstrat}
\Adjoint{\LCover{X}}{\overcat{\LocStrat}{X}}{L}{}
\end{equation}
which displays $\LCover{X}$ as a reflective subcategory of $\overcat{\LocStrat}{X}$.

\subsection{}
Let $x$ be a point of a locally stratified space $X$, applying the factorisation \eqref{eq:factorisation-left-cover-locstrat} to the morphism $x\colon \term\to X$ yields a diagram:
\begin{diagram}
\term
	\ar[r, "\tilde{x}"']
	\ar[rr, bend left, "x"]&
\tilde{X}_{x}
	\ar[r,"\tilde{p}"']&
X
\end{diagram}
The morphism $\tilde{p}\colon \tilde{X}_{x}\to X$ is called the \emph{universal left cover} of $X$ at $x$.
Moreover, we have the following result:

\begin{lemma}
\label{lem:universal-left-cover-represents-fibers-locstrat}
Let $X$ be a locally stratified space and let $\tilde{X}_{x}$ be the universal left cover of $X$ at $x$.
Then, for every left cover $E$ of $X$, evaluation at $\tilde{x}$ induces an isomorphism:
\begin{equation}
\label{eq:universal-left-cover-represents-fibers}
\LCover{X}(\tilde{X}_{x}, E)=\fib_{x}E,
\end{equation}
natural in $E$.
\end{lemma}
\begin{proof}
This follows immediately from the definitions.
\end{proof}

\subsection{}
Given a locally stratified space $X$, and a simplicial set $A$ we have adjunctions:
\begin{equation}
\label{eq:sset-locstrat-adjunction-slice}
\Adjoint{\overcat{\sSet}{A}}{\overcat{\LocStrat}{\dabs{A}}}{\dabs{\blank}}{\dSing_{A}} \quad
\Adjoint{\overcat{\sSet}{\dSing X}}{\overcat{\LocStrat}{X}}{\dabs{\blank}_{X}}{\dSing}
\end{equation}
where the functors $\dabs{\blank}$ and $\dSing$ are the induced functors on the slice categories
while the functor $\dSing_{A}$ maps a morphism $f\colon B\to \dabs{A}$ to the base change
\begin{equation}
\pull{\dSing_{A}Y}{\dSing Y}{\dSing\dabs{A}}{X}{}{\dSing f}{\eta_{A}}{}
\end{equation}
and $\dabs{\blank}_{X}$ takes a map $g\colon Y\to \dSing X$ to the composite
\begin{equation}
\begin{tikzcd}
\dabs{Y}
\ar[r, "\dabs{g}"]&
\dabs{\dSing X}
\ar[r, "\epsilon_{X}"]
& X
\end{tikzcd}
\end{equation}
Hence, the adjunction \ref{eq:adjoint-left-cover-locstrat} implies that we have adjunctions:
\begin{equation}
\label{eq:sset-locstrat-adjunction-covers}
\Adjoint{\LCover{A}}{\LCover{\dabs{A}}}{L\dabs{\blank}}{\dSing_{A}} \quad
\Adjoint{\LCover{\dSing X}}{\LCover{X}}{L\dabs{\blank}_{X}}{\dSing}
\end{equation}
In particular, when $X=\dabs{A}$, the second adjunction of \ref{eq:sset-locstrat-adjunction-covers} factors through the first as:
\begin{equation}
\begin{tikzcd}[ampersand replacement=\&]
{\LCover{\dSing\dabs{A}}} \arrow[r, "{}" {name=F}, shift left]
\& {\LCover{A}} \arrow[l, "{}" {name=U}, shift left]
\arrow[phantom, from=F, to=U, symbol=\dashv]
 \arrow[r, "{L\dabs{\blank}}" {name=F}, shift left]
\& {\LCover{\dabs{A}}} \arrow[l, "{\dSing_{A}}" {name=U}, shift left]
\arrow[phantom, from=F, to=U, symbol=\dashv]
\end{tikzcd}
\end{equation}

The following Examples show that, differently from the case of topological covers (see Proposition \ref{prop:sing-realisation-cover}), the functor $\dabs{\blank}\colon \LCover{A}\to \slice{\LocStrat}{\dabs{A}}$ alone does not preserve left covers.

\begin{example}
\label{ex:universal-cover-1-simplex}
The universal left cover of $\Delta^{1}$ at the vertex $1$ is the simplicial set $\Delta^{0}$ together with the face map $\partial^{1}_{0}\colon \Delta^{0}\to\Delta^{1}$ (see Example \ref{ex:left-cover-standard-simplex}).
Taking realisations, we see that the morphism $1\colon \term\to\dabs{\Delta^{1}}$ is not a left cover, since, for example, the path:
\begin{align*}
\alpha\colon[0,1]&\to [0,1]\\
t&\mapsto 1-\frac{1}{2}t
\end{align*}
has no lift.
On the other hand, taking the trivially locally stratified half-open interval $i\colon(0,1]\to \dabs{\Delta^{1}}$ yields a left cover of $\dabs{\Delta^{1}}$.
Moreover, $i$ is the universal left cover of $\dabs{\Delta^{1}}$ at $1$.
Indeed, as we will show in Corollary \ref{cor:chaotic-left-property-cover}, for every left cover $p\colon E\to\dabs{\Delta^{1}}$ and for every $t>0$ we can find a lift in the diagram:
\begin{equation}
\lift{\term}{E}{\dabs{\Delta^{1}}}{[t,1]}{e}{p}{}{1}{\alpha_{t}}
\end{equation}
Moreover, the maps $\alpha_{t}$ are compatible with the inclusions $[t,1]\subset[t',1]$ for $0<t\le t'$, by uniqueness of lifts.
Therefore, they induce a unique morphism $\alpha\colon (0,1]\to E$ such that $\alpha(1)=e$.
Applying Lemma \ref{lem:universal-left-cover-represents-fibers-locstrat} we are finished.
\end{example}

\begin{example}
\label{ex:universal-cover-n-simplex-locstrat}
More generally, let $x$ be a point of $\dabs{\Delta^{n}}$ and assume that $x$ belongs to the global $k$-stratum.
Then, analogously to Example \ref{ex:universal-cover-1-simplex}, the universal left cover of $\dabs{\Delta^{n}}$ at $x$ is the complement of the union of strata:
\[\bigcup_{j<k} \sabs{\Delta^{n}}_{j}\]
equipped with the subspace locally stratified structure.
\end{example}

\begin{example}
\label{ex:left-cover-1-simplex}
Let us consider the simplicial set $\Lambda^{2}_{2}$.
The natural projection map $p\colon\Lambda^{2}_{2}\to \Delta^{1}$ induced by the degeneracy $\sigma^{0}\colon \Delta^{2}\to \Delta^{1}$ is the left cover associated to the functor $F\colon [1]\to \Set$ that sends $0$ to the set with two elements, and $1$ to the singleton (see Example \ref{ex:fundamental-category-standard-simplex} and Theorem \ref{thm:left-covers-fiber-functors}).
However, the locally stratified realisation of $p$ is not a left cover in $\LocStrat$.
Indeed, since all the points in the complement of $0$ in $\dabs{\Delta^{1}}$ are chaotically equivalent, it is easy to find an elementary local exit path that admits multiple lifts in $\dabs{\Delta^{2}_{2}}$.
For example, one can choose the elementary local exit path whose underlying path is given by:
\begin{align*}
\alpha\colon [0,1]&\to [0,1]\\
t & \mapsto \begin{cases}
	t+\frac{1}{2} & t\le \frac{1}{2}\\
	\frac{3}{2}-t & t\ge \frac{1}{2}.
	\end{cases}
\end{align*}
On the other hand, let us consider the pushout:
\begin{equation}
\push{(0,1]}{\dabs{\Lambda^{2}_{2}}}{E}{\dabs{\Lambda^{2}_{2}}}{\dabs{\partial_{0}}}{}{}{\dabs{\partial_{1}}}
\end{equation}
where $\dabs{\partial_{\epsilon}}$ denotes the restriction of the realisation of the face map $\partial_{\epsilon}\colon \Delta^{1}\to \Lambda^{2}$ to the half-open interval $(0,1]$.
Then, $E$ is a locally stratified space whose underlying topological space is the segment with two distinct origins.
In particular, the induced map $E\to \dabs{\Delta^{1}}$ is a left cover and one can see that $E=L\dabs{\Lambda^{2}_{2}}\to \dabs{\Delta^{1}}$.
\end{example}

\begin{remark}
Let $\dabs{\Delta^{n}}$ be the locally stratified standard $n$-simplex and let $p\colon E\to \dabs{\Delta^{n}}$ be a left cover. 
Then, the underlying stratified map $\Gamma p\colon \Gamma E\to \sabs{\Delta^{n}}$ (see \ref{subsec:global-strat-space}) has the unique right lifting property with respect to every map $0\colon \term\to \sabs{\Delta^{m}}$.
Indeed, let us consider a lifting problem:
\begin{equation}
\lift{\term}{\Gamma E}{\sabs{\Delta^{n}}}{\sabs{\Delta^{m}}}{e}{\Gamma p}{\alpha}{0}{k}.
\end{equation}
Then, since the points of $\Gamma E$ coincide with the points of $E$, by Proposition \ref{prop:locstrat-standard-simplex-strat} this defines a lifting problem for $p$:
\begin{equation}
\lift{\term}{E}{\dabs{\Delta^{n}}}{\dabs{\Delta^{m}}}{e}{p}{\alpha}{0}{h}
\end{equation}
and it is enough to take $k =\Gamma h$.
\end{remark}

\begin{remark}
Let $p\colon E\to A$ be a left cover of simplicial sets and let $\dabs{p}\colon \dabs{E}\to \dabs{A}$ be its locally stratified realisation.
Examples \ref{ex:universal-cover-1-simplex} and \ref{ex:left-cover-1-simplex} show two different phenomena that occur when applying the functor $L$ to $\dabs{p}$.
The first example shows that given a point $e$ in the realisation of $E$, the functor $L$ \emph{``spreads out''} $e$ to the set of points chaotically equivalent to $e$.
This ensures that lifts in $L\dabs{E}$ exist against all first vertex inclusion.
On the other hand, the second example shows that, given two realisations of non-degenerate simplices in $E$ projecting down to the same simplex in $\dabs{A}$, the functor $L$ glues them together away from the bottom face.
This ensures that lifts are unique.
\end{remark}

\begin{example}
\label{ex:left-cover-circle-locstrat}
Let $\dabs{S^{1}}$ be the locally stratified circle. Recall that the universal cover of $S^{1}$ at $0$ is the simplicial set $e\colon \Sp^{\infty}\to S^{1}$ (see Example \ref{ex:infinity-spine}).
Taking realisations we see that $\dabs{\Sp^{\infty}}\to \dabs{S^{1}}$ does not define a left cover of $\dabs{S^{1}}$. 
Indeed, similarly to Example \ref{prop:universal-cover-functor}, any non-constant chaotic clockwise path starting at $0$ in $\dabs{S^{1}}$ has no lift starting at $0$ in $\dabs{\Sp^{\infty}}$.
Let $L$ be the following pushout:
\begin{equation}
\push{\term}{\dabs{\Sp^{\infty}}}{L}{(0,1]}{0}{}{}{1}
\end{equation}
and let us consider the morphism $q\colon (0,1]\to \dabs{S^{1}}$ given by the restriction of the quotient map $\dabs{\Delta^{1}}\to \dabs{S^{1}}$ to $(0,1]$.
Then, $e$ and $q$ induce a unique map $p\colon L\to \dabs{S^{1}}$ and we claim that $p$ is the universal left cover of $\dabs{S^{1}}$ at $0$.
Indeed, $L$ is a left cover of $\dabs{S^{1}}$ since, every counterclocwise elementary local exit path starting at $0$ in $\dabs{S^{1}}$ has a lift in $\dabs{\Sp^{\infty}}$ starting at every point in the fibre of $0$.
On the other hand, every chaotic path starting at $0$ has a lift either in $\dabs{\Sp^{\infty}}$ if we choose a point $k>0$ in the fibre of $0$, or in $(0,1]$ if we choose $0$ as a starting point, by construction.
Moreover, the morphism $0\colon \term\to L$ is a colimit of maps in $\dabs{\tensor[^{l}]{\Delta}{_{0}}}$, by Example \ref{ex:universal-cover-1-simplex} and Example \ref{ex:infinity-spine}.
\end{example}
\section{\'Etale morphisms and left covers}
\label{ch:5sec:5}
A morphism of locally stratified spaces $p\colon E\to X$ is \'etale precisely when it is a local isomorphism.
Due to the local nature of streams, \'etale morphisms are uniquely determined by \'etale maps of the underlying spaces.
In Corollary \ref{cor:characterisation-left-cover} we characterise left covers over the realisation of a simplicial set in terms of \'etale maps.
For a simplicial set $A$, under suitable assumptions, we give an explicit construction that associates to every functor $F\colon \tau_{1}A\to \Set$ a left cover $C(F)$ over the realisation of $\dabs{A}$.
This allows us to show that, under the same assumptions, the category of left covers over $A$ is equivalent to the category of left covers over its realisation (Corollary \ref{cor:simplicial-left-cover-locstrat}).
We conclude the section, showing that every left cover over the realisation of $R$, is trivial (Corollary \ref{cor:cover-R-trivial}).
In particular, as suggested by Remark \ref{rmk:fundamental-category-R-unequivalent}, the category of left covers over $R$ is not equivalent to the category of left covers over its realisation.

\begin{definition}
Let $p\colon E\to X$ be a morphism of locally stratified spaces.
We say that $p$ is an \'etale morphism if every point $e\in E$ has an open neighborhood $U$ such that, if we endow $U$ with the initial structure, the restriction:
\[p_{U}\colon U\to X\]
is an isomorphism onto its image.
\end{definition}

\begin{lemma}
\label{lem:etale-lifts}
Let $X$ be a  locally stratified space and let:
\[p\colon E\to UX\]
be an \'etale continuous map.
Then, there exists a unique  locally stratified structure on $E$ making $p\colon E\to X$ an \'etale morphism of locally stratified spaces.
\end{lemma}
\begin{proof}
Since $p$ is an \'etale map, every point $e\in E$ has an open neighbourhood $U_{e}$ which is mapped homeomorphically onto $pU_{e}$.
Since $\Str\to \Topa$ is a topological functor, there exists a unique stream structure on $U_{e}$ making the restriction $p_{e}\colon U_{e}\to pU_{e}$ an isomorphism of streams.
Moreover, by Lemma \ref{prop:open-numerically-generated}, $U_{e}$ is a numerically generated locally stratified space.
We endow $E$ with the final structure with respect to the family of maps $\{U_{e}\to E: e\in E\}$.
Then, $E$ is a numerically generated locally stratified space, since $U_{e}$ is.
The map $p$ defines a morphism of streams since, given an open subset $V$ of $X$, we can cover $f^{-1}V$ with the set $\{f^{{-1}}V\cap U_{e}: e\in V\}$ and the claim follows from \eqref{eq:circulation} and the fact that $p_{e}$ is an isomorphism of streams for every $e\in E$.
To conclude, $p$ is an \'etale morphism of locally stratified spaces, by construction.
\end{proof}

\begin{proposition}
\label{prop:etale-chaotic-chaotic}
Let $p\colon E\to X$ be an \'etale morphism of locally stratified spaces and assume that $X$ is chaotically locally stratified.
Then, $E$ is chaotically locally stratified.
\end{proposition}
\begin{proof}
If $Up$ is \'etale and we equip $E$ with the chaotic structure, $p$ is an  \'etale morphism of locally stratified spaces, since the chaotic object functor reflects isomorphisms.
Lemma \ref{lem:etale-lifts} implies that this is the only possible structure on $E$ making $p$ \'etale.
\end{proof}

\begin{proposition}
\label{prop:etale-strat-strat}
Let $p\colon E\to X$ be an \'etale morphism of locally stratified spaces and assume that $X=RS$ is the locally stratified space associated to a stratified space $S$.
Then $p\colon E\to X$ is the locally stratified morphism associated to a stratified map $p\colon \Gamma E\to S$.
\end{proposition}
\begin{proof}
By Proposition \ref{prop:delta-stream-associated-to-delta-prespace} it is enough to show that, for every open subset $U$ in $E$, we have that $x\preceq_{U}y$ if and only if there exists an exit path $\alpha\colon \sabs{\Sp^{n}}\to \Gamma E$ from $x$ to $y$ in the stratified space associated to $E$, whose image lies in $U$. Since \'etale neighbourhoods cover $E$, we can assume that $p$ induces an isomorphism $p_{U}\colon U\to pU$. Then, since $X$ is the locally stratified space associated to $S$, by Proposition \ref{prop:delta-stream-associated-to-delta-prespace} we have that $px\preceq_{pU} py$ if and only if there exists an exit path $\alpha\colon \sabs{\Sp^{n}}\to S$ from $px$ to $py$ whose image is contained in $pU$. Composition with the inverse of $p_{U}$ defines then an exit path in $U$ from $x$ to $y$ and we are done.
\end{proof}

\begin{lemma}
\label{lem:etale-pullback}
Let $p\colon E\to X$ be an \'etale morphism of locally stratified spaces and let $f\colon Y\to X$ be any morphism.
Then, the base change $p_{Y}\colon E_{Y}\to Y$ is an \'etale morphism.
\end{lemma}
\begin{proof}
Since \'etale maps of topological spaces are stable under pullback, it suffices to prove that the unique structure on $E_{Y}$ making $p_{Y}$ an \'etale map coincides with the pullback structure.
Let $(Z,\preceq)$ be a locally stratified space and assume we have morphisms $g\colon Z\to E$ and $h\colon Z\to Y$ of locally stratified spaces such that $pg=fh$.
Let $k\colon Z\to E_{Y}$ be the induce continuous map.
Let $U$ be an open subset of $E_{Y}$ such that $p_{U}\colon U\to p_{Y}U$ is an isomorphism of locally stratified space. 
Then $z\preceq_{k^{-1}U} z'$ in $Z$ implies that $hz\preceq_{p_{Y}U}hz'$ in $Y$, since $h$ is a morphism of locally stratified spaces.
Therefore, as $p_{U}$ is an isomorphism, $kz\preceq_{U}kz'$ in $E_{Y}$ and we are done.
\end{proof}

\begin{proposition}
\label{prop:chaotic-left-cover-top-cover}
Let $X$ be a locally $0$-connected and locally $1$-connected chaotic locally stratified space.
A morphism $p\colon E\to X$ of locally stratified spaces is a left cover precisely when the following properties hold:
\begin{enumerate}
\item
The underlying continuous map $Up\colon UE\to UX$ is a topological cover.
\item
The locally stratified structure of $E$ is chaotic.
\end{enumerate}
\end{proposition}
\begin{proof}
Since $X$ is chaotic, every commutative square:
\begin{equation}
\label{prop:equation-1}
\csquare{\term}{E}{X}{\dabs{\Delta^{n}}}{e}{p}{\sigma}{0}
\end{equation}
is equivalent to its underlying commutative square:
\begin{equation}
\label{prop:equation-2}
\csquare{\term}{UE}{UX}{\abs{\Delta^{n}}}{e}{Up}{\sigma}{0}
\end{equation}
Assuming that $p\colon E\to X$ is a left cover, Corollary \ref{cor:cover-sset-top-equivalence} implies that $Up\colon UE\to UX$ is a topological cover.
In particular $p$ is \'etale, hence $E$ is chaotic by Proposition \ref{prop:etale-chaotic-chaotic}.
Conversely, if $Up$ is a topological cover and $E$ is chaotic, the adjunction $(U,\ind{(\blank)})$ implies that the lifting problem determined by \ref{prop:equation-1} is equivalent to the lifting problem determined by \ref{prop:equation-2}. Hence, Corollary \ref{cor:cover-sset-top-equivalence} implies that $p$ is a left cover.
\end{proof}

\begin{corollary}
\label{cor:chaotic-left-property-cover}
The initial vertex $0\colon \term\to \II$ of the chaotic interval has the unique left lifting property with respect to all left covers.
\end{corollary}
\begin{proof}
It is enough to show that every left cover over $\II$ has the unique right lifting property with respect to $0$.
We conclude by Proposition \ref{prop:chaotic-left-cover-top-cover}.
\end{proof}

Recall that, if $X$ is a homotopically stratified set with locally 1-connected and locally 0-connected strata, a stratified cover of $X$ is a stratified space $p\colon E\to X$ over $X$ where $p$ is a local homeomorphism and restricts to a covering space over each stratum (see Definition \ref{def:stratified-cover}).

\begin{lemma}
\label{lem:strat-cover-left-cover}
Let $p\colon E\to \sabs{\Delta^{n}}$ be a stratified cover. Then the locally stratified morphism associated to $p$ is a left cover.
\end{lemma}
\begin{proof}
Let $p\colon E\to \sabs{\Delta^{n}}$ be a stratified cover. Then, we need to show that for every lifting problem:
\begin{equation*}
\lift{\sabs{\Delta^{0}}}{E}{\sabs{\Delta^{n}}}{\sabs{\Delta^{m}}}{e}{p}{\alpha}{0}{\tilde{\alpha}}
\end{equation*}
There exists a unique lift $\tilde{\alpha}\colon \sabs{\Delta^{m}}\to E$ filling the diagram.
We do so by induction on $m$. For $m=0$ the claim is vacuously true and we can assume that $m>0$ and that the claim is true for $k<m$. 
For every $s,t\in [0,1]$ with $s\le t$ we let $A^{s}_{t}$ be the stratified subspace of $\sabs{\Delta^{m}}$ given by the points $(t_{0},\ldots, t_{m})$ such that $s\le t_{m}\le t$.
Moreover, we set $A_{t}= A^{0}_{t}$ and $I_{t}=A^{t}_{t}$.
We define  $L\subset[0,1]$ to be the set of parameters for which the restriction $\alpha_{t}\colon A_{t}\to \sabs{\Delta^{n}}$ of $\alpha$ to $A_{t}$ has a unique lift $\tilde{\alpha_{t}}\colon A_{t}\to E$.
Notice that $A_{0}^{0}=A_{0}=I_{0}$ is isomorphic to $\sabs{\Delta^{m-1}}$ so $0$ belongs to $L$ by the base of the induction.
Assume that $t$ is in $L$, and let $\tilde{\alpha_{t}}\colon A_{t}\to E$ be the unique lift of $\alpha_{t}$. 
Then, for every point $e$ in the image of $I_{t}$ by $\tilde{\alpha_{t}}$ there exists an \'etale neighbourhood $U_{x}$ around $x$ such that $U_{x}\to pU_{x}$ is an isomorphism.
In particular, since $\sabs{\Delta^{n}}$ can be covered by convex balls and the image of $I_{t}$ via $\alpha_{t}$ is a compact subset of $\sabs{\Delta^{m}}$, we can reduce $U_{x}$ to a finite cover $U_{1},\ldots, U_{k}$ for some $k$, covering the image of $I_{t}$ via $\tilde{\alpha_{t}}$.
Each element $U_{i}$ of the cover allows us to extend uniquely the lift $\tilde{\alpha_{t}}$ further to a unique lift of $A_{t}\cup \tilde{\alpha_{t}}^{-1}U_{i}$. Moreover, since $I_{t}$ is connected, these unique lifts patch together to a unique lift of $A_{t}\cup \left(\bigcup_{i=1}^{k}\tilde{\alpha_{t}}^{-1}U_{i}\right)$.
To conclude, finiteness of the cover implies that we can find a small enough $\epsilon$ such that $[t,t+\epsilon)$ is contained in $L$, hence $L$ is open.
Now assume that $s$ belongs to $L$ for every $s<t$. Then, for every $0<s<t$ there exists a unique lift $\tilde{\alpha_{s}}\colon A_{s}\to E$ to $\alpha_{s}$. 
Let us fix $0<s<t$ and consider the lifting problem
\begin{equation*}
\lift{I_{s}}{E}{\sabs{\Delta^{n}}}{A^{s}_{t}}{e_{s}}{p}{\alpha^{s}_{t}}{0}{\tilde{\alpha^{s}_{t}}}
\end{equation*}
where $e_{s}\colon I_{s}\to E$ is the restriction of $\tilde{\alpha_{s}}$ to $I_{s}$.
Notice that both $I_{s}$ and $A^{s}_{t}$ are trivially stratified spaces and that the inclusion $I_{s}\to A^{s}_{t}$ is naturally isomorphic to the inclusion $\abs{\Delta^{m-1}}\times\{0\}\to \abs{\Delta^{m-1}}\times \II$.
Therefore, since the strata of $\sabs{\Delta^{n}}$ are locally 0-connected and locally 1-connected  and $p$ restricts to a covering space over each stratum, there exists a unique lift filling the diagram.
Pasting together the unique lifts $\tilde{\alpha_{s}}$ and $\tilde{\alpha^{s}_{t}}$ yields a unique lift $\tilde{\alpha_{t}}\colon A_{t}\to E$.
In particular, $L$ is closed. Hence $L=[0,1]$ and the claim is proven.
\end{proof}

\begin{lemma}
\label{lem:pullback-cover-left-cover}
Let $p\colon E\to X$ be an \'etale morphism of locally stratified spaces.
Assume that, for every morphism $f\colon Y\to X$, where $Y$ is a locally $0$-connected and locally $1$-connected chaotic locally stratified space, the base change:
\[p_{Y}\colon E_{Y}\to Y\]
is a topological cover. 
Then, $p$ is a left cover.
\end{lemma}
\begin{proof}
By change of base, it is enough to show the claim for $X=\dabs{\Delta^{n}}$. 
By Proposition \ref{prop:etale-strat-strat} and since strata of $\sabs{\Delta^{n}}$ are locally 0-connected and locally 1-connected, it is enough to show that every stratified cover $p\colon E\to \sabs{\Delta^{n}}$ defines a left cover of the associated locally stratified spaces. This is precisely Lemma \ref{lem:strat-cover-left-cover}.
\end{proof}

\subsection{}
\label{subsec:GZ-nghb}
Let $A$ be a simplicial set.
Following \cite[Chapter III, 1.7 and 1.8]{GZ67} every point $x$ in $\abs{A}$ has a contractible open neighbourhood $V_{x}=V_{x}^{1,1}$.
Let $x'$ be a point in $V_{x}$ and let $\sigma\colon \Delta^{n}\to A$ and $\sigma'\colon \Delta^{m}\to A$ be the unique non degenerate simplices corresponding to $x$ and $x'$ respectively.
Then $m\ge n$ and $x$ belongs to a face of $\sigma'$.
We call $V_{x}$ the \emph{Gabriel-Zisman neighbourhood} of $x$ in $\abs{A}$.

\begin{lemma}
\label{lem:GZ-lifting}
Let $p\colon E\to \dabs{A}$ be a left cover and let $x\in \dabs{A}$.
Let $i\colon V_{x}\subset \dabs{A}$ be the Gabriel-Zisman neighbourhood of $x$, equipped with the subspace structure. 
Then, a lifting problem of the form:
\begin{equation*}
\lift{\term}{E}{\dabs{A}}{V_{x}}{e}{p}{i}{x}{h}
\end{equation*}
has a unique solution.
\end{lemma}
\begin{proof}
For every point $x'\in V_{x}$ we can find a triangulation of $V_{x}$ by locally stratified simplices, covering $x'$. Since $E$ is a cover, we can lift uniquely each of those simplices, therefore defining a set theoretic map $h\colon V_{x}\to E$ fitting in the diagram above.
By uniqueness of lifts, the map does not depend on the choice of the triangulation and it is easily seen to be continuous and a morphism of locally stratified spaces.
\end{proof}

\begin{lemma}
\label{lem:left-cover-etale}
Let $A$ be a simplicial set and let $p\colon E\to \dabs{A}$ be a left cover. Then $p$ is an \'etale morphism.
\end{lemma}
\begin{proof}
Since $E$ is numerically generated, a subset $U$ of $E$ is open in $E$ if and only if the preimage of $U$ via every map
\[\sigma\colon \dabs{\Delta^{n}}\to E\]
is open in $\dabs{\Delta^{n}}$.
Since $\dabs{\Delta^{n}}$ is an $n$-dimensional CW-complex, there exists some $k\in \NN$ such that the composition $p\sigma\colon \dabs{\Delta^{n}}\to \dabs{A}$ factors through $\abs{\sk_{k}A}$.
Let $E_{k}=p^{-1}\abs{\sk_{k}A}$, then by the above argument, $U$ is open in $E$ if and only if the intersection $U\cap E_{k}$ is open in $E_{k}$ for every $[k]\in \DDelta$.

Let $e$ be a point of $E$ and let $x$ be the image of $e$ via $p$.
Let us consider the Gabriel-Zisman neighbourhood $V_{x}$ of $x$ in $\dabs{A}$.
Then, by Lemma \ref{lem:GZ-lifting} we can find a unique lift $h\colon V_{x}\to E$ and to prove that $E$ is \'etale over $\dabs{A}$, it is enough to prove that $U_{e}=hV_{x}$ is open in $E$.
By the first paragraph, it suffices to show that $U_{e}^{k}= U_{e}\cap E_{k}$ is open in $E_{k}$, for every $k$.
We prove the claim by induction.
Let $\sigma_{x}\colon \Delta^{m}\to A$ be the unique non degenerate simplex associated to $x$, then the claim is vacuously true for $k<m$, since $V_{x}$ does not intersect $\abs{\sk_{k}A}$ by construction (see \cite[1.8]{GZ67}).
For $k=m$ we have that the intersection $V_{x}^{m} =V_{x}\cap \abs{\sk_{m}A}$ is a convex chaotic open neighbourhood of $x$ in $\abs{\sk_{m}A}$ and the restriction of $E$ to $V_{x}^{m}$ is a chaotically stratified cover, hence $U^{m}_{e}$ is open in $E_{m}$.
Assuming that $U_{e}^{k}$ is open in $E_{k}$, it follows that $U_{e}^{k+1}$ is open in $E_{k+1}$ since the fibres of $p$ are discrete and $V_{x}^{k+1}$ is the union of $V_{x}^{k}$ with a convex subset of the interior of a $(k+1)$-cell.
\end{proof}

\begin{corollary}
\label{cor:characterisation-left-cover}
Let $A$ be a simplicial set.
A morphism $p\colon E\to \dabs{A}$ of locally stratified spaces is a left cover if and only if the following properties hold:
\begin{enumerate}
\item
The morphism $p$ is \'etale.
\item 
For every morphism $f\colon Y\to \dabs{A}$, where $Y$ is a locally $0$-connected and locally $1$-connected chaotic locally stratified space, the base change:
\[p_{Y}\colon E_{Y}\to Y\]
is a topological cover.
\end{enumerate}
\end{corollary}
\begin{proof}
The conditions are sufficient by Lemma \ref{lem:pullback-cover-left-cover}.
Conversely, if $p$ is a left cover, then $p$ is an \'etale map by Lemma \ref{lem:left-cover-etale} and it is a cover over $Y$ by Proposition \ref{prop:chaotic-left-cover-top-cover}.
\end{proof}

\begin{lemma}
\label{lem:GZ-nghb}
Let $A$ be a simplicial set and let $x$ be a point in  $\abs{A}$.
Then every point $x'$ in the Gabriel-Zisman neighbourhood $V_{x}$ of $x$ induces a morphism:
\[\alpha_{x,x'}\colon \lv x\to \lv x'\]
from the geometric last vertex of $x$ to the geometric last vertex of $x'$.
\end{lemma}
\begin{proof}
Since $V_{x}$ is contractible, up to homotopy, there exists a unique path $\alpha\colon \II\to \abs{A}$ from $x$ to $x'$.
Let $\sigma\colon \Delta^{n}\to A$ and $\sigma'\colon \Delta^{m}\to A$ be the unique non degenerate simplices associated to $x$ and $x'$, respectively.
Then $m\ge n$ and there exists a unique $n$-dimensional face of $\Delta^{m}$ such that $\alpha$ is the image of a continuous path in $\abs{\Delta^{m}}$ under the map $\abs{\sigma'}\colon \abs{\Delta^{m}}\to \abs{A}$.
To conclude, taking the last vertex of $\Delta^{n}$ in $\Delta^{m}$ defines a unique $1$-simplex  $\beta\colon \Delta^{1}\to\Delta^{m}$ from the vertex $n$ to the vertex $m$, and we take $\alpha_{x,y}\colon \lv x \to \lv y$ to be the morphism in $\tau_{1}A$ determined by the composition:
\begin{equation}
\begin{tikzcd}
\Delta^{1}\ar[r, "\beta"]&\Delta^{m}\ar[r, "\sigma'"]& A
\end{tikzcd}
\end{equation}
\end{proof}

\begin{construction}
\label{con:tentative-left-cover}
Let $A$ be a simplicial set and let $F\colon \tau^{1}A\to \Set$ be a functor.
We define a locally stratified space $p\colon C(F)\to \dabs{A}$ over $\dabs{A}$ as follows.
The underlying set of $C(F)$ is the coproduct:
\[\bigcoprod_{x\in \dabs{A}}F(\lv x).\]
where $\lv\colon \abs{A}\to A_{0}$ is the geometric last vertex map defined in \ref{subsec:last-vertex-geometric}.
Let $V_{x}$ be a Gabriel-Zisman neighbourhood of $\abs{A}$, for every point $(x,e)\in C(F)$ we take $V_{x,e}$ to be the set of points 
\[V_{x,e}=\left\{\left(x',F\left(\alpha_{x,x'}\right)e\right) : x'\in V_{x}\right\}\]
We take the topology on $C(F)$ generated by the set $\{V_{x,e}:x\in\abs{A}, e\in F(\lv x)\}$.
By construction, $p\colon C(F)\to \abs{A}$ is an \'etale map of topological spaces.
Therefore, by Lemma \ref{lem:etale-lifts} there exists a unique locally stratified structure on $C(F)$ making the map $p\colon C(F)\to\dabs{A}$ an \'etale morphism of locally stratified spaces.
\end{construction}

\subsection{}
Let $A$ be a simplicial set and let $x$ and $y$ be two points in $\abs{A}$. 
We define a relation $R$ on the underlying set of $\abs{A}$ by $x \mathrel{R}y$ if $y$ belongs to the Gabriel-Zisman neighbourhood $V_{x}$ of $x$ and the induced map $\alpha_{x,y}\colon \lv x \to \lv y$ is an isomorphism.
Let $\sim_{R}$ be the equivalence relation generated by $R$. If $x\sim_{R}y$ we say that $x$ is \emph{GZ-equivalent} to $y$.

\begin{definition}
Let $A$ be a simplicial set.
We say that $A$ \emph{satisfies the GZ-property} if two points $x$ and $y$ in $\dabs{A}$ are chaotically equivalent precisely when they are GZ-equivalent. 
\end{definition}

\begin{proposition}
Let $A$ be a simplicial set that satisfies the GZ-property.
Then, for every functor $F\colon \tau_{1}A\to \Set$, the morphism $p\colon C(F)\to \dabs{A}$ is a left cover.
\end{proposition}
\begin{proof}
By Corollary \ref{cor:characterisation-left-cover} it is enough to show that, for every morphism:
\[f\colon Y\to \dabs{A}\]
of locally stratified spaces, where $Y$ is a chaotic locally stratified space, the pullback $p_{Y}\colon C(F)_{Y}=C(F)\times_{\dabs{A}}Y\to Y$ of $p$ along $f$ is a topological cover.
The underlying set of $C(F)_{Y}$ is the coproduct:
\[\bigcoprod_{y\in Y}F(\lv(fy))\]
Let $y$ be a point of $Y$ and let us consider the Gabriel-Zisman neighbourhood $V_{x}$ of the image $x=fy$ of $y$ in $\abs{A}$.
To show that $C(F)_{Y}$ is a topological cover of $Y$, it is enough to show that the preimage of $U_{y}=f^{{-1}}V_{x}$ via $p_{Y}$ is a disjoint union of open subsets of $C(F)$ mapped homeomorphically onto $U_{y}$.
We claim that $p_{Y}^{-1}(U_{y})$ is equal to the product:
\[U_{y}\times F(\lv x).\]
Take a pair $(y', e)$ in $U_{y}\times F(\lv x)$, then since $fy'\in V_{x}$ there exists a morphism $\alpha_{x}\colon x\to f_{y'}$ in $\tau_{1}X$ but then, since $y$ is chaotically equivalent to $y'$, we have that $x$ is chaotically equivalent to $fy'$ which implies that $\alpha_{x}$ is an isomorphism.
In particular, we have that $(y',e)$ determines a unique point $(y', \alpha_{x,*}e)$ in $p_{Y}^{{-1}}U_{y}$.
To conclude the proof, it is enough to observe that $U_{y}\times\{e\}$ is open in $C(F)_{Y}$, which is true since it is the preimage of $U_{x,e}$ via $C(F)_{Y}\to C(F)$.
\end{proof}

\begin{remark}
Let $R$ be the simplicial set defined in Example \ref{ex:walking-retraction-simplicial-set}. 
We observe that $R$ does not satisfy the GZ-property.
Indeed, all points in $\dabs{R}$ are chaotically equivalent, while there are two GZ-equivalence classes, namely $\lv^{-1}0$ and $\lv^{-1}1$.
\end{remark}

\begin{example}
Let $R$ be as above and let
\[F\colon \Ret\to \Set\]
be the functor that maps $0$ to a singleton, that we label $\{a\}$ and $1$ to a set with two elements, labelled $\{b,b'\}$, with $Fi\colon \{a\}\to \{b,b'\}$ picking out the element $b$.
Then, the construction \ref{con:tentative-left-cover} yields a locally stratified space $C(F)$ over $\dabs{R}$ which is the disjoint union of a copy of $\dabs{R}$ with vertices $a$ and $b$, together with a chaotic half-open arc $U$ ending at $b'$ and projecting down to the complement of $0$ in  $\dabs{i}$.
In particular, $C(F)$ is not a left cover, since the local exit path $\dabs{r}$ from $1$ to $0$ has no lift starting at $b'$.
\end{example}

\begin{proposition}
\label{prop:tentative-left-cover-ff}
Let $A$ be a simplicial set that satisfies the GZ-property and let $F\colon \tau_{1}A\to \Set$ be a functor. Then, there exists a natural isomorphism:
\begin{equation}
\label{eq:C(E)}
F\cong \fib_{C(F)}
\end{equation}
where $\fib_{C(F)}$ denotes the fiber functor of $C(F)$.
\end{proposition}
\begin{proof}
By construction the value of $F$ at every $0$-simplex of $A$ is equal to the fibre of $C(F)$ at $a$.
Therefore, it is enough to show that for every morphism $\alpha\colon a\to a'$ in $A$ and an element $e\in F(a)$, the endpoint of the unique lift $\tilde{\alpha}\colon \dabs{\Delta^{1}}\to C(F)$ of $\dabs{\alpha}$ is $(a', F\alpha(e))$.
To do so, notice that the Gabriel-Zisman neighbourhood $V_{a}$ of $a$ in $\abs{A}$ intersects the image of $\alpha$ and for every point $y$ in the intersection, the path $\alpha_{a,y}\colon a\to \lv(y)$ is $\alpha$, by construction.
\end{proof}

\begin{theorem}
\label{thm:functor-left-cover-locstrat}
Let $A$ be a simplicial set and assume that $A$ satisfies the GZ-property.
Then, the functor $C$ fits in an equivalence of categories:
\begin{equation}
\Adjoint{\Set^{\tau^{1}A}}{\LCover{\dabs{A}}}{C}{\fib}
\end{equation}
\end{theorem}
\begin{proof}
Proposition \ref{prop:tentative-left-cover-ff} proves a half of the claim. To show the other half, let $p\colon E\to \dabs{A}$ be a left cover and let us consider the left cover $C(\fib E)$. 
By construction, the underlying sets of $E$ and $C(\fib E)$ coincide. The claim follows from Lemma \ref{lem:etale-lifts}, since both $E\to \dabs{A}$ and $C(\fib E)\to \dabs{A}$ are \'etale morphisms by Corollary \ref{cor:characterisation-left-cover}.
\end{proof}

\begin{corollary}
\label{cor:simplicial-left-cover-locstrat}
Let $A$ be a simplicial set and assume that $A$ satisfies the GZ-property. 
Then, the adjunction:
\begin{equation}
\Adjoint{\LCover{A}}{\LCover{\dabs{A}}}{L\dabs{\blank}}{\Sing_{A}}
\end{equation}
is an equivalence of categories.
Moreover, for every left cover $E$ of $\dabs{A}$, there exists a natural isomorphism:
\[L\dabs{E}\cong C(\fib E).\]
\end{corollary}
\begin{proof}
The claim follows from Theorem \ref{thm:functor-left-cover-locstrat} and Theorem \ref{thm:left-covers-fiber-functors}.
\end{proof}

\begin{corollary}
\label{cor:cover-R-trivial}
For every left cover $E$ over $\dabs{R}$ there exists a set $F$ such that $E$ is naturally isomorphic to $\dabs{R}\times F$ over $\dabs{R}$.
In particular, there is an equivalence of categories:
\[\LCover{\dabs{R}}\cong \Set.\]
\end{corollary}
\begin{proof}
Let $p\colon E\to \dabs{R}$ be a left cover and let us consider the pullback $p_{\sigma}\colon E_{\sigma}\to \dabs{\Delta^{2}}$ of $p$ along the realisation of the canonical map $\sigma\Delta^{2}\to R$.
Since $\dabs{\Delta^{n}}$ satisfies the GZ-property, and by Theorem \ref{thm:functor-left-cover-locstrat}, there exists a functor $F_{\sigma}\colon [2]\to\Set$ unique up to isomorphism such that $E_{\sigma}=C(F_{\sigma})$.
Moreover, since $E_{\sigma}$ is the pullback of a cover over $R$, the morphsm from $0$ to $2$ in $[2]$ induces the identity $F(0)=(2)$ on the fibres.
On the other hand, since the path from $0$ to $1$ is a chaotic path in $R$, the morphism from $0$ to $1$ in $[2]$ induces an isomorphism by Corollary \ref{cor:chaotic-left-property-cover}.
Therefore, the cover $E_{\sigma}$ is trivial over $\dabs{\Delta^{2}}$.
Therefore, $E\to \dabs{R}$ is a trivial cover with fibre $F(0)$.
\end{proof}

\begin{remark}
As the fundamental category of $R$ is the walking retraction category $\Ret$, Theorem \ref{thm:left-covers-fiber-functors} together with Corollary \ref{cor:cover-R-trivial} imply that the category $\LCover{R}$ of left covers over $R$ is not equivalent to the category $\LCover{\dabs{R}}$ of left covers over its realisation.
\end{remark}

\backmatter				%

\sloppy
\printbibliography[
heading=bibintoc,
title={Bibliography}
] 
\printindex
\end{document}